\documentclass[11pt, reqno]{amsart}
\usepackage{amssymb}
\usepackage{amsmath}
\allowdisplaybreaks[1]
\usepackage{graphicx,color,subcaption}
\usepackage{wrapfig,framed,array}

\usepackage{makecell}

\usepackage{microtype}
\usepackage[height=22.0cm, width=16cm, hmarginratio={1:1}]{geometry}
\usepackage{hyperref,bbm}

\theoremstyle{plain}
\newtheorem{theorem}{Theorem}
\newtheorem*{theorem*}{Theorem}
\newtheorem{prop}{Proposition}
\newtheorem{lemma}{Lemma}
\newtheorem{cor}{Corollary}
\theoremstyle{definition}
\newtheorem{defi}{Definition}
\newtheorem{remark}{Remark}
\newtheorem{example}{Example}

\newcommand{\beq}{\begin{equation}}
\newcommand{\eeq}{\end{equation}}
\newcommand{\nn}{\nonumber}

\newcommand{\p}{\partial}
\newcommand{\F}{\mathcal{F}}
\newcommand{\RR}{{\mathbb R}}
\newcommand{\QQ}{{\mathbb Q}}
\newcommand{\ZZ}{{\mathbb Z}}
\newcommand{\CC}{{\mathbb C}}

\DeclareMathOperator{\res}{res}
\newcommand{\bt}{{\bf t}}

\newcommand{\e}{\epsilon}

\newcommand{\M}{\mathcal{M}}

\newcommand{\g}{{\mathfrak g}}
\newcommand{\h}{{\mathfrak h}}
\newcommand{\ad}{{\rm ad}}

\def\tC{\widetilde C}

\def\={\; = \;}
\def\+{\, + \,}
\def\:={\; := \; }

\newcolumntype{?}{!{\vrule width 1pt}}

\newcommand{\indicationfootnote}{\thanks}

\makeatletter
\def\titlefootnote{\ifx\protect\@typeset@protect\expandafter\footnote\else\expandafter\@gobble\fi}
\makeatother

\begin{document}

\title[Geometry and arithmetic of integrable 
hierarchies]{Geometry and arithmetic of integrable 
hierarchies \\ of KdV type. I. Integrality}
\author[Dubrovin]{Boris Dubrovin$^{\dagger}$}
\indicationfootnote{$^{\dagger}$Deceased on March 19, 2019.}
\author[Yang]{Di Yang}
\author[Zagier]{Don Zagier}

\begin{abstract}
For each of the simple Lie algebras $\g=A_l$, $D_l$ or $E_6$, 
we show that the all-genera one-point FJRW invariants of $\g$-type,  
after multiplication by suitable products of Pochhammer symbols, 
are the coefficients of an algebraic generating function  
and hence are integral. Moreover, we find that 
 the all-genera invariants themselves coincide with 
the coefficients of the unique calibration of the Frobenius manifold of $\g$-type
evaluated at a special point. For the $A_4$ (5-spin) case we also 
find two other normalizations of the sequence that are again integral and of at most exponential growth, 
and hence conjecturally are the Taylor coefficients of some period functions.
\end{abstract}
\maketitle

\setcounter{tocdepth}{1}
\tableofcontents

\section{Introduction and description of the results for the 5-spin case} 
In this paper, we will study certain intersection numbers $\tau_{\g}(g)$ 
(the precise definition will be given in Section~\ref{section2point1}) 
on the moduli space of stable algebraic curves $\overline{\mathcal{M}}_{g,n}$~\cite{DeligneMumford}
associated to simple Lie algebras, the case of $A_{r-1}$-type simple 
Lie algebra being essentially one-point $r$-spin 
intersection numbers, introduced by Witten~\cite{Witten2}. 
(What we call $\tau_{A_{r-1}}(g)$ would be 
$\langle \tau_{s,m}\rangle$ in Witten's notation, where $2(r+1)g=r(s+1)+m+2$ with $s\geq0$, $0 \le m\le r-2$;  
our notation in later sections will be slightly different from Witten's.) 
In particular, we will give recursive, closed, and asymptotic formulas for these numbers. 
Using these formulas, we will show for $\g=A_l~(l\geq1)$, or $D_l~(l\geq4)$, or $E_6$ that 
by multiplying~$\tau_\g(g)$ by appropriate gamma 
factors (products of Pochhammer symbols) we obtain new numbers whose generating functions are {\it algebraic}. 
In particular, these renormalized numbers are integral and grow only exponentially in~$g$. 
Moreover, for the case of~$A_4$, we find that there are {\it different} normalizations of the~$\tau_{\g}(g)$, obtained by 
multiplying by other gamma factors, that are again integral and of exponential growth, so that each of 
the corresponding generating series is conjecturally a period function for some family of algebraic 
varieties (or equivalently, a solution of some Picard--Fuchs differential equation).  This latter point 
is also of interest from the point of view of the general arithmetic theory of differential equations 
(see for instance~\cite{Zagier3}, where the conjecture relating integrality and geometric origin 
 is discussed on pp.~728--729 and the $A_4$ example on pp.~768--769) and will be discussed from
this point of view in the later paper~\cite{YangZagier}.

For most of this introduction, we will assume that $\g=A_4$ and describe our results in detail only for that case, 
indicating briefly at the end of the introduction where the statements of the general results can be found in the paper.
For convenience, we write $\tau_g=\tau_{A_4}(g)$ for $g>0$, 
 and also set $\tau_0=1$ and $\tau_g=0$ for $g<0$. The first values are
\begin{align}
& \arraycolsep=5.1pt\def\arraystretch{1.5}
\begin{array}{?c?c|c|c|c|c|c|c|c|c|c|c?}
\Xhline{2\arrayrulewidth}  g & 0& 1 & 2 & 3 & 4 & 5 & 6 &7 & 8 & 9 & 10 \\
\Xhline{2\arrayrulewidth}  \tau_g & 1&  \frac16 & \frac{11}{2^43^25^2}  & 0 & \frac{341}{2^93^45^4} 
& \frac{161}{2^{10}3^55^5} & \frac{3397}{2^{13}3^65^6} & \frac{3421}{2^{13}3^85^7}  
& 0 & \frac{1670581}{2^{20}3^{10}5^{9}7^1} & \frac{26605753}{2^{23}3^{12}5^{12}}  \\
\Xhline{2\arrayrulewidth}
\end{array} \nn \\
& \qquad \qquad \qquad \qquad \qquad\qquad    \mbox{One-point 5-spin intersection numbers}  \nn
\end{align}

The following theorem, which will be proved in Section~\ref{sectionaseries}, 
gives three different integrality statements about the numbers~$\tau_g$. One of these 
statements (the integrality of the numbers $a_g$) will be generalized to all $A_l$, $D_l$ 
and~$E_6$ in Sections~\ref{sectionaseries}--\ref{sectione6} below. The two others will be given in this paper for the 
$A_4$~case only, with a discussion of the integrality properties of~$\tau_\g(g)$ for other 
simple Lie algebras postponed to the later paper~\cite{YangZagier} mentioned above.

\begin{theorem}\label{thmintegrality}
Each of the renormalized values
\begin{align*}
 a_{5n} 
 & =  -\tfrac1{5}  \bigl(\tfrac45\bigr)_{2n-1}  \tau_{5n} , \quad 
 & b_{5n} & = \bigl(\tfrac45\bigr)_n \bigl(\tfrac15\bigr)_n  \tau_{5n} , \quad  
 & c_{5n} & =  \bigl(\tfrac45\bigr)_n \bigl(\tfrac35\bigr)_n \tau_{5n} ,  \\
  a_{5n-1} 
 & =  -\tfrac1{5}  \bigl(\tfrac25\bigr)_{2n-1}   \tau_{5n-1}, \quad
  & b_{5n-1} & =  \bigl(\tfrac25\bigr)_n \bigl(\tfrac35\bigr)_n \tau_{5n-1} , \quad 
  & c_{5n-1}  & = \bigl(\tfrac25\bigr)_n \bigl(\tfrac45\bigr)_n \tau_{5n-1}  ,  \\
  a_{5n-3}  
  & = \tfrac1{5}  \bigl(\tfrac35\bigr)_{2n-2}  \tau_{5n-3} , \quad 
  & b_{5n-3} & = \bigl(\tfrac35\bigr)_{n-1} \bigl(\tfrac25\bigr)_n  \tau_{5n-3} , \quad 
  & c_{5n-3} & =  \bigl(\tfrac35\bigr)_{n-1} \bigl(\tfrac15\bigr)_n \tau_{5n-3} ,  \\
  a_{5n-4}  
 & = \tfrac1{5}  \bigl(\tfrac15\bigr)_{2n-2}  \tau_{5n-4}, \quad 
  & b_{5n-4} & =  \bigl(\tfrac15\bigr)_n \bigl(\tfrac45\bigr)_{n-1} \tau_{5n-4}, \quad 
  & c_{5n-4} & = \bigl(\tfrac15\bigr)_n \bigl(\tfrac25\bigr)_{n-1} \tau_{5n-4} 
\end{align*}
belongs to $\ZZ\bigl[\frac1{30}\bigr]$. Here $(x)_k:=x(x+1)\cdots(x+k-1)$ denotes the ascending Pochhammer symbol.
\end{theorem}

The point here is that the numbers $\tau_g$ decay like $1/\Gamma(\frac{2g}5)$, 
as we will see in a moment, and that therefore each of the numbers $a_g, b_g$ and $c_g$, 
as well as being integral (away from the primes 2, 3 and~5), is of only exponential growth in~$g$ and hence is 
expected to be the $g$th Taylor coefficient of some period function.  
For~$a_g$, Theorem~\ref{mainr5} below includes the stronger statement that the generating function $\sum a_g x^g$ is not only a
period function, but is in fact algebraic.
For a long time we were unable to identify the other two generating functions $\sum b_g x^g$ and $\sum c_g x^g$ as period functions,
but in the end it turned out that they were both not just period functions, but again algebraic, although much more complicated 
and of much higher degree than $\sum a_g x^g$.  This will be discussed in a later publication~\cite{YangZagier}.

\begin{remark}
The formulas for~$a_g$ in Theorem~\ref{thmintegrality} can be written more uniformly as  
\beq\label{abcuniform}
a_g \= \frac{(-1)^m}{5} \, (A)_m  \, \tau_g \,,
\eeq
where $m=\bigl[\frac{2g-1}5\bigr]$ and  $A=\bigl\{\frac{2g-1}5\bigr\}$ denote the integral and fractional parts of $\frac{2g-1}5$, 
respectively (recall that the fractional part of a real number~$x$ is defined as $\{x\}:=x-[x]$), and similarly the arguments
of the Pochhammer symbols in the formulas for $b_g$ and $c_g$ are always $A$~and~$B$ or $A$~and~$C$, respectively, 
where $B=\bigl\{\frac{3g+1}5\bigr\}$ and $C=\bigl\{\frac{4g+3}5\bigr\}$. Note also that $A = 0$ if $g\equiv 3 \, ({\rm mod}\,5)$, 
and that both $a_g$ (as defined for all~$g$ by~\eqref{quinticeqn} below) and $\tau_g$ vanish in this case.
\end{remark}

In the following theorem we collect many further properties of~$\tau_g$. 
All of the statements of this theorem will be generalized in the main body of the paper to the $A_l, D_l$, and~$E_6$ cases.
\begin{theorem} \label{mainr5}
The numbers~$\tau_g=\tau_{A_4}(g)$ have the following properties:

\smallskip

\noindent {\rm (i)} 
{\rm [recursion]} 
The numbers $\tau_g$ 
satisfy the recursion relation
\begin{align}
& 2^8 \, 3^4  \, 5^{17} \, 31 \, g \, (g-1) \, (g-2) \, (g-4) \, \tau_g \nn \\
& \quad -\,  5^{11} \, 
\bigl(2^8 \, 3^4 \, g^4-2^{13} \, 3^4 \, g^3+2^4 \, 3^2 \, 54331\, g^2-2^4 \, 3^2 \, 43^1 \, 6329\, g
+5^1 \, 7^1 \, 2013229\bigr) \, \tau_{g-5} \nn\\
& \quad +\, 2^2 \, 5^6 \, \bigl(2^2 \, 3^2 \, 5\, g^2- 2^2 \, 3^3 \, 5^1 \, 7\, g+19739\bigr) \, \tau_{g-10} \,-\, \tau_{g-15}  \= 0 \,, \quad g\in \ZZ \label{recursion5spin}
\end{align}
with the initial conditions given by the above table.

\smallskip

\noindent {\rm (ii)} 
{\rm [dual topological ODE]}
The generating series  
\beq\label{generatingfivespinphi}
\varphi(X) \:= \sum_{g\geq 0} \, \tau_g \, X^{\frac{g}5}
\eeq
satisfies the fourth-order linear differential equation $Q\varphi=0$, where  
\begin{align}
& Q \= 2^8 \, 3^4 \, 5^{15} \, X^3 \, \Bigl(X-5^6 \, 31\Bigr) \, \frac{d^4}{dX^4} 
\+ 2^8 \, 3^4 \, 5^{14} \, X^2 \, \Bigl(2^1 \, 3^2 \, X-5^6 \, 23^1 \, 31\Bigr) \, \frac{d^3}{dX^3}  \nn\\
& \quad \quad  \quad
\,-\, 2^4 \, 3^2 \, 5^9 \, X \, \Bigl(X^2- 5^4 \, 6091\, X+ 2^6 \, 3^3 \, 5^{10} \, 7^1 \, 31\Bigr) \, \frac{d^2}{dX^2} \nn\\
& \quad \quad  \quad
\,-\, 2^5 \, 3^2 \, 5^8 \, \Bigl(2 \, X^2-5^4 \, 3209\, X+ 2^5 \, 3^3 \, 5^{10} \, 31\Bigr) \, \frac{d}{dX} \nn\\
& \quad \quad \quad
\+ \Bigl(X^2+2^2 \, 5^6 \, 61\,X+5^{13} \, 7^1 \, 23^1 \, 31\Bigr) \,. \label{odethma4}
\end{align}

\smallskip

\noindent {\rm (iii)} {\rm [algebraicity]}
The generating function of the numbers~$a_g$ defined in Theorem~\ref{thmintegrality} is algebraic.
More precisely, we have
\beq\label{quinticeqn}
y \:= \sum_{g\geq 0} a_g \, z^g \quad \Rightarrow \quad y^5 \,-\, \frac{z}6 \, y^3 \+ \frac{z^2}{400} \, y  \= 1 \,.
\eeq

\smallskip

\noindent {\rm (iv)} {\rm [closed formula]} Denote $m=[(2g-1)/5]$ as above, and 
define $c_{p,j}\in \QQ$ $(0 \leq p\leq j)$ by 
\beq
c_{p,j} \:= {\rm coefficient~of}~x^{j}~{\rm in}~\frac1{p!} \, \biggl(\frac{(1+x)^{6}-1-6 \, x}{6 \, x}\biggr)^{p} \,.
\eeq
Then for all $g\geq 0$ with $g\not\equiv 3 \,({\rm mod}\,5)$, we have
\beq\label{closedformulaacase}
\tau_g 
\= \Gamma\bigl(\bigl\{\tfrac{2g-1}5\bigr\}\bigr) \, \frac{(-1)^{g+m-1}}{5^g} \, \sum_{p=0}^{2g} 
 \frac{c_{p,2g}}{\Gamma\bigl(\frac{2g-1}5-p+1\bigr)} \,.
\eeq 

\smallskip

\noindent {\rm (v)} {\rm [product formula]} 
Let \beq w(u) \= 1 \+ \sum_{n\ge0} \, C_n \, u^{n+1} \= 1 \+ u \,-\, \frac23 \, u^2 \+ \frac{11}{18}u^3 - \dots \eeq be the 
unique power-series solution in $1 + u + u^2 \QQ[[u]]$ to the sextic equation
\beq
\frac{w^6}{30} \,-\, \frac{w}{5} \+ \frac16 \= \frac{u^2}2\,,
\eeq
and let $f(T):=\sum_{k\geq 0} (2k-1)!! \, C_{2k} \, (-T)^k$. 
Then
\beq
f(T) \, f(-T) \= \sum_{g\geq0} \, \Bigl(1+\frac{2g-1}{5}\Bigr)_{2g} \,  (-1)^{g-1} \, 5^{3g}  \, a_g \, T^{2g}
\eeq
with~$a_g$ as 
in part~{\rm (iii)}.

\smallskip

\noindent {\rm (vi)} {\rm [terminating hypergeometric sum]} 
For all $g\in \ZZ$ with $g\not\equiv 3 \,({\rm mod}\,5)$, we have
\beq\label{BH-hyper}
\tau_g \=  \frac{6^{-g}}{\Gamma\bigl(1- \bigl\{\frac{2g-1}5\bigr\}\bigr)} \, 
\sum_{0\le s \le g/2} \, (-1)^s \, \Bigl(\frac{3}{10}\Bigr)^{2s} \, \frac{\Gamma\bigl(\frac{1+3g-5s}{5}\bigr)}{  s!\,(g-2s)! }  \,.
\eeq

\smallskip

\noindent {\rm (vii)} {\rm [asymptotics]} 
For $g\not\equiv 3 \,({\rm mod}\,5)$ and as $g\to\infty$, $\tau_g$ is given asymptotically by  
\beq
\tau_g \; \sim \; \frac{5}{\sqrt{6}^{\frac35}\sqrt{\pi g^3}} \; 
\frac{\sin\bigl(\bigl\{\frac{2g-1}5\bigr\} \pi \bigr)}{\bigl(\bigl\{\frac{2g-1}5\bigr\}\bigr)_m} \, 
\, \Bigl( 6^{\frac25} \, 20 \, \sin^2\bigl(\frac{\pi}5\bigr) \Bigr)^{-g}\,,
\eeq
where $m=[(2g-1)/5]$ is again as above.
\end{theorem}
\noindent The proof of this theorem is given in Sections~\ref{section2point1} and~\ref{sectionaseries}. 

\begin{remark}
Of course, parts (i) and (ii) are equivalent, by the usual bijection between power series satisfying 
a linear differential equation with polynomial coefficients and sequences of numbers 
satisfying a recursion with polynomial coefficients.  More explicitly,
since the recursion~\eqref{recursion5spin} involves only $\tau_g$, $\tau_{g-5}$, $\tau_{g-10}$ and~$\tau_{g-15}$, 
it is equivalent to five separate recursions for~$\tau_g$ with 
$g \equiv s~({\rm mod}\,5)$, $0\le s\le4$, the values for $s=3$ being uninteresting because $\tau_g=0$.
Similarly, by the Frobenius method, the differential equation in part~(ii) has four fundamental 
solutions 
in $x^\lambda \QQ[[x]]$ with $\lambda = 0,\frac15,\frac25,\frac45$, and the generating function~\eqref{generatingfivespinphi}
 is simply a certain linear combination of these with rational coefficients.
\end{remark}

\begin{remark}\label{remarkintag}
The integrality (away from 2, 3, and 5) of~$a_g$ in Theorem~\ref{thmintegrality} follows 
immediately from part (iii) of Theorem~\ref{mainr5}, 
since the Taylor coefficients of any algebraic function are integral after some fixed rescaling. 
In fact, $\sum_{g\neq 3} (g-3)^{-1} a_g x^g$ is also algebraic, as we will see later at the end of Section~\ref{sectionaseries}, and 
correspondingly we have $(g-3)^{-1} a_g\in \ZZ[\frac1{30}]$ for $g\geq3$.
The integrality of~$b_g$ and~$c_g$ is harder and will follow from formula~\eqref{BH-hyper}, 
as a consequence of the stronger statement, also proved at the end of  Section~\ref{sectionaseries},
that each summand on the right-hand side of~\eqref{BH-hyper} becomes $p$-adically integral for every $p>5$
after multiplication with the two Pochhammer symbols given in Theorem~\ref{thmintegrality}.
\end{remark}

\begin{remark}
We say something here about the origin of the
problem and about some of the various approaches that can be used to solve it.
The way that we are approaching intersection numbers  
is to use integrable systems. In 1990, Witten~\cite{Witten1} proposed his famous conjecture 
that the partition function of $\psi$-class intersection numbers 
 is a tau-function for the Korteweg-de Vries (KdV) integrable hierarchy. This conjecture was later proved 
 by Kontsevich~\cite{Kontsevich}. The $r$-spin Witten conjecture~\cite{Witten2}, stating that the partition function of  
 the $r$-spin intersection numbers gives a tau-function for the Gelfand--Dickey integrable hierarchy~\cite{Dickey}, 
 was proved by Faber--Shadrin--Zvonkine~\cite{FSZ}. (The $r=2$ case corresponds to 
 Witten's conjecture in 1990.) More generally, the ADE Witten conjecture~\cite{Witten2}, given its precise form 
 by Fan--Jarvis--Ruan~\cite{FJR} (cf.~Givental--Milanov~\cite{GM}), states that the 
 partition function of the FJRW invariants associated to a certain simple singularity~\cite{A72} gives a tau-function for the 
 Drinfeld--Sokolov (DS) integrable hierarchy~\cite{BDY2,CaWu,DSKV,DS,DLZ}. 
 The ADE Witten conjecture was proved by Fan--Jarvis--Ruan~\cite{FJR}  (cf.~\cite{FGM,GM,LRZ})  
 with the $D_4$ case confirmed in~\cite{FFJMR}. 
 For all these cases, the main mathematical object of the study is the 
 so-called {\it topological solution} $u^{{\rm top}}$~\cite{Du-8,DZ-norm} to the corresponding 
 integrable system together with its tau-function. Perhaps, the simplest 
 way to describe this particular solution is to use its initial value (observe that each member of the DS hierarchy 
 is an evolutionary PDE), which in terms of the normal coordinates~\cite{DZ-norm} $r_1,\dots,r_l$ (where $l$ is the rank) reads:
 \beq
 r^{\rm top}_\alpha|_{{\rm higher~times}=0} \= \eta_{\alpha1} \, t^{1,0}\,,
 \eeq
where $\alpha=1,\dots,l$ and $\eta_{\alpha1}$ are constants.
One could then apply methods in integrable systems to 
compute~$u^{{\rm top}}$ and more importantly its {\it tau-function}. These methods include  
the wave-function approach \cite{BBT,BaY, BYZ, BY, Buryak1, CaWu, Dickey, Du-kdv, DYZ, KS, Sato,SW}, 
the $\Psi {\rm DOs}$~\cite{CMZ,Dickey,KhZ,LVX}, 
the Sato Grassmannian approach \cite{BY,DJKM,Dickey,Sato,SW}, 
the Dubrovin--Zhang approach \cite{BPS,Du-3,Du-8,DLYZ,DZ1,DZ-norm}, 
the Givental quantization~\cite{BPS,Givental-1,Givental-2,Givental-3,GM}, 
the topological recursion of Chekhov--Eynard--Orantin type (or of Bouchard--Eynard type)
 (cf.~\cite{BouchardE, DBNOPS, Zhou0, Zhou12} and the references therein),
the matrix-resolvent method~\cite{BDY0,BDY1,BDY2,DYZ0,DYZ} (see also~\cite{Buryak2,LX1,LX2,OK}), etc. 
One of the keys for several of these methods is the {\it Lax pair} (see~\cite{BBT,Dickey,DS,Du-kdv}). 
In the second paper~\cite{DYZ2} of the series, the mathematical object will be 
a different solution to the DS hierarchy, whose tau-function has a different topological meaning; 
more general cases are considered in joint papers currently in preparation with Daniele Valeri.
Taking the dispersionless limit of the DS hierarchy, one can also obtain an interesting geometric 
structure, suitable for studying 
{\it arbitrary} solutions, i.e., the Frobenius structure~\cite{Du-1,Du-2,Du-3,Du-8,DZ-norm}, which plays one of the central 
roles in understanding the above-mentioned methods. Besides the methods from integrable systems,
there are other important methods for computing the intersection numbers from other theories, 
including the theory of matrix models \cite{AvM,AM,BH1,BH2,BH-book,Deift,HZ,Kontsevich,Mehta}, 
vertex algebras \cite{BM,DVV1,FGM,KW,LYZ,Zhou1,Zhou3}, emergent geometry~\cite{Zhou12,Zhou2,Zhou4}, etc.
\end{remark}

Our next result (see Theorem~\ref{thmduality0infty} and its corollary) gives an explicit relationship between the all-genera 
FJRW intersection numbers of $\g$-type (with $g=A_l$, $D_l$ or $E_6$) and 
genus zero, which again will be stated in the introduction only for the $5$-spin case. Let~$B$ be the 
 Frobenius manifold associated to the $A_4$ Coxeter group~\cite{Du-3,Du-5,Sa3, Zuber}, 
and let $(\theta_{\alpha,m})_{\alpha=1,\dots,4,m\ge0}$ be 
 the unique {\it calibration}~\cite{DLYZ,DZ-norm} on~$B$.
 The reader who is not familiar with the theory of Frobenius manifolds could simply identify~$\theta_{\alpha,m}$ 
 with the following formal power series of infinitely many variables:
 \beq
 \theta_{\alpha,m} \= \frac{\p^2\mathcal{F}_0(\bt)}{\p t^{\alpha,m} \p t^{1,0}} \,,
 \eeq
 where $\bt=(t^{\alpha,q})_{\alpha=1,\dots,4, \, q\geq 0}$ and the definition of~$\mathcal{F}_0$ can be found in~\eqref{FgCohFT}. 
 Denote $v_\alpha=\theta_{\alpha,0}$, $\alpha=1,\dots,4$, and denote $v=(v_1,\dots,v_4)$. 
 Then from the theory of Frobenius manifolds (see Appendix~\ref{appfm}) we know that 
 all the $\theta_{\alpha,m}$ are {\it polynomials in}~$v$, with 
 the first values being 
\begin{align}
&\theta_{\alpha,0} \= v_\alpha\,,\nn\\
&\theta_{1,1} \= v_1 v_4 +v_2v_3 \,, \nn\\
&\theta_{2,1} \= \frac{v_3^2}{2}+\frac{1}{10} v_1^2 v_3+v_4 v_2+\frac{1}{10} v_2^2 v_1 \,, \nn\\
&\theta_{3,1} \= \frac{v_2^3}{15}+\frac{1}{75} v_1^3 v_2+\frac{1}{5} v_3 v_1 v_2+v_4 v_3 \,, \nn\\
&\theta_{4,1} \= \frac{v_1^5}{2500}+\frac{1}{50} v_2^2 v_1^2+\frac{1}{10} v_3^2 v_1+\frac{v_4^2}{2}+\frac{1}{10} v_3 v_2^2 \,, \nn\\
& \theta_{1,2} \= \frac{v_1^6}{3750}+\frac{1}{50} v_2^2 v_1^3+\frac{1}{10} v_3^2 v_1^2
+\frac{1}{2} v_4^2 v_1+\frac{1}{5} v_3 v_2^2 v_3+\frac{v_2^4}{30}+\frac{v_3^3}{6}+v_4 v_3 v_2 \,.\nn
\end{align}

\begin{theorem} \label{thmduality0inftya4}
Define 
\[v^*_1=\frac16\,, \quad v^*_2=0\,,  \quad v^*_3=\frac{11}{3600}\,, \quad v^*_4=0\,.\] 
Then for all $g\geq 1$ with $g\not\equiv3\,({\rm mod}\,5)$, we have
\beq\label{thetataudualitya45spin}
\tau_{g} \= \theta_{\alpha,m}(v^*) \,,
\eeq
where $\alpha\in\{1,2,3,4\}$ and $m\geq0$ are such that $2g-1=\alpha+5m$.
\end{theorem}

\medskip

\noindent {\bf Organization of the paper.}
In Section~\ref{section2point1}, we recall the definition of~$\tau_{\g}(g)$ for all~$\g$. 
In Section~\ref{dualtopoodesection} we generalize parts (i), (ii) of Theorem~\ref{mainr5} from~$A_4$ to an arbitrary~$\g$.
In Section~\ref{techsection}, we provide several technical 
preparations for the subsequent sections. 
In Section~\ref{sectionaseries} we prove the generalization of the rest of Theorem~\ref{mainr5} for the
$A$ series, as well as proving Theorem~\ref{thmintegrality}.
The analogues of Theorem~\ref{mainr5} 
for the~$D_l$ and~$E_6$ cases are given in Sections~\ref{sectiondseries} and~\ref{sectione6},
respectively, while Theorem~\ref{thmduality0inftya4} is generalized to the $A_l$, $D_l$ and~$E_6$ cases in Section~\ref{sectionsurprising}.
The necessary material on the wave-function-pair approach for computing 
residues of pseudodifferential operators and 
a brief review of the theory of Frobenius manifolds are provided 
in Appendices~\ref{wavebi} and~\ref{appfm}.

\medskip

\noindent {\bf Note on authorship.}  In answer to questions that we have been asked, we would like to make the following statement
to clarify the role of the first author, 
since for various reasons (several other joint projects, the time spent to settle the $E_6$ case, 
and our efforts to resolve the questions about period functions posed in Section~1) 
this paper was only finalized nearly two years after his death.  
The project started in Trieste when B.D.~and~D.Y. were members of SISSA and D.Z. was a senior visiting scientific 
member of~ICTP. It originated from a question raised by~D.Z. during 
a talk given by~B.D. at ICTP on topological ODEs~\cite{BDY1} in November~2015, concerning
the integrality of the invariants in the 5-spin case after multiplication by suitable Pochhammer
symbols. This question was answered affirmatively within the next few months in two completely
different ways, by D.Y.~and~D.Z. by proving the integrality (away from 30) of the numbers 
denoted~$c_g$ in Section~1, and then by B.D.~and~D.Y. by proving the algebraicity 
of the generating function of the differently normalized numbers~$a_g$.  Both results were 
reported by~D.Z. in~\cite{Zagier3}, but it was only the second discovery, on the
algebraicity of the generating series~$\sum a_g x^g$, that generalized to all ADE cases 
(except possibly $E_7$ and~$E_8)$ and became the central result of the current paper.  
During 2015--2018 all three authors communicated frequently by email and also met in person 
several times in Trieste to work on the project. 
The general algebraicity conjecture for all $ADE$ cases was made in 2017, and the precise 
form of the current Theorem~\ref{thmduality0infty} for all $ADE$ cases was stated as a
conjecture in an early draft of the paper (March 2018) signed by all three of us. 
These conjectures were then proved for the $A$ cases using the method sketched 
in Remark~\ref{BrezinHikami} in Section~5, and also for $D_4, D_5, D_6$ using topological ODEs.
Then in May, B.D. suggested the idea, motivated by his explicit computation for the $D_4$~case,
of proving the conjectures 
by using wave functions rather than topological ODEs, and in the subsequent two months 
D.Y.~and~D.Z. worked out the details of this proof for all A- and D-cases 
and sent it by email 
to Boris. In July, D.Y. again visited B.D. in Trieste and together they checked 
all details of this complete proof. A draft of the paper was written out at the end of July~2018 
and contained all parts, including most of the appendices, of the current paper except for
the details of the proof of Theorem~\ref{odethmg}, 
the precise formula and proof of the asymptotics (parts~(vii)) in 
Theorems~\ref{algebraicthma},~\ref{algebraicthmd}, and 
the proof of algebraicity for the~$E_6$ case, which was found (now again using topological ODEs) by 
D.Y.~and~D.Z. only after Boris passed away. The  
second and third authors also wrote a new introduction 
and edited and polished many parts of the remaining text. Thus all of the present manuscript 
except for the proofs of Theorem~4 and of the $E_6$ case and the asymptotics statements in Theorems~5 and~6
was joint work with~B.D. and was seen and approved by him.
The arXiv submission of the paper~\cite{DYZ}, made while Boris was still alive, cited 
both the present paper and its planned second part, 
but only a small part of that paper was finished and 
it is now planned as a publication of only 
the two other authors~\cite{DYZ2}, with a statement there specifying his role. 

\medskip

\noindent {\bf Acknowledgements.} 
We would like to thank Giordano Cotti, Yuri Manin, Maxim Smirnov, 
Jian Zhou and Wadim Zudilin for their interest and helpful discussions. 
D.Y. is grateful to Youjin Zhang and Marco Bertola for their advice. 
We also thank Elba Garcia-Failde for a very careful reading of this paper and 
for very helpful comments.
Part of the work of D.Y. was done during his postdoctoral studies and visits at SISSA and MPIM; 
he thanks both institutes for their excellent working conditions, warm hospitality and financial support. 
The work was partially supported by the 
National Key Research and Development Project ``Analysis and Geometry on Bundles" 
 2020YFA0713100, and by NSFC No.~12061131014.

\section{Cohomological field theories and FJRW invariants of $\g$-type}\label{section2point1}
Let $\g$ be a simply-laced simple Lie algebra. 
In this section, we review the definitions of the enumerative invariants 
$\tau_\g(g)$, 
called the one-point FJRW (Fan--Jarvis--Ruan, Witten) invariants, that are studied in this paper.

We start with the definitions for the $A$ case. 
For $g,n\geq 0$, let $1\leq \alpha_1, \dots, \alpha_n \leq r$ be integers satisfying the divisibility condition
$r \, \big| \, \sum_{i=1}^n \alpha_i-n-(2g-2)$. 
For an algebraic curve~$C$ of genus~$g$ with~$n$ distinct marked points $x_1$, \dots, $x_n$, 
there exists a line bundle ${\mathcal T}$ on~$C$ such that
\begin{equation}\label{root}
{\mathcal T}^{\otimes r} \= K_C\otimes{\mathcal O}\bigl((1-\alpha_1) \, x_1\bigr)\otimes\dots \otimes{\mathcal O}\bigl((1-\alpha_n) \, x_n\bigr) \,,
\end{equation}
where $K_C$ is the canonical class of~$C$.
For $C$ smooth there are $r^{2g}$ such line bundles. 
A choice of such an ``$r$-th root" of the bundle~\eqref{root} is called an $r$-spin structure, 
and it defines a point in a covering of~$\overline{\mathcal M}_{g,n}$. 
After a suitable compactification, this covering is denoted by
$$p: \overline{\mathcal M}_{g,n}^{1/r}(\alpha_1, \dots, \alpha_n) 
\to \overline{\mathcal M}_{g,n}\,.$$
In genus zero, 
for a point $\bigl( C, x_1, \dots, x_n, {\mathcal T}\bigr)$ in the covering space 
$\overline{\mathcal M}_{0,n}^{1/r}(\alpha_1, \dots, \alpha_n)$, denote $V=H^1(C, {\mathcal T})$. 
This gives a vector bundle ${\mathcal V}\to \overline{\mathcal M}_{0,n}^{1/r}(\alpha_1, \dots, \alpha_n)$ 
as the space $V$ has constant dimension thanks to the fact that $H^0(C, {\mathcal T})$ vanishes. Put  
\beq\label{Wittenclassdef}
c_0(\alpha_1, \dots, \alpha_n) \:=  p_* \,  e\bigl({\mathcal V}^\vee\bigr) \, \in \, H^{2(s-1)}\bigl( \overline{\mathcal M}_{0,n}\bigr)  \, ,
\eeq
where $e({\mathcal V}^\vee)$ is the Euler class of the dual vector 
bundle ${\mathcal V}^\vee$, and $s := \frac{\sum_{i=1}^n \alpha_i-n+2}{r}$.  
The cohomology class $c_0(\alpha_1, \dots, \alpha_n)$ is called the {\it Witten class}. 
For higher genus, $H^0(C,\mathcal{T})$ 
is only generically zero and the vector bundle can only be defined on a generic stratum. The Witten class $c_g(\alpha_1, \dots, \alpha_n)$ 
could still be defined as a particular cohomology class in $H^{2(s+g-1)}(\overline{\mathcal{M}}_{g,n})$ 
with $s=\frac{\sum_{i=1}^n \alpha_i-n-(2g-2)}{r}$,
but the construction is more involved (cf.~\cite{CLL, C1, C2, FJR0, FJR, JKV, Mochizuki, Polishchuk, PV, Witten2}).  
The genus~$g$ $r$-spin intersection numbers are defined as the following integrals:
\begin{equation}\label{rspinn}
 \int_{\overline{\mathcal M}_{g,n}} 
 c_g(\alpha_1,\dots, \alpha_n) \, \psi_1^{q_1} \cdots \psi_n^{q_n} \; =: \;  
 \langle \tau_{\alpha_1, q_1} \cdots \tau_{\alpha_n, q_n} \rangle_{g}\,, 
 \qquad q_1,\dots,q_n \geq 0 \,,
\end{equation}
where $\psi_i$~$(1\leq i\leq n)$ denotes the first Chern class of the~$i$th tautological line bundle over 
$\overline{\mathcal M}_{g,n}$.
These integrals vanish unless the degree and the dimension match:
\[
s \+ g \,-\, 1 \+ q_1 \+ \cdots \+ q_n \= 3g \,-\, 3 \+ n\,.
\]
The so-called Vanishing Axiom, conjectured in~\cite{JKV} and proved in~\cite{Polishchuk,PV}, 
says that the Witten class $c_g(\alpha_1, \dots, \alpha_n)$ vanishes
if any of $\alpha_1, \dots, \alpha_n$ reaches~$r$. Therefore we assume that 
$\alpha_1, \dots, \alpha_n$ are in $\{1,\dots, r-1\}$.
The numbers~$\tau_{A_{r-1}}(g)$ that we are looking at are the $r$-spin 
intersection numbers with $n=1$ (one-point). More precisely, 
they are defined by 
\beq\label{acasedefinition}
\tau_{A_{r-1}}(g) \:= \langle\tau_{\alpha,q}\rangle_{g} \,,
\eeq
where $\alpha\in\{1,\dots,r-1\}$ and $q\geq0$ are uniquely determined by 
\beq\label{rgaq}
2 \, (r+1) \, g \,-\, 1 \= \alpha \+ r \, (q+1) \,.
\eeq

The Witten class gives a particular {\it cohomological field theory (CohFT)}~\cite{KontsevichManin, Manin} 
of rank~$(r-1)$. Let us recall the definition of a general CohFT.
A rank~$l$ CohFT is a quadruple 
$\bigl(V^l, \langle,\rangle, \mathbbm{1}, \Omega_{g,n}\bigr)$, 
where $V$ is a $\CC$-vector space, $\langle,\rangle$ is a symmetric non-degenerate bilinear form,  
$\mathbbm{1}$ 
is a particular element in~$V$, called the unity, and $\{\Omega_{g,n}\}_{2g-2+n>0}$ is a collection 
of linear maps from $V^{\otimes n}$ to $H^{\rm even}\bigl(\overline{\mathcal M}_{g,n};\CC\bigr)$,
satisfying the following axioms:

\noindent C1. (total symmetry) 
Each $\Omega_{g,n}$ is $S_n$-invariant, where the action of $S_n$ permutes both the marked points of $\overline{\mathcal{M}}_{g,n}$ and 
the tensor products $V^{\otimes n}$. 

\noindent C2. (splitting) Let $e_\alpha$ be a basis of~$V$. Denote 
$\eta_{\alpha\beta} := \langle e_\alpha , e_\beta\rangle$, $\eta=(\eta_{\alpha\beta})$,
$\bigl(\eta^{\alpha\beta}\bigr):=\eta^{-1}$, and denote by~$q_{g, n}$ and~$s_{g_1,n_1,g_2,n_2}$ the gluing maps
\begin{align}
& q_{g,n}: \, \overline{\mathcal{M}}_{g-1,n+2}\rightarrow  \overline{\mathcal{M}}_{g,n} \, , \\
& s_{g_1,n_1,g_2,n_2}: \, \overline{\mathcal{M}}_{g_1,n_1+1} \times \overline{\mathcal{M}}_{g_2,n_2+1} \rightarrow 
\overline{\mathcal{M}}_{g=g_1+g_2,n=n_1+n_2} \, . 
\end{align}
Then it is required that  $\forall~x_1,\dots,x_n\in V$,
\begin{align}
& q_{g,n}^*\Omega_{g,n}(x_1,\dots,x_n) \= \Omega_{g-1,n+2}(x_1,\dots,x_n,e_\alpha,e_\beta) \, \eta^{\alpha\beta} \, ,\\
& s_{g_1,n_1,g_2,n_2}^*\Omega_{g,n}(x_1,\dots,x_n) \= \Omega_{g_1,n_1+1}(x_1,\dots,x_{n_1},e_\alpha) \, \eta^{\alpha\beta} \, \Omega_{g_2,n_2+1}(e_\beta,x_{n_1+1},\dots,x_n) \,.
\end{align}
Here and below, free Greek indices take the integer values from~1 to~$l$, and 
 the Einstein summation convention 
is used for repeated Greek indices with one up and one down, and the tensors $\eta^{\alpha\beta}$ and $\eta_{\alpha\beta}$ will be used to raise 
and lower the Greek indices.

\noindent C3. (unity) Let $p: \overline{\mathcal{M}}_{g,n+1}\rightarrow \overline{\mathcal{M}}_{g,n}$ be the forgetful map. Then $\forall~x_1,\dots,x_n\in V$,
\begin{align}
& \Omega_{g, n+1}(x_1,\dots,x_n, \mathbbm{1}) \= p^* \Omega_{g,n}(x_1,\dots,x_n) \,, \\
& \Omega_{0,3}(x_1,x_2, \mathbbm{1}) \= \eta(x_1,x_2) \,. 
\end{align}

Take $e_\alpha$ a basis of~$V$ with $e_1=\mathbbm{1}$.
It is natural to view~$V$ as the complex coordinate space $\{(v^1,\dots,v^l)\mid v^\alpha\in\CC\}$ and therefore as 
the complex manifold~$\CC^l$. Define a power series~$F=F(v)\in\CC[[v^1,\dots,v^l]]$ by
\beq\label{cohof}
F\:=\sum_{n\geq 3} \int_{\overline{\M}_{0,n}} 
\Omega_{0,n}(e_{\alpha_1},\dots,e_{\alpha_n}) \, \frac{v^{\alpha_1} \cdots v^{\alpha_n}}{n!} \, .
\eeq
We call~$F$ the {\it genus-zero primary potential} (cf.~\cite{Du-3,DZ-norm,KontsevichManin,Manin,PPZ1,PPZ2,Teleman}).
Denote by $B\subset \CC^l$ the domain 
of convergence for~$F$ around $v=0$. 
Throughout the paper we assume that~$B$ contains an open ball centered at $v=0$. 

\begin{remark}\label{cohomonovikov}
For a projective variety~$X$, the Gromov--Witten classes associated to~$X$ give rise to a CohFT.
For this case, it is sometimes helpful or even necessary to replace the ring
 $H^{\rm even}\bigl(\overline{\mathcal M}_{g,n};\CC\bigr)$    
with $\mathcal{N} \otimes H^{\rm even}\bigl(\overline{\mathcal M}_{g,n};\CC\bigr)$, where 
$\mathcal{N}$ is the Novikov ring~\cite{KontsevichManin,Manin}. We did not 
write the axioms involving the Novikov ring because for the specific CohFTs that are   
considered in this paper the power series~$F$ is actually a polynomial and so $B=\CC^l$.
Nevertheless, it will be interesting to generalize the results of this paper to 
the situation when the Novikov ring is introduced.
\end{remark}

A CohFT $\bigl(V, \langle,\rangle, \mathbbm{1}, \Omega_{g,n}\bigr)$ 
is called {\it homogeneous} of {\it charge}~$d$ if it satisfies the following axiom:

\noindent C4. There is a vector field~$E$ on~$B$ which for some choice of the basis~$e_\alpha$ of~$V$ has the form 
\beq\label{evf}
E \= \bigl(1-\tfrac d 2\bigr) \, v^\alpha \, \p_\alpha  \+ r^\alpha \, \p_\alpha  
\,- \, \sum_{\alpha=1}^l \mu_\alpha \, v^\alpha \, \p_\alpha \,,
\eeq
and such that if we define the action of~$E$ on~$\Omega$ by
\begin{align}
& (E\Omega)_{g,n}(e_{\alpha_1},\dots,e_{\alpha_n}) \:= 
\biggl({\rm gr} + \tfrac{2-d}2 n -\sum_{i=1}^n \mu_{\alpha_i} \biggr) \, \Omega_{g,n}(e_{\alpha_1},\dots,e_{\alpha_n})  \nn\\
& \qquad\qquad\qquad\qquad\qquad\qquad \+  p_* \, \Omega_{g,n+1}\bigl(e_{\alpha_1},\dots,e_{\alpha_n}, \, r^\alpha e_{\alpha}\bigr) \,,\label{EOmega}
\end{align}
where $p$ denotes the forgetful map, and ${\rm gr}$ is the grading operator defined by
\beq\label{gradingop}
{\rm gr} \, \phi \:= q\, \phi \, ,\quad \mbox{if }  \phi\in H^{2q}(\overline{\mathcal{M}}_{g,n}; \CC)\,,
\eeq
then
\beq\label{cohoft4defeqn}
(E\Omega)_{g,n} \= \bigl((g-1)\,d + n\bigr) \, \Omega_{g,n} \,,\quad \forall~2g-2+n>0 \, .
\eeq

\smallskip

The genus~$g$ correlators of the CohFT are defined by
\begin{equation}\label{deficorrelator}
 \int_{\overline{\mathcal M}_{g,n}} 
 \Omega_{g,n}(e_{\alpha_1},\dots, e_{\alpha_n}) \, \psi_1^{q_1} \cdots \psi_n^{q_n} \; =: \;
 \langle \tau_{\alpha_1, q_1}\cdots \tau_{\alpha_n , q_n} \rangle_{g}^{\Omega}\, , \qquad q_1,\dots,q_n \geq 0 \, .
\end{equation}
Below, the label~$\Omega$ will often be omitted. 
The genus~$g$ free energy of the CohFT is defined by
\beq\label{FgCohFT}
\F_g(\bt) \:= \sum_{n\geq0} \sum_{q_1,\dots,q_n\geq 0}  
\frac{t^{\alpha_1,q_1} \cdots t^{\alpha_n,q_n}}{n!}
\; \langle \tau_{\alpha_1, q_1}\cdots \tau_{\alpha_n, q_n} \rangle_g \,,
\eeq
where $\bt:=(t^{\alpha,q})_{\alpha=1,\dots,l, q\geq 0}$ denotes the infinite vector 
of indeterminates. The exponential 
\beq\label{defpartitionfunction}
e^{\sum_{g\geq 0} \e^{2g-2} \F_g(\bt)} \; =: \;  Z
\eeq
is called the {\it partition function} of the CohFT, where~$\e$ is an indeterminate. It 
satisfies the following string equation
\beq\label{stringgeneralZ}
\sum_{q\geq 1} t^{\alpha,q} \frac{\p Z}{\p t^{\alpha,q-1}} \+ \frac1{2\e^2} \eta_{\alpha\beta} t^{\alpha,0} t^{\beta,0} Z 
\= \frac{\p Z}{\p t^{1,0}}\,.
\eeq
An immediate consequence 
of the string equation is that
\beq\label{stringrec}
\langle \tau_{\alpha, q-1} \rangle_g \= \langle \tau_{\alpha, q} \tau_{1,0} \rangle_g \,.
\eeq

The CohFT $\Omega_{g,n}$ given by the Witten class $c_g(\alpha_1, \dots, \alpha_n)$
is also recognized as the FJRW CohFT of $A_{r-1}$-type.  
This is the reason for the notation used in~\eqref{acasedefinition}.
For each simply-laced simple Lie algebra~$\g$ of rank~$l$, 
Fan, Jarvis and Ruan~\cite{FJR0,FJR} constructed a rank~$l$ 
homogeneous CohFT from a certain simple singularity such that the corresponding Frobenius manifold is isomorphic to the Frobenius 
manifold associated to the Weyl group of~$\g$~\cite{Du-5,Zuber}. This CohFT will be referred to as the 
FJRW CohFT of $\g$-type, whose partition function is shown to be a particular tau-function for the 
DS hierarchy of~$\g$-type~\cite{DS,FJR,Witten1,Witten2} (cf.~\cite{DLZ, FFJMR, FGM, Givental-3, GM, KW, Kontsevich, LRZ, Wu}.
In~\cite{LRZ} Liu, Ruan and~Zhang introduced  
the notion of a {\it partial CohFT}, and 
constructed certain partial CohFTs associated to simple singularities, whose   
partition functions are proved {\it ibid}.~to be tau-functions for   
the DS hierarchies of BCFG-type. 
Correlators of these CohFTs or partial CohFTs are referred to as the FJRW--LRZ invariants.
The main focus of this paper will be on the ADE cases, leaving  
the more detailed studies of the BCFG cases to future publications. 
It should be noted that the terminology 
``DS hierarchy of $\g$-type" refers to the DS hierarchy, under the choice of a principal nilpotent element, 
associated to the untwisted affine Kac--Moody algebra~$\hat\g^{(1)}$ (see page~1402 of~\cite{BDY1} for more details and see~\cite{BDY2}).

Now let $\g$ be a {\it simply-laced} simple Lie algebra of rank~$l$ with the normalized Cartan--Killing form~$(\cdot|\cdot)$.
Denote by~$r$ the Coxeter number of~$\g$, and by $m_1,\dots,m_l$ the exponents of~$\g$. 
Here we order $m_1,\dots,m_l$ such that $1=m_1\le m_2 \le \dots\le m_l=r-1$, except 
for the $D_l$ case, where we set 
$m_\alpha=2\alpha-1$ for $\alpha=1,\dots,l-1$ and $m_l=l-1$. 
(Observe the difference with~\cite{BDY1,BDY2}, where for the $D_l$ case 
the exponents are numbered by $1=m_1\le m_2\le \dots\le m_l=r-1$.)
Let $\bigl(V^l, \langle,\rangle, \mathbbm{1}, \Omega_{g,n}\bigr)$ denote the FJRW CohFT of $\g$-type. 
Choose a basis~$e_\alpha$ of~$V$ satisfying   
\beq
\Omega_{g,n}(e_{\alpha_1},\dots,e_{\alpha_n})  
\, \in \, H^{2(s+g-1)}\bigl( \overline{\mathcal M}_{g,n} \bigr) \, , 
\quad s \:= \frac{\sum_{i=1}^l m_{\alpha_i}-n-(2g-2)}{r} \, . 
\eeq
The existence of such a choice can be verified case by case from the data provided in~\cite{FJR}, and this   
 implies that the CohFT~$\Omega_{g,n}$ is {\it homogeneous} of charge~$d=\frac{r-2}r$.
The associated correlators~\eqref{deficorrelator} are called {\it genus $g$ FJRW invariants of $\g$-type}. 
The degree-dimension matching implies that 
$\langle \tau_{\alpha_1, q_1}\cdots \tau_{\alpha_n , q_n} \rangle_g$ vanishes unless
\beq\label{ddg}
\frac{\sum_{i=1}^n m_{\alpha_i}-n-(2g-2)}{r} \+ g-1 \+ \sum_{i=1}^n q_i  \= 3g \,-\, 3 \+ n \,.
\eeq
For given $(\alpha_i,q_i)$, $i=1,\cdots,n$, for simplicity we sometimes omit the 
subindex~$g$ in $\langle\cdots\rangle_g$, because for all possibly non-zero invariants this subindex can be 
reconstructed by~\eqref{ddg}; in such an omission, if the reconstructed~$g$
is not an integer, $\langle \tau_{\alpha_1, q_1} \cdots \tau_{\alpha_n, q_n} \rangle$ 
is defined as~0. It is also convenient to take~$\e=1$ in~\eqref{defpartitionfunction} for the 
definition of the partition~$Z$, i.e., we have
\beq
Z \= Z(\bt) \:= e^{\sum_{g\geq 0} \F_g(\bt)} \,,
\eeq
and the string equation reads
\beq\label{stringgeneralZFJRW}
\sum_{q\geq 1} \, t^{\alpha,q} \, \frac{\p Z(\bt)}{\p t^{\alpha,q-1}} \+ \frac1{2} \, \eta_{\alpha\beta} \, t^{\alpha,0} \, t^{\beta,0} \, Z(\bt) 
\= \frac{\p Z(\bt)}{\p t^{1,0}}\,.
\eeq

The numbers $\tau_{\g}(g)$ that we are studying in this paper are defined by 
\beq
\tau_{\g}(g) \:=
 \langle \tau_{\alpha,q}\rangle\,,\quad g\geq0\,,
\eeq
where 
\beq\label{generalgap}
2\, (r+1) \, g \, - \, 1 \= m_\alpha  \+ r\,(q+1) \,.
\eeq
It should be noticed that for the case $\g=D_l$ with $l$ being an even number, there 
are two equal exponents~$m_{l/2}$ and~$m_{l}$, so for this case, 
 the different $(\alpha=l/2,q)$ and $(\alpha=l,q)$ correspond to the same~$g$. However, 
with an appropriate choice of the basis, the numbers 
$\langle \tau_{l,q} \rangle$ vanish, and we use $\tau_{\g}(g)$ to denote $\langle \tau_{l/2,q} \rangle$.

\section{Differential equation}\label{dualtopoodesection}
The topological and dual topological ODEs of~$\g$-type are introduced in~\cite{BDY1} 
for computing the FJRW--LRZ invariants, 
which are obtained as a result of the theorems of Fan--Jarvis--Ruan and Liu--Ruan--Zhang 
together with an application of the 
matrix-resolvent method~\cite{BDY0,BDY1,BDY2} for the Drinfeld--Sokolov hierarchy of~$\g$-type~\cite{DS}.
Differential equations similar to the topological ODE of~$\g$-type were also considered 
from other perspectives; see for examples~\cite{FF11, FV09, MRV}.
In this section, via reducing the dual topological ODE to a scalar differential equation,
we generalize parts~(i) and~(ii) of Theorem~\ref{mainr5} of the Introduction from~$A_4$ 
to an arbitrary simple Lie algebra~$\g$. 

Fix~$\h$ a Cartan subalgebra of~$\g$ and let $\triangle$ be the root system. 
Choose a set of simple roots~$\Pi$, and let $E_1,\dots,E_l,F_1,\dots,F_l$ be the Weyl generators.
Denote by~$\theta$ the highest root with respect to~$\Pi$, and by~$E_{-\theta}$ a root vector associated to~$-\theta$.
The semisimple element~$\Lambda=\Lambda(\lambda)$ (aka the Kostant element) is defined by 
$\Lambda(\lambda) = \sum_{\alpha=1}^l E_\alpha + \lambda\, E_{-\theta}$,~$\lambda\in \CC$.
Kostant~\cite{Kostant} shows that for any $\lambda\neq 0$, 
$\g$ has the orthogonal decomposition with respect to~$(\cdot|\cdot)$, i.e. 
$\g = {\rm Ker} \, {\rm ad}_{\Lambda(\lambda)} \oplus {\rm Im} \, {\rm ad}_{\Lambda(\lambda)}$, 
${\rm Ker} \, {\rm ad}_{\Lambda(\lambda)} \perp {\rm Im} \, {\rm ad}_{\Lambda(\lambda)}$.
Denote by $L(\g) = \g\otimes \CC\bigl[\lambda,\lambda^{-1}\bigr]$ the loop algebra, and by~$\rho^\vee$ 
the Weyl co-vector (which is the unique element in~$\h$ satisfying 
$[\rho^\vee,E_\alpha]=E_\alpha$). Introduce a grading operator on~$L(\g)$
\beq
{\rm gr} \:= {\rm ad}_{\rho^\vee} \+ r \, \lambda\, \frac{d}{d\lambda} \,. 
\eeq
An element $q$ in~$L(\g)$ is called homogeneous of principal degree~$k$ if ${\rm gr} \, q = k\, q$. 
It was proven by Kac~\cite{Kac} that the kernel of 
${\rm ad}_{\Lambda(\lambda)} \mbox{ in } L(\g)$ has the form 
$\bigoplus_{j\in E} \CC \Lambda_j(\lambda) $.
Here, $E$ is the exponent set of $L(\g)$, and $\Lambda_j(\lambda)\in L(\g)$ are homogeneous of principal degree~$j$, normalized by 
$\Lambda_1=\Lambda$, 
$\Lambda_{m_\alpha+\ell h} (\lambda) = \lambda_{m_\alpha}(\lambda) \, \lambda^\ell$, $\ell\geq 0$, 
and $\bigl(\Lambda_{m_\alpha}|\Lambda_{m_\beta}\bigr)= r \, \lambda \, \eta_{\alpha\beta}$ 
for some non-degenerate symmetric constant matrix~$\eta$. 

Write \[\rho^\vee \= \sum_{i=1}^n x_i \,H_i,\,x_i\in\CC\,,\] 
and define 
$
I_-=2\sum_{i=1}^{l} x_i \,F_i\,.
$
Then $I_+,I_-,\rho^\vee$ form an $sl_2(\mathbb{C})$ Lie algebra:
\beq
\label{SL2princ}
[\rho^\vee, I_+] = I_+\,,\quad [\rho^\vee,I_-]=-I_-\,,\quad [I_+,I_-]=2\,\rho^\vee\,.
\eeq
According to~\cite{BFRFW,Kostant},  there exist elements $\gamma^1,\dots,\gamma^{l}\in \g$ such that
\beq
{\rm Ker} \, {\rm ad}_{I_-}={\rm Span}_{\mathbb{C}} \bigl\{\gamma^1,\dots,\gamma^{l}\bigr\}\,, \qquad  [\rho^\vee,\gamma^i] \= -m_i \, \gamma^i\,.
\eeq
Fix $\bigl\{\gamma^1,\dots,\gamma^{l}\bigr\}$, 
then the lowest weight decomposition of~$\g$ has the form
\beq
\g\=\bigoplus_{i=1}^{l} \mathcal{L}^i \,,\quad 
\mathcal{L}^i \= {\rm Span}_{\mathbb{C}} \bigl\{\gamma^i,\ad_{I_+} \gamma^i, \dots, \ad_{I_+}^{2m_i} \gamma^i \bigr\} \,.
\eeq
Here each $\mathcal{L}^i$ is an $sl_2(\mathbb{C})$-module. Any $\g$-valued function
$M(\lambda)$ can be uniquely represented as
\beq\label{;;}
M(\lambda)=\sum_{i=1}^{l} S_{i}(\lambda)  \, \ad_{I_+}^{2m_i} \gamma^i+\sum_{i=1}^{l} \sum_{m=0}^{2m_{i}-1} K_{im}(\lambda) \, \ad_{I_+}^m \gamma^i \,,
\eeq
where $S_{i}(\lambda)$, $K_{im}(\lambda)$ are certain complex-valued functions. Note that $\ad_{I_+}^{2m_i} \gamma^i$ is the highest weight vector 
of $\mathcal{L}^i$, $i=1,\dots, l$.  

The dual topological ODE of $\g$-type~\cite{BDY1} is an ODE for a $\g$-valued function~$G$ defined by
\beq\label{dualtopo}
\biggl[\frac{dG}{dx} , E_{-\theta}\biggr] \+ \bigl[G , I_+\bigr] \+ x \, G \= 0 \,.
\eeq
Write 
\beq
G(x) \= \sum_{\alpha=1}^l \phi_\alpha (x)  \, {\rm ad}_{I_+}^{2m_\alpha} \gamma^\alpha 
\+ \sum_{\alpha=1}^l \sum_{m=0}^{2m_\alpha-1} \widetilde K_{\alpha m}(x) \, {\rm ad}_{I_+}^m \gamma^\alpha \,.
\eeq
It is shown in~\cite{BDY1} that the ODE~\eqref{dualtopo} is equivalent to an ODE for $\phi=(\phi_1,\dots,\phi_l)^T$ 
\beq\label{dualphi}
\frac{d\phi}{dx} \= \sum_{i=-1}^{2r-2} x^i \, V_i \, \phi   \,, 
\eeq
where $V_i$ ($i=-1,\dots,2r-2$) are constant $l\times l$ matrices and 
$V_{-1}={\rm diag}(-m_{l+1-\alpha}/r)_{\alpha=1,\dots,l}$.
(Equivalence between systems of ODEs means that their solutions 
have a one-to-one correspondence.)
Here we note that $x=0$ is a Fuchsian singular point of~\eqref{dualphi} and that 
the matrices $V_{-1}$, \dots , $V_{2h-2}$ are determined by the 
lowest-weight-structure-constants of~$\g$. 
It is obvious from~\eqref{dualphi} that the dimension of the space of solutions to~\eqref{dualtopo} is equal to 
the rank of~$\g$.

Denote by~$(G_\alpha)_{\alpha=1,\dots,l}$ a solution basis of~\eqref{dualtopo}, 
and by~$\phi_{\alpha;1},\dots,\phi_{\alpha;l}$ the $\phi$-coefficients of~$G_\alpha$. 
The matrix $\Phi$ defined by $\Phi_{\beta\alpha}:=\phi_{\alpha;\beta}$ is called the fundamental solution matrix, which can 
be normalized by using the following initial condition near $x=0$:
\beq\label{normalized}
\Phi  \= D \, \Bigl( I  \+ \sum_{m\geq 1} \Phi_m \, x^m \Bigr) \,, \qquad D\= {\rm diag} \Bigl( x^{-\frac{m_{l+1-\alpha}}{r}} \Bigr)_{\alpha=1,\dots,l} \,. 
\eeq
For each fixed~$\alpha$ we allow $G_\alpha$ to have a multiplicative non-zero constant 
and keep using the notations $G_{\alpha}$ and~$\phi_{\alpha;\beta}$. 
The following proposition is proved in~\cite{BDY1}, and we review its proof here.
\begin{prop}\label{dualandinvsg}
The series $\phi_{\alpha;l}$ can be expressed in terms of the FJRW invariants of $\g$-type by 
\begin{align}
\phi_{\alpha;l}(x) 
 \= \frac{x^{-\frac1r}}{\Gamma(\frac{m_\alpha}{r})} \, \delta_{\alpha, l}  \+  \frac{x^{\frac{m_\alpha}r}}{\Gamma(\frac{m_\alpha}r)}  
 \sum_{w \geq 0} \, c_{\alpha,w}  \, x^{m_\alpha+(r+1)w+1} \,, \label{phialphalint}
\end{align}
where 
\beq\label{calphaw}
c_{\alpha,w} \= (-1)^{m_\alpha+(r+1)w} \, \bigl\langle\tau_{\alpha,m_\alpha+(r+1)w} \bigr\rangle \, (-r)^{\frac{m_\alpha+1+r w}2}\,, \quad w\geq0\,.
\eeq
\end{prop}
\noindent Here we note that when $\g$ is a non-simply-laced simple Lie algebra, the 
$\langle\tau_{\alpha,m_\alpha+(r+1)w} \rangle$ in the right-hand side of~\eqref{calphaw} 
 should be considered as the one-point correlators of the corresponding Liu--Ruan--Zhang 
partial CohFT (cf.~\cite{BDY1,BDY2,LRZ}).

\begin{proof}[Proof of Proposition~\ref{dualandinvsg}]
Recall that the topological ODE of $\g$-type~\cite{BDY1} 
is the $\g$-valued ODE
\beq \label{topoODE}
\frac{dM(\lambda)}{d\lambda} \= \bigl[M(\lambda),\Lambda(\lambda)\bigr] \,.
\eeq  
About this ODE, the following statements are proved in~\cite{BDY1,BDY2}:

\noindent {(a)}
The dimension of the formal Puiseux series solutions to~\eqref{topoODE} is equal to the rank of~$\g$.

\noindent {(b)} 
There exists a unique basis $M_1,\dots,M_l$ of the formal solutions to~\eqref{topoODE} 
such that  
\begin{align}
& M_\alpha(\lambda) \= \lambda^{-\frac{m_\alpha}{r}} \biggl[ \Lambda_{m_\alpha}(\lambda)  \+ \sum_{k\geq1} M_{\alpha,k}(\lambda) \biggr] \,, \\
& M_{\alpha,k}(\lambda) \, \in \, L(\g) \,, \quad {\rm gr} \, M_{\alpha,k}(\lambda) \= \bigl[m_\alpha-(r+1) k\bigr] \, M_{\alpha,k}(\lambda) \,.  
\end{align}
\noindent {(c)} 
Denote $\kappa=(\sqrt{-r})^{-r}$. The following identity for one-point FJRW invariants of~$\g$-type is true:
\begin{align}
\kappa^{\frac2{r+1}} \sqrt{-r} \sum_{g,q\geq 0} 
(-1)^q \, \frac{(\frac{m_\alpha}r)_{q+2}}{(\kappa^{\frac1{r+1}} \lambda)^{\frac{m_\alpha}r+q+2}} \, 
\langle\tau_{\alpha,q}\rangle_g \= \bigl(E_{-\theta} | M_\alpha(\lambda)\bigr) 
\,-\, \lambda^{-\frac{r-1}r} \, \delta_{\alpha,l}\,.\label{onepoint}
\end{align}  
Write 
\beq
\bigl(E_{-\theta} | M_\alpha(\lambda)\bigr) \= \lambda^{-\frac{m_\alpha}r} \, \sum_{q \geq 0} 
S_{\alpha,q} \, \lambda^{-q}  \,=:\,S_\alpha(\lambda) \,, \qquad S_{\alpha,q} \in \QQ\,.
\eeq
We have for $w\in \ZZ$ and $(1+m_\alpha+r w)/2 \in \ZZ_{\geq 0}$, 
\beq
\bigl\langle\tau_{\alpha,m_\alpha+(r+1)w} \bigr\rangle \=  (-1)^{m_\alpha+(r+1)w}  \frac{S_{\alpha,m_\alpha+(r+1)w+2}}{(-r)^{\frac{m_\alpha+1+r w}2}  \, \bigl(\frac{m_\alpha}r\bigr)_{m_\alpha+(r+1)w+2}} \,. \label{intersectionnumberS}
\eeq
According to the definition given in~\cite{BDY1}, 
solutions to topological and dual topological ODEs are related via the Laplace transform
$M(\lambda) = \int G(x) \, e^{-\lambda \, x} \, dx$.
In particular, 
the series $\phi_{\alpha;l}$ are related to~$S_\alpha$ by 
\beq
S_\alpha(\lambda) \= \int_C \phi_{\alpha;l} (x) \, e^{-\lambda \, x} \, dx \,,
\eeq
where $C$ is a carefully chosen contour (which can depend on~$\alpha$) on the $x$-plane. We have
\beq
\phi_{\alpha;l}(x) \=  \sum_{q\geq 0} \, S_{\alpha,q}  \,  \frac{x^{q+\frac{m_\alpha}r-1}}{\Gamma(q+\frac{m_\alpha}r)} 
\= \frac{x^{\frac{m_\alpha}r-1}}{\Gamma(\frac{m_\alpha}r)} \,  \sum_{q\geq 0} \, S_{\alpha,q}  \,  \frac{x^q}{\bigl(\frac{m_\alpha}r\bigr)_q } \,.
\eeq
The proposition is proved.
\end{proof}

An immediate consequence of Proposition~\ref{dualandinvsg} is that 
the intersections numbers $\tau_\g(g)$ grow at most exponentially.  
Let us introduce the series~$a(x)$ by
\beq\label{defvarphi}
a (x) \:= (-r)^{-\frac1{2(r+1)}} \, \sum_{\alpha=1}^l \, \Gamma\Bigl(\frac{m_\alpha}r\Bigr) \, (-1)^{\frac{m_\alpha-r}{r}} \, \phi_{\alpha;l}(x) \,,
\eeq
where $\phi_{\alpha;l}(x)$ are given by~\eqref{phialphalint}.
More explicitly, 
\begin{align}
& a(x) \= (-r)^{-\frac1{2(r+1)}} (-x)^{-\frac1r}  \+  \sum_{\alpha=1}^l  \sum_{w \geq 0} 
(-1)^{m_\alpha+(r+1)w+\frac{m_\alpha-r}{r}} \nn\\ 
& \qquad \qquad \qquad \qquad \langle\tau_{\alpha,m_\alpha+(r+1)w} \rangle \,
(-r)^{\frac{m_\alpha+1+r w}2}  x^{\frac{(r+1)m_\alpha+r(r+1)w+r}r}  \nn\\
& \qquad 
\= \Bigl(-(-r)^{\frac{r}{2(r+1)}}x\Bigr)^{-\frac1r}  \+ \Bigl( -(-r)^{\frac{r}{2(r+1)}}x \Bigr)^{-\frac1r}  \sum_{g\geq 0}
\tau_\g (g)   \, \Bigl(- (-r)^{\frac{r}{2(r+1)}} x \Bigr)^{2g+\frac{2g}r} \,. \nn
\end{align}
Denote $t=-(-r)^{\frac{r}{2(r+1)}}x$ and 
\[b(t) \= t^{-\frac1r}  \+  t^{-\frac1r}  \sum_{g\geq 1} \,
\tau_\g (g)  \, t^{2g+\frac{2g}r} \=   \sum_{g\geq 0} \,
\tau_\g (g)  \, t^{2g\frac{r+1}r-\frac1r} \,. \]
We have $a(x)=b(t)$.
We are ready to generalize part~(i) of Theorem~\ref{mainr5} by showing that 
 $b(t)$ satisfies an ODE with polynomial coefficients.
\begin{theorem}\label{odethmg} Let $\g$ be an arbitrary simple Lie algebra of rank~$l$.
The component~$\phi_l$ satisfies a linear ODE with polynomial coefficients of order at most~$l$. In other words, 
for every fixed $\alpha$, the series $\phi_{\alpha;l}$ satisfies this ODE. In particular, $b(t)$ satisfies an ODE of order at most~$l$.
\end{theorem}
\begin{proof} \footnote{The proof of Theorem~\ref{odethmg} was written after the first author~B.D. passed away.}
For $\g$ with $l=1$, the statement is trivial. Below we consider $l\geq 2$.
As we have mentioned, the dual topological ODE is equivalent to system~\eqref{dualphi}. 
This system consists of 
$l$~linear equations for $\phi_1,\dots,\phi_l$. Denote these equations by $e_1,\dots,e_l$. 
Let $N$ be a positive integer. 
Now consider the following linear 
combination of equations
\beq
e \:= \sum_{k=0}^N c^{l}_k \frac{d^k e_l}{dx^k} \+ \sum_{\alpha=1}^{l-1}\sum_{k=1}^N c^{\alpha}_k \frac{d^{k-1} e_\alpha}{dx^{k-1}}\,.
\eeq
Requiring that the coefficients of $\phi_1,\dots,\phi_{l-1}$, $\frac{d\phi_1}{dx},\dots,\frac{d\phi_{l-1}}{dx}$, \dots, 
$\frac{d^{N-1}\phi_1}{dx^{N-1}},\dots,\frac{d^{N}\phi_{l-1}}{dx^{N}}$ vanish gives rise to a system of linear equations
for $c^l_0$ and $c^{\alpha}_k$, $k=1,\dots,N$. The number of equations is equal to $(l-1)(N+1)$, while the number of 
unknowns is $1+Nl$. Take $N= l-1$, we have
\beq
1+Nl \,-\, (l-1)(N+1) \= 2 \,+\, N \,-\, l \= 1\,.
\eeq
According to the Gauss elimination we know that this linear system over the field~$\CC(x)$ has a non-zero solution 
$c^l_0$ and $c^{\alpha}_k$, $k=1,\dots,N$. The theorem is proved by further noticing that  
$\frac{d^k e_l}{dx^k}$ ($k=0,\dots,N$) and $\frac{d^{k-1} e_\alpha}{dx^{k-1}}$ ($\alpha=1,\dots,l-1$, $k=1,\dots,N$) together are linearly independent. 
\end{proof}

It is clear that part~(i) of Theorem~\ref{mainr5} is a special case of Theorem~\ref{odethmg} with the particular form 
of the linear ODE~\eqref{odethma4} computed from the standard $sl_5(\CC)$ realization with the $\Lambda(\lambda)$ given by
\beq
\Lambda(\lambda) \= \begin{pmatrix} 0 & 1 & 0 & 0 & 0 \\ 0 & 0 & 1 &0 &0 \\  0 &0 &0 &1 &0 \\ 0 &0 &0 &0 &1 \\ \lambda & 0 &0&0&0\end{pmatrix} \,,
\eeq
and from the change of independent variable $ t \mapsto X=t^{\frac{2(r+1)}r}$ and $\varphi(X) = b(t) \, t^{\frac1r}$.

\section{Computing residues}\label{techsection}
In this section, we carry out several technical preparations for the later sections. 
It is very convenient here to permit the variable~$r$, which originally
came from $r$-spin classes and hence was a positive integer, to take on arbitrary
complex values, since essentially all of the identities we need are polynomial
in~$r$ (at least at the level of the individual coefficients of the generating
functions appearing).

Let~$r$ be a complex number and let $L$ be the pseudodifferential operator
\beq
L \= \p^r \+ C \, x \,,
\eeq 
where $\p=d/dx$ and $C$ is an arbitrarily given constant. 
We are to compute the residues~$z_k(x)$ of the pseudodifferential operators~$L^{k/r}$, i.e.
\beq\label{Lzk}
z_k(x)\:=\res L^{\frac{k}{r}} \,.
\eeq
We note that
most of the time we consider~$k$ as a nonnegative integer, although sometimes 
we give results valid for any~$k\in\CC$. For the definition of a 
pseudodifferential operator and its residue we refer to~\cite{CMZ,Dickey,KhZ}.

\subsection{Closed formula} 
Let us first introduce some notations. Define polynomials $c_{p, j}(r)\in\QQ[r]$ $(0 \leq p \leq j)$ by 
the generating function
\beq\label{cpjgeneratingdef}
 \frac1{p!} \, \biggl(\frac{(1+x)^{r+1}-1-(r+1) \,x}{(r+1) \, x}\biggr)^{p} \= \sum_{j=p}^{\infty} \, c_{p,j}(r) \, x^j
\eeq
(then $c_{p,j}(r)$ has degree~$j$ and is divisible by $r^{p}$), the first few values being

\begin{align}
& \arraycolsep=5.0pt\def\arraystretch{1.5}
\begin{array}{?c?c|c|c|c|c|c?}
\Xhline{2\arrayrulewidth}   & p=0 & p=1 & p=2 & p=3 & p=4 & p=5 \\
\Xhline{2\arrayrulewidth} j=0  & 1 &  &  &  &  &   \\
\hline j=1 & 0 & r/2 &  &  &  &  \\
\hline j=2 & 0 & \frac13 \binom{r}{2} & \frac{(r/2)^2}{2!} &  &   &   \\
\hline j=3 & 0 & \frac14 \binom{r}{3} & \frac{r^2(r-1)}{12} & \frac{(r/2)^3}{3!} &   &   \\
\hline j=4 & 0 & \frac15 \binom{r}{4} & \frac{r(5r-8)}{72} \binom{r}{2} & \frac{r^3(r-1)}{48} & \frac{(r/2)^4}{4!} &   \\
\hline j=5 & 0 & \frac16 \binom{r}{5} & \frac{r(4r-7)}{60} \binom{r}{3} & \frac{r^3(r-1)(7r-10)}{576} & \frac{r^4(r-1)}{288} & \frac{(r/2)^5}{5!}  \\
\Xhline{2\arrayrulewidth}
\end{array} \nn \\
& \qquad \qquad \quad \quad \mbox{The polynomials } c_{p,j}(r)~( 0 \le p \le j \le 5) \nn
\end{align}

The result of this subsection is given by the following proposition.

\begin{prop}\label{Don'sformula}
Let $L=\p^r+Cx$.
For any $r,\lambda\in\CC$, we have
\beq\label{Don'sformulaL}
L^\lambda \= \sum_{ j,s\geq 0} \, C^{j+s} \, d_{j,s}(\lambda,r) \, x^s \, \p_x^{r(\lambda-s)-(r+1)j} \, ,
\eeq
where $d_{j,s}(\lambda,r)$ is a polynomial in~$\lambda$ and~$r$ with rational coefficients,  
given explicitly by 
\beq\label{dcsi} d_{j,s}(\lambda,r) \= \frac1{s!} \, \sum_{p=0}^{j} c_{p,j}(r) \, (\lambda)_{s+p+j}^- \,. \eeq
Here $(\lambda)_n^-:= \prod_{m=1}^n (\lambda-m+1)$ denotes the descending Pochhammer symbol. 
\end{prop}
\begin{proof}
By the rules for manipulating pseudodifferential operators~\cite{CMZ,Dickey,KhZ}, 
it is clear that $L^\lambda$ has the form~\eqref{Don'sformulaL} for some polynomials $d_{j,s}(\lambda,r)$. 
 It then suffices to prove~\eqref{dcsi} for $\lambda\in \ZZ_{\ge 0}$, which we do by induction, 
the case $\lambda=0$ being trivial (both sides of~\eqref{dcsi} then reduce to $\delta_{j,0} \, \delta_{s,0}$).
 Expanding the identity $L^{\lambda+1} = L^\lambda L$, we find that 
$d_{j,s}$ satisfies the recursive relation 
$$ d_{j,s} (\lambda+1,r) \,-\, d_{j,s-1} (\lambda,r) \,-\,  d_{j,s}(\lambda,r) 
 \= \bigl( r (\lambda-s)-(r+1)(j-1)\bigr) \,  d_{j-1,s}(\lambda,r) \,. $$
If we multiply both sides of this by $s!$ and replace $d_{j,s}$ everywhere by the expression given in~(74), 
then the left-hand side equals 
$$ \sum_{p\geq0} \, c_{p,j}\,\bigl[ (\lambda+1)_{s+p+j}^{-} - s \, (\lambda)_{s+p+j-1}^{-} - (\lambda)_{s+p+j}^{-}\bigr] 
   \= \sum_{p\geq0} \, (p+j) \, c_{p,j}\,  (\lambda)_{s+p+j-1}^{-} $$ 
and the right-hand side equals 
$$ \sum_{p\ge0} \, c_{p,j-1} \, \bigl[\,r\,(\lambda-s)-(r+1)\,(j-1)\,\bigr] \, (\lambda)_{s+p+j-1}^{-} 
 \= \sum_{p\ge0} \, c_{p,j-1} \, \bigl[\,r \, (\lambda)_{s+p+j}^{-} + (rp-j+1) \, (\lambda)_{s+p+j-1}^{-}\,\bigr]\,. $$
(Define $c_{p,j}=0$ if $p>j$.)
Comparing the coefficients of~$(\lambda)_n^-$ on both sides, we see that the claim follows from the recursion
$ (p+j) \, c_{p,j} = (rp-j+1) \, c_{p,j-1} + r \, c_{p-1,j-1}$,
which follows easily by differentiating~\eqref{cpjgeneratingdef}.

We observe that there is a slightly different proof, which does not depend on the polynomiality, 
obtained by using the identity $L\,L^\lambda = L^\lambda\, L$ instead of $L^\lambda L = L^{\lambda+1}$. 
Expanding both sides and comparing coefficients, we find the identity 
 $$  \bigl(\, r \, (\lambda-s+1)-(r+1)\,j\,\bigr) \,  d_{j,s-1}(\lambda,r)  \,-\, 
   r \, s \,  d_{j,s}(\lambda,r)
 \= \sum_{m=1}^{j} \binom r{m+1} \, (s+m)^-_{m+1} \,  d_{j-m,s+m}(\lambda,r)    \,,  $$ 
which determines the~$d_{j,s}$ completely by a double induction (first on~$j$, and then for a given~$j$ on~$s$).
The proposition then follows by verifying that the right-hand side of~\eqref{dcsi} satisfies the same identity, 
which is an elementary exercise using the generating function~\eqref{cpjgeneratingdef}.
\end{proof}

\begin{cor}\label{corclosedformula}
For arbitrary positive integers~$r$ and~$k$, the polynomial $z_k(x) =\res L^{\frac{k}{r}}$ 
is given by $z_k(0) = C^j d_{j,0}(k/r,r)$ if $k=-1 + (r+1)j$ with $j\ge 0$ and is~$0$ otherwise, 
and similarly for any $s\ge0$, $z_k^{(s)}(0)$ is equal to $C^{s+j} d_{j,s}(k/r,r)$ if $k=-1 + (r+1)j$ with $j\ge 0$ and vanishes otherwise.
\end{cor}

\subsection{Product formula} 
In this subsection we will derive another formula for~$z_k(0)$ using the wave-function-pair 
approach (see Appendix~\ref{wavebi} or~\cite{DYZ}).  
Before doing this, we will first introduce some functions and prove several lemmas that are 
useful for the construction. 
\begin{defi}\label{defcfj}
Define a power series (algebraic if $r$ is rational)
\beq \label{wudef}
w \= 1 \+ u \, - \, \frac{r-1}6\,u^2 \+ \frac{(r-1)(2r+1)}{72}\,u^3 \,-\, \cdots
\eeq
by 
 \beq \label{hypergeom}
 \frac{w^{r+1}}{r(r+1)} \, - \, \frac wr \+ \frac1{r+1}\= \frac{u^2}2
 \eeq
and define coefficients $C_n(r,j)\in\QQ[r,j]$, $n\geq 0$ by
\beq \label{defc}
   \frac{w^{j+1} -1}{j\+1}\= \sum_{n\ge0} C_n(r,j)\,u^{n+1} 
   \quad\text{(replace LHS by $\,\log w\,$ if $j=-1$)\,,} 
\eeq
the first few values being
\begin{align} 
& C_0(r,j)\=1\,, \qquad C_1(r,j)\ = \frac j2 \,-\, \frac{r-1}6\,, \qquad
   C_2(r,j) \= \frac{j(j-r)}6  \+ \frac{(r-1)(2r+1)}{72} \,, \nn \\
& C_3(r,j)\= \frac{j(j-r)(j-r-1)}{24} \,-\, \frac{(r-1)(r+2)(2r+1)}{540}\, .  \nn 
\end{align}
 Define $C_{-2}(r,j)=C_{-1}(r,j)=0$. 
 \end{defi}
 
\begin{remark} 
Immediately from the definition we find a further property for the coefficients $C_n(r,j)$, that is, 
they have an $S_3$-symmetry generated by two involutions
\beq 
C_n(r,j) \= (-1)^n\,C_n\Bigl(-r-1,\frac{n-3}2-j\Bigl) 
  \= (r+1)^n\,C_n\Bigl(\frac{-r}{r+1},\frac{j-r}{r+1}\Bigr)\,.
 \eeq
We hope to investigate this very intriguing property and its applications later.
\end{remark}
 
We now define a power series $f_j(T)=f_{r,j}(T)\in\QQ[r,j][[T]]$ (we usually omit~$r$) by
\beq\label{fjT}
 f_j(T) \:= \sum_{k\ge0} \, (2k+1)!!\,C_{2k}(r,j) \,(-T)^k \,.
\eeq
Note that this is defined for all~$r$ and~$j$ in~$\CC$, not just for non-negative integers, 
because the $C_n(r,j)$ are polynomials.
 
 \begin{lemma}\label{fjTrecursionlemma}
The power series 
$f_j(T)$ $(j\in \CC)$ satisfy the following two identities:
\begin{align}
& f_{j+1}(T) \= \biggl(1\+ \Bigl(\frac{r-1}2-j\Bigr)\,T \+  (r+1)\,T^2\frac d{dT} \biggr)\,f_j(T)\,, \label{fid1}\\
& f_{j+r}(T) \= f_j(T) \,-\, r \, j \, T \, f_{j-1}(T)\,. \label{fid2}
\end{align}
\end{lemma}
\begin{proof}
We have by definition 
\beq \label{ddefc}
w^j \, \frac{dw}{du} \= \sum_{n\geq0} \, (n+1) \, C_n(r,j) \, u^n \,. 
\eeq
Applying~$d/du$ to~\eqref{hypergeom} gives 
$$
\Bigl(\frac{w^r}{r} - \frac1r \Bigr) \, \frac{dw}{du} \= u   \quad \Rightarrow  \quad 
w^{r+j} \, \frac{dw}{du} - \, w^j \frac{dw}{du}  \=  r\, u\, w^j \,. 
$$
Substituting~\eqref{ddefc} and \eqref{defc} into this identity, comparing the Taylor coefficients of~$u$, and noticing that 
$C_0(r,j)=1$, $C_1(r,j)= \frac j2 - \frac{r-1}6$, we obtain that for all $n\geq 0$, 
$$
(n+1) \, C_n(r,r+j) \= (n+1) \, C_n(r,j) \+ r \, j \, C_{n-2} (r,j-1) \+ r \, \delta_{n,1} \,, 
$$
which proves~\eqref{fid2}. Similarly, we have 
\begin{align}
\bigl(w^{j+1}-w^j \bigr) \, \frac{dw}{du} 
& \= w^{j+1} \Bigl( w^r \frac{dw}{du}-r\,u\Bigr)  \,-\, w^j \, \frac{dw}{du} \nn\\
& \= w^j \Bigl(\frac{r(r+1)}2 u^2 -r +(r+1) w\Bigr) \, \frac{dw}{du} \,-\, r\, u\, w^{j+1}   \,-\, w^j \, \frac{dw}{du} \,. \nn
\end{align}
Namely,
\begin{align}
\bigl(w^{j+1}-w^j \bigr) \, \frac{dw}{du} 
 \= -\, \frac{r+1}{2} u^2 \, w^j\, \frac{dw}{du} \+  u\, w^{j+1} \nn \,,
\end{align}
which implies~\eqref{fid1}.  The lemma is proved.
\end{proof}

We note that the above relations~\eqref{fid1} and~\eqref{fid2} 
determine $f_{r,j}(T) \in 1 + T \, \CC[r,j][[T]]$ uniquely for all $r,j\in\CC$. To prove this, 
we may assume that $j\ge0$ and $r\ge2$ are integers (since these countably many 
values define the polynomial coefficients uniquely). Then~\eqref{fid1} implies by induction that $f_1,f_2,f_3,\dots,f_r$ are uniquely determined by~$f_0$ 
and then~\eqref{fid2} with $j=0$ gives a differential equation for~$f_0$, which has a unique solution 
in $1 + T \, \CC[r,j][[T]]$. 
(In fact, this proof shows that the $f_{j,r}(T)$ are uniquely determined by just~\eqref{fid1} and the equation $f_0=f_r$.)
 
We now define an odd Laurent series (again algebraic if $r$ is rational)
 $$ X \= \frac1u \,-\, \frac{r-1}6\,u \+ \frac{(r-1)(2r+1)(r-3)}{360}\,u^3 \+ \cdots
  \quad\in\;\frac1u\,\QQ[r][[u^2]] $$
by the equation
\beq\label{deftc1}
 \frac{(X+1)^{r+1} \,-\, (X-1)^{r+1}}{2\, (r+1)} \= \frac1{u^r}  \,,
\eeq
and define coefficients $\tC_n(r,i,j)\in\QQ[r,i,j]$ by
\beq\label{deftc2}
 -\, (X+1)^i\,(X-1)^j\,\frac{dX}{du} \= \sum_{n\geq0} \tC_n(r,i,j)\,u^{n-i-j-2}\;. 
\eeq
The first two $\tC_n(r,i,j)$ are 
given by
\beq
\tC_0(r,i,j) \= 1 \,, \quad \tC_1(r,i,j) \=  i-j\,.
\eeq
Taking $i=j=0$ in~\eqref{deftc2}, we have 
\beq\label{deftc2ij0}
 -\, \frac{dX}{du} \= \sum_{n\ge0} \, \tC_n(r,0,0)\,u^{n-2}\;. 
\eeq
Integrating this equality with respect to~$u$ we find that
\beq\label{deftc2ij0inttoX}
 -X \= - \, \frac1u \+ \sum_{n\geq2}  \, \frac1{n-1} \, \tC_n(r,0,0)\,u^{n-1} \,.
\eeq
\begin{lemma} The numbers $\tC_n(r,i,j)$ satisfy the following relations:
\begin{align}
& \tC_n(r,i+1,j) \,-\, \tC_n(r,i,j+1) \= 2 \, \tC_{n-1}(r,i,j) \,,  \label{tcid1} \\
& \tC_n(r,i+r+1,j) \,-\, \tC_n(r,i,j+r+1) \= 2\, (r+1) \, \tC_{n-1}(r,i,j) \,.  \label{tcid2}
\end{align}
\end{lemma}
\begin{proof}
We have from the defining equation~\eqref{deftc2} that 
\begin{align}
 & -\, (X+1)^{i+1}\,(X-1)^j\,\frac{dX}{du} \= (X+1) \, \sum_{n\ge0} \tC_n(r,i,j)\,u^{n-i-j-2}\,, \nn\\
 & -\, (X+1)^{i}\,(X-1)^{j+1}\,\frac{dX}{du} \= (X-1) \, \sum_{n\ge0} \tC_n(r,i,j)\,u^{n-i-j-2}\,.  \nn
\end{align}
Subtracting these two identities, using again~\eqref{deftc2}, and comparing 
the coefficients of powers of~$u$ gives~\eqref{tcid1}. Similarly, multiplying~$(X+1)^i(X-1)^j$ to 
both sides of the defining equation~\eqref{deftc1}, using ~\eqref{deftc2}, and 
comparing the coefficients of powers of~$u$, we obtain~\eqref{tcid2}.
\end{proof}

\begin{prop}\label{propequivanlencebetweenalgprod}
We have the following identities: $\forall\, i,j\in\CC$,
 \beq\label{fifjTmT}
 f_i(T)\,f_j(-T) \= \sum_{n\geq0} \, \bigl(1+\tfrac{n-i-j-1}r\bigr)_n \,\tC_n(r,i,j)\;\Bigl(\frac{r\,T}2\Bigr)^n\,,
 \eeq
where $(s)_n=s(s+1)\cdots(s+n-1)$ denotes the ascending Pochhammer symbol. 
\end{prop}
\begin{proof}
Write 
\beq
 f_i(T)\,f_j(-T) \= \sum_{n\geq0} \, \bigl(1+\tfrac{n-i-j-1}r\bigr)_n \, P_{i,j,n}(r) \; \Bigl(\frac{r\,T}2\Bigr)^n\,.
\eeq
It then suffices to show that the $P_{i,j,n}(r)$ satisfy the same 
recursion relation as $\tC_n(r,i,j)$. This can be verified straightforwardly using~\eqref{fjT}--\eqref{fid2}.
The proposition is proved.
\end{proof}

Let us now assume that $r$ is a positive integer and use the 
wave-function-pair approach (see Appendix~\ref{wavebi}) to 
compute the residue of~$L^{k/r}$, where we remind the 
reader that 
\[L\=\p^r\+Cx\,.\]
First we construct a particular pair of wave functions 
 $(\psi,\psi^*)$ of~$L$. Start with solving the equation $L\psi=z^r\psi$, that is,
\beq\label{equationpsi}
\bigl(\p^r \,-\, C \, (z^r/C-x) \bigr)\, \psi \= 0\,. 
\eeq
Denote $X = z^r/C - x$. Then we have 
\beq
\bigl((-\p_X)^r \,-\, C X \bigr)\, \psi \= 0 \,.
\eeq
Using the formal saddle point method, we know that this linear ODE has a unique formal Puiseux series solution of the form:
\beq\label{P1a}
P_1(X)\=e^{- \frac{r}{r+1}C^{\frac1r} \, X^{\frac{r+1}{r}}} X^{-\frac{r-1}{2r}}
 \, \sum_{m\geq 0} \frac{c_m}{X^{ \frac{(r+1)m}{r}}} \,, \quad c_0 \:=1 \,.
\eeq
Therefore, any wave function~$\psi(x,z)$ for~$L$ 
can be expressed as $\alpha_1(z)\,P_1(X)$ for some $\alpha_1(z)$. 
Although as we know that the choice of~$\alpha_1(z)$ is not unique, 
let us show that the following function gives a particular choice:
\beq
\alpha_1^{\rm bisp}(z) \:=  \frac{1}{e^{- \frac{r}{r+1}C^{\frac1r} X^{\frac{r+1}{r}}} X^{-\frac{r-1}{2r}}}\bigg|_{x=0} 
\=  C^{-\frac{r-1}{2r}} e^{ \frac{r}{r+1} C^{-1} z^{r+1}} z^{\frac{r-1}2}  \,.  
\eeq
Indeed, define 
\beq\label{bispsia}
\psi \=\psi(x,z) \:= \alpha_1^{\rm bisp} (z) \, P_1(X) \,.
\eeq 
Then it is easy to see that $\psi$ has the form
$\psi = \Phi_1 (e^{xz})$, where  
$\Phi_1 = \sum_{k\geq 0} \phi_{1,k}(x) \, \p^{-k}$ with $\phi_{1,0}\equiv1$. 
Hence the function~$\psi$ constructed by~\eqref{bispsia} is indeed a wave function of~$L$.  
We call this choice of~$\alpha_1(z)$ the {\it bispectral} one.
Similarly, denote by  
\beq\label{P2a}
P_2(X)\:=e^{\frac{r}{r+1}C^{\frac1r} \, X^{\frac{r+1}{r}}} X^{-\frac{r-1}{2r}} 
\, \sum_{m\geq 0} \frac{c_m^*}{X^{ \frac{(r+1)m}{r}}}\,, \quad c_0^*\:=1
\eeq
the unique formal solution to the linear ODE
\beq
\bigl(\p_X^r \,-\, C X \bigr)\, \psi^* \= 0\,.
\eeq
Define 
\beq
\alpha_2^{\rm bisp}(z) \:= C^{-\frac{r-1}{2r}} \, e^{ - \frac{r}{r+1} C^{-1} z^{r+1}} z^{\frac{r-1}2} \,,
\eeq
and construct 
\beq\label{bispsistara}
\psi^*\=\psi^*(x,z) \:= \alpha_2^{\rm bisp} (z) \, P_2(X) \,.
\eeq 
Then it is easy to see that $\psi^*$ can be written as $\psi^*=\Phi_2(e^{-xz})$, 
where $\Phi_2 = \sum_{k\geq 0} \phi_{2,k}(x) \, \p^{-k}$ with $\phi_{2,0}\equiv1$.

\begin{prop}\label{pairprop}
The above $\psi, \psi^*$ form a particular pair of wave and dual wave functions of~$L$. 
\end{prop}
\begin{proof}
Since we already know that $\psi$ is a wave function and $\psi^*$ is a dual wave function,  
we are left to show that $\psi$, $\psi^*$ form a pair. 
From the definition it then suffices to show that $\Phi_1\circ \Phi_2^*\equiv1$. 
Note that for any two pseudo-differential operators $P,Q$, 
$\res_z P (e^{xz}) \, Q (e^{-xz}) \, dz  = \res \,  P \circ Q^*$.
Taking $P=\p^i \circ \Phi_1$ and $Q=\Phi_2$ in this identity we find 
$\res \,  \p^i \circ \Phi_1 \circ \Phi_2^* \= \res_z \, \p^i \circ \Phi_1 (e^{xz}) \, \Phi_2 (e^{-xz}) \, dz $.
Therefore, showing $\Phi_1\circ \Phi_2^*\equiv1$ is further equivalent to showing that for all $i\geq 0$, 
\beq
\res_z \, \p^i\bigl(\psi(z,x)\bigr) \, \psi^*(z,x) \, dz \= 0\,. 
\eeq

Before continuing the proof let us do preparations. 
Following~\cite{BY, Buryak1} (see also~\cite{KS}), 
introduce the following linear operators~$S_z$ and~$S_z^*$:
\begin{align}
& S_z \:= \frac{C}{rz^{r-1}} \p_z \,-\, \frac{r-1}{2  r} C  z^{-r} \,-\,  z \= \frac{C}r z^{-\frac{r-1}2} \circ \p_z\circ z^{-\frac{r-1}2} \,-\,z \,,  \\
& S_z^* \:= - \frac{C}{rz^{r-1}} \p_z \+ \frac{r-1}{2  r}  C  z^{-r} \,-\,  z  \= \frac{C}r z^{-\frac{r-1}2} \circ (-\p_z)\circ z^{-\frac{r-1}2} \,-\,z \,. 
\end{align}
Then we have the following lemma.
\begin{lemma} \label{lemmapxi} 
For any~$i\geq 0$, we have
\begin{align}
& \p^i \bigl(\psi(x,z)\bigr) \= \bigl(-S_z\bigr)^i \, \bigl(\psi (x,z)\bigr) \,, \qquad \p^i \bigl(\psi^*(x,z) \bigr) \= \bigl(S_z^*\bigr)^i \, \bigl(\psi^*(x,z) \bigr) \,. 
\end{align}
\end{lemma}
\begin{proof}  
By straightforward calculations using the definitions~\eqref{bispsia}, \eqref{bispsistara}.
\end{proof}
As in Appendix~\ref{wavebi}, define 
\[c(z)\:=\psi(0,z)\,, \qquad c^*(z)\:=\psi^*(0,z)\,.\]
We then further have the following lemma.
\begin{lemma}\label{psipsistarexpression} 
The $\psi$ and~$\psi^*$ defined by~\eqref{bispsia} and~\eqref{bispsistara} have the expressions:
\begin{align}
& \psi(x,z) \= \sum_{i\geq0} 
\frac{(-1)^i}{i!}  S_z^{\,i} \bigl(c(z)\bigr) \, x^i \= \sum_{i\geq0} \, \frac{(xz)^i}{i!}  f_i \bigl(\tfrac{C/r}{z^{r+1}}\bigr)  \,,  \label{951}\\
& \psi^* (x,z) \= 
\sum_{i\geq0} \frac{1}{i!} \, (S_z^*)^i \bigl(c^*(z)\bigr) \, x^i \= \sum_{i\geq0} \, (-1)^i \frac{(xz)^i}{i!} f_i\bigl(\tfrac{-C/r}{z^{r+1}}\bigr)  \label{952} \, ,
\end{align}
where we recall that $f_i$ are given by~\eqref{fjT}.
\end{lemma}
\begin{proof}
Performing the Taylor expansion of~$\psi$ with respect to~$x$ at $x=0$ and 
using Lemma~\ref{lemmapxi} we immediately get the first equality of~\eqref{951}. By using~\eqref{P1a} we find that 
$c(z)$ has the form
\beq
c(z) \= \sum_{m\geq 0} C^{\frac{(r+1)m}{r}} \frac{c_m}{z^{(r+1)m}}\,,\quad c_0\=1\,.
\eeq
Define 
\[\tilde f_i \:= z^{-i} \, (-S_z)^i(c(z)) \,.\]
Then by using the definition of~$S_z$ we observe that $\tilde f_i \in \CC[[1/z^{r+1}]]$. Write
\[\tilde f_i \= \tilde f_i(T) \,, \quad  T\= \frac{C}{r \, z^{r+1}} \,.\] 
Using again the definition of~$S_z$ as well as~\eqref{equationpsi}, we find that 
\begin{align}
& \tilde f_{j+1}(T) \= \biggl(1\+ \Bigl(\frac{r-1}2-j\Bigr)\,T \+  (r+1)\,T^2\frac d{dT} \biggr)\, \tilde f_j(T)\,, \nn\\
& \tilde f_{j+r}(T) \= \tilde f_j(T) \,-\, r \, j \, T \, \tilde f_{j-1}(T)\,. \nn
\end{align}
Comparing these with~\eqref{fid1}--\eqref{fid2} and using the uniqueness of the recursion we conclude that $f_i=\tilde f_i$. This proves~\eqref{951}. 
The proof of~\eqref{952} is similar. 
The lemma is proved.
\end{proof}

\noindent{\it End of the proof of Proposition~\ref{pairprop}.} 
Using the above Lemma~\ref{psipsistarexpression} and~\eqref{fifjTmT}, we have
\begin{align}
&   \res_z \p_x^i\bigl(\psi(z,x)\bigr) \, \psi^*(z,x) \, dz   \nn\\
 \=  & \res_z \sum_{m\ge0} \frac{z^{m+i} x^{m}}{m!} f_{m+i}\bigl(\tfrac{C/r}{z^{r+1}}\bigr) 
\sum_{\ell\ge0} (-1)^\ell \frac{(xz)^\ell}{\ell!} f_\ell\bigl(\tfrac{-C/r}{z^{r+1}})  \, dz \nn\\
 \= & \res_z \sum_{m,\ell,q \geq 0}  (-1)^\ell \, \frac{z^{m+\ell+i-q(r+1)} \, x^{m+\ell}}{m! \, \ell!}  
\Bigl(1+\frac{q-i-m-\ell-1}{r}\Bigr)_q \, \widetilde C_q(r,m+i,\ell)   \Bigl(\frac r2\Bigr)^q\, dz \= 0 \, .  \nn
\end{align}
The proposition is proved. 
\end{proof}

\begin{cor}
Let $L= \p^r + C x$ and $z_k(x)$ be defined by~\eqref{Lzk}. We have
\beq\label{hk0tcformula}
z_k(0) \= 
\Bigg\{ \begin{array}{cc} (-1)^{(r+1)n}\bigl(1+\tfrac{n-1}r\bigr)_n\,\tC_n(r,0,0)\, \frac{C^n}{2^n}\,,  & {\rm if} ~ \, k=-1+(r+1)n ~ \, {\rm with} ~ \, n\geq 0\,, \\ 
0\,, & {\rm otherwise.} \end{array}
\eeq 
\end{cor}
\begin{proof}
According to~\eqref{formulaHz} and~\eqref{defccH}, and using Lemma~\ref{psipsistarexpression}, we have 
\beq
 \sum_{k\geq -1} (-1)^{k+1} \, z_{k} (0)  \, z^{-(k+1)} 
\= f_0 \bigl(\tfrac{C/r}{z^{r+1}}\bigr) \, f_0 \bigl(-\tfrac{C/r}{z^{r+1}}\bigr)\,.
\eeq
Here $z_{-1}(0):=1$.
Substituting~\eqref{fifjTmT} in this equality we find 
\beq
\sum_{n\ge0} \bigl(1+\tfrac{n-1}r\bigr)_n \, \frac{\tC_n(r,0,0)\, C^n}{2^nz^{n(r+1)}} \= 
1 \+ \sum_{k\geq 0} (-1)^{k-1} \, z_{k} (0)  \, z^{-(k+1)}  \,.
\eeq
The corollary is proved by comparing the coefficients of powers of~$z^{-1}$.
\end{proof}

\section{The $A$ series}\label{sectionaseries}
In this section, we will prove Theorem~\ref{algebraicthma}, which generalizes all parts of Theorem~\ref{mainr5} 
(except for (i) and~(ii), which were generalized to all~$\g$, not just~$A_l$, in Section~\ref{dualtopoodesection}) 
from~$A_4$ to~$A_{r-1}$ for all~$r$. 
We will also prove the integrality result for $b_g$~and~$c_g$ in Theorem~\ref{thmintegrality}, which we have obtained for $r=5$ only.
The odd-looking numbering is meant to correspond to the various parts of Theorem~\ref{mainr5}.

\begin{theorem}\label{algebraicthma} 
Let $m:= \bigl[(2g-1)/r\bigr]$. The following statements are true:

\smallskip

\noindent {\rm (iii)} {\rm [algebraicity]}
Define $\tilde a_g(r)$ from the generating function 
\beq\label{algbraicacase}
y(x) \= \sum_{g \geq 0} \, \tilde{a}_g(r) \, (2x)^{2g} \= 1 \,-\, \frac{r-1}6 \, x^2  \+ \frac{(r-1)(r-3)(2r+1)}{360} \, x^4   \+ \cdots \,, 
\eeq 
where $y(x)$ is the unique solution in~$1+x^2\QQ[r][[x^2]]$ of the polynomial equation
\beq\label{algbraicacaset}
  \frac{(y+x)^{r+1} \,-\, (y-x)^{r+1}}{2\,x} \,-\,r\,-\,1 \= 0\,.
\eeq
Then for all $g\geq1$ with $2g-1\not\equiv 0 \,({\rm mod}\,r)$,
\beq\label{tauagag}
\tau_{A_{r-1}}(g) \= \frac{(-1)^{m+g} \, \, r^{1-g}}{\bigl(\bigl\{\frac{2g-1}r\bigr\}\bigr)_m} \, \tilde{a}_g(r)  \,.
\eeq

\smallskip
\noindent {\rm (iv)} {\rm [closed formula]} 
Define $c_{p,j}(r)$ as in~\eqref{cpjgeneratingdef}. Then for all $g\geq 0$, we have
\begin{align}
& \tau_{A_{r-1}}(g) \= 
 \frac{(-1)^{g+m-1} \, r^{-g}}{\bigl(\bigl\{\frac{2g-1}r\bigr\}\bigr)_{2g+m+1}} \, \sum_{p=0}^{2g} \, \biggl(\frac{2g(r+1)-1}r\biggr)^{-}_{p+2g} \, c_{p,2g}(r) \,. \label{ci2g1}
\end{align}

\smallskip
\noindent {\rm (v)} {\rm [product formula]} 
Define $C_n(r,j)\in \QQ[r,j]$ by~\eqref{wudef}--\eqref{defc} 
and $f_j(T)$ by~\eqref{fjT}, i.e., $f_j(T) = \sum_{k\ge0} (2k+1)!!\,C_{2k}(r,j) \,(-T)^k$. 
Then the following identity holds true:
\beq
f_0(T) \, f_0(-T) \=  \sum_{g\geq0} \, \Bigl(1 - \frac{1-2g}{r}\Bigr)_{2g} \, (1-2g) \, r^{2g}  \,  \tilde{a}_g(r) \, T^{2g} \,.
\eeq

\smallskip
\noindent {\rm (vi)} {\rm [finite hypergeometric sum]}  
Set $R=[r/2]$.
For all $g\geq 1$ with $2g-1\not\equiv 0 \,({\rm mod}\,r)$, we have
\beq\label{tauge2g-aa}
\tau_{A_{r-1}}(g) \= \frac{(-1)^{g+m}}{\bigl(\bigl\{\frac{2g-1}r\bigr\}\bigr)_m} \, \frac{r^{1-g}}{1-2g} \, \sum_{d\geq0} \, \binom{\frac{2g-1}{r}}{d} \, K_{g,d} \,,
\eeq
where
\beq
K_{g,d} \= \frac1{4^g\,(r+1)^d} \,
\sum_{m_1+m_2+\dots +m_R=d \atop m_1+2m_2+\dots+R \, m_R=g} \, \binom{d}{m_1,\dots,m_R} \,
\prod_{i=1}^R \binom{r+1}{2i+1}^{m_i} \,.
\eeq

\smallskip
\noindent {\rm (vii)} {\rm 
[asymptotics]\footnote{The precise form of part~(vii) was written by D.Y. and~D.Z. after the first author~B.D. passed away.}}
For $r\ge3$, 
as $g\to\infty$ with $2g-1\not\equiv 0 \,({\rm mod}\,r)$, we have
\beq\label{rspinasymptotics}
\tau_{A_{r-1}}(g) \; \sim \; \frac{r \, \sqrt{\pi}}{\sqrt{r+1}^{\frac{r-2}r}} \;
\frac{1}{\Gamma\bigl(1- \bigl\{\frac{2g-1}r\bigr\}\bigr)} \, \frac{1}{ \Gamma(\frac{2g-1}{r}) } \,  g^{-\frac32} \,
\Bigl(4 \, r \, (r+1)^{\frac2r}\,  \sin^2\bigl(\tfrac{\pi}r\bigr)\Bigr)^{-g}\,.
\eeq
\nopagebreak   For $r=2$ (the $A_1$ case) the asymptotic formula is the same with an extra factor of~$1/2$.
\end{theorem} 

\begin{proof} 
Let us start from recalling the $r$-spin Witten conjecture~\cite{Witten2} (the Faber--Shadrin--Zvonkine theorem~\cite{FSZ}). 
Introduce the Lax operator
\beq\label{laxrspin}
L\=  \p^r \+ \sum_{\alpha=1}^{r-1} u_{\alpha} \, \p^{r-1-\alpha}\,.
\eeq
The Gelfand--Dickey hierarchy~\cite{Dickey} is the following infinite family of PDEs for $(r-1)$ unknown functions $u_1,\dots,u_{r-1}$:
\beq\label{GDhierarchy}
\frac{\p L}{\p s^{\alpha,q}} \= \Bigl[ \bigl(L^{k/r}\bigr)_+ \,, \,L \Bigr] \,, \quad q\ge 0\,,
\eeq
where $k=\alpha+rq$.
Similarly as before, denote ${\bf s}:=(s^{\alpha,q})_{\alpha=1,\dots,r-1, \, q\geq 0}$.  
Since 
\[\frac{\p u_\alpha}{\p s^{1,0}} \= \frac{\p u_\alpha}{\p x}  \,, \]
we identify~$s^{1,0}$ with~$x$. 
Denote by~$Z(\bt)$ the $r$-spin partition function defined by~\eqref{defpartitionfunction} 
with~$\Omega$ being the Witten class. 
Then the $r$-spin Witten conjecture~\cite{Witten2} 
is stated as follows:
 the partition function~$Z=Z(\bt({\bf s}))$ is a particular tau-function for the 
Gelfand--Dickey hierarchy~\eqref{GDhierarchy}, where
 \beq
 t^{\alpha,q} \= (-1)^{q+1}\,  (\sqrt{-r})^{\frac{3k}{r+1}+1} \, \Bigl(\frac{\alpha}r\Bigr)_{q+1} \, s^{\alpha,q} \,, \quad q\geq 0\,;
 \eeq
moreover, the partition function~$Z$ satisfies the string equation~\eqref{stringgeneralZFJRW} with 
$\eta_{\alpha\beta}=\delta_{\alpha+\beta,r}$. 
Note that the string equation~\eqref{stringgeneralZFJRW} written in terms of the time variables~$s^{\alpha,q}$ reads
\beq\label{stringZrspin}
\sum_{\alpha=1}^l \sum_{q\geq 0} \, (\alpha+rq+r) \, s^{\alpha,q+1} \,
\frac{\p Z}{\p s^{\alpha,q}} \+ \frac12 \, \sum_{\alpha=1}^l  \alpha \, (r-\alpha) \, s^{\alpha,0}  s^{l+1-\alpha,0} \, Z \= \frac{\p Z}{\p s^{1,0}}\,.
\eeq

It is known that the differential polynomials $h_k$ defined by 
\beq\label{Wtauham}
h_k \:=  \frac{(-1)^{q+1}}{(\sqrt{-r})^{3 \frac{1+k}{r+1}} (\frac{\alpha}{r})_{q+1} } \, {\rm res} \, L^{\frac{k}r}
\qquad \bigl(k \in \ZZ_{>0} \smallsetminus  r \ZZ_{>0}\bigr)
\eeq
are the tau-symmetric Hamiltonian densities
for the Gelfand--Dickey hierarchy (see~\cite{Dickey,Witten2}; also cf.~\cite{BY, BCT, BPS, DZ-norm}). Here $k=\alpha+rq$.
Denote by $u_\alpha({\bf s})$ the solution 
corresponding to the particular tau-function~$Z(\bt({\bf s}))$. 
The Lax operator~$L$ given by~\eqref{laxrspin} is now also subjected to this solution.
We have
\beq\label{conjwthmfsz}
h_k({\bf s}({\bf t})) \= \frac{\p^2 \log Z(\bt)}{\p t^{\alpha,q} \p t^{1,0}} \,.
\eeq
Then using the string equation together with a Miura-type transformation 
 we find that the initial Lax operator $L|_{s^{\alpha,q}=x\,\delta^{\alpha,1}\delta^{q,0}}$, still denoted by~$L$, 
  has the following explicit expression~\cite{BY}:
\beq
L \= \p^r \+ r \, x\,.
\eeq

Now by using the above~\eqref{stringrec}, \eqref{Wtauham}, \eqref{conjwthmfsz} we have
\beq\label{taughap}
\tau_{A_{r-1}}(g) \= (-1)^{q+g} \, \frac{z_{2g(r+1)-1}(0)}{r^{3g} \, (\frac{\alpha}{r})_{q+2}} \,, \quad \forall\, g\geq 0\,, 
\eeq
where $\alpha\in\{1,\dots,l\}$, $q\geq 0$ are the unique integers satisfying~\eqref{rgaq}, and we recall that 
$z_k(x):=\res L^{k/r}$. Therefore, by using~\eqref{hk0tcformula}, Proposition~\ref{propequivanlencebetweenalgprod}, 
and the corollary to Proposition~\ref{Don'sformula} from the previous section,  
we arrive at parts (iii)--(v). We remark in passing 
that equation~\eqref{algbraicacaset} multiplied out, takes the form 
\beq\label{equivpolya}
\sum_{i=0}^{r/2} \,\binom{r+1}{2i+1} \, x^{2i} \, y^{r-2i} \, - \, r -1  \= 0 \,,
\eeq
which for $r=5$ agrees with the equation in~\eqref{quinticeqn} (part (iii) of Theorem~\ref{mainr5}) if we set $z=-\,20\,x^2$.

To show~\eqref{tauge2g-aa}, we will prove a more general statement (suitable for the later as well). Denote 
\[ \lambda(p;s) \=  p^{r} \+ \sum_{i=1}^{r-1} s_i \, p^{r-1-i}\,,\] where 
$s=(s_1,\dots,s_l)$. Consider the following algebraic equation for~$p$:
\begin{align}
& \lambda(p;s) \= \xi^r \,. \label{inversefunctionapp}
\end{align}
This equation has a unique formal solution~$p(\xi;s)$ in $\xi+\QQ[s][[\xi^{-1}]]$. Write 
\beq
p(\xi;s) \= \xi \+ \sum_{k\geq 1} u_k(s) \, \xi^{-k}\label{defintionQstargen}
\eeq
with $u_k(s)\in \QQ[s]$, $k\geq 1$. 
Differentiating~\eqref{defintionQstargen} with respect to~$\xi$, we find 
\begin{align}
\frac{dp(\xi;s)}{d\xi}
 \= 1  \,-\,  \sum_{k\geq 1} k \, u_k(s) \, \xi^{-k-1} \,. \label{Qstarform}
\end{align}
Therefore, 
\begin{align}
k \, u_k(s) \= \res_{\xi=\infty} \frac{dp(\xi;s)}{d\xi} \xi^{k} d\xi \= 
\res_{\xi=\infty} \rho^k dp(\xi;s) \= \res_{p=\infty} f(p;s)^{\frac{k}{r}} dp \,.
\end{align}
Noting that 
\begin{align}
&\biggl( p^{r} \+ \sum_{i=1}^{r-1} s_i \, p^{r-1-i}  \biggr)^{\frac {k} r}
\= p^{k} \sum_{j\ge0}   \binom{\frac {k} r}{j} 
\sum_{i_1,\dots,i_j=1}^{r-1} s_{i_1}\dots s_{i_j} \, p^{-j-i_1-\dots-i_j} \,,
\end{align}
we obtain the following expressions for~$u_k(s)$:
\begin{align}
k \, u_k(s)
&\=  - \sum_{j\geq0} \binom{\frac {k} r}{j} 
\sum_{1\leq i_1,\dots,i_j\leq r-1\atop j+i_1+\dots+i_j=k+1 } s_{i_1}\dots s_{i_j} \label{kuk1}\\
&\=  -  \sum_{n_1,\dots,n_{r-1}\geq 0 \atop  2n_1+3n_2+\dots+rn_{r-1} =k+1 }  \binom{\frac {k} r}{\sum_i n_i}  
\binom{\sum_i n_i}{n_1,\dots,n_{r-1}}s_1^{n_1} \dots s_{r-1}^{n_{r-1}} \,. \label{kuk2}
\end{align}
Now specializing the above to $s=s^*$ given by 
\[ s^*_{2i}=0\,,\quad  s^*_{2i-1}=\frac1{4^i(r+1)}\binom{r+1}{2i+1}\,,\] 
we obtain 
\begin{align}
(1-2g) \, u_{2g-1}(s^*) \=  \sum_{d\ge0}  \frac{\binom{\frac{2g-1}r}{d}  }{4^g \, (r+1)^d} 
\sum_{m_1+ \dots+m_{R}=d \atop  m_1+2m_2+\dots+R \, m_{R} = g } 
\binom{d}{m_1,\dots,m_{R}} \,
\prod_{i=1}^R \, \binom{r+1}{2i+1}^{m_i} \,. \label{kuk2sp}
\end{align}
Here $m_i=n_{2i-1}$.
Combining with~\eqref{algbraicacaset} we obtain~\eqref{tauge2g-aa}. 

It remains to prove part~(vii).
In view of~\eqref{tauagag}, 
\eqref{rspinasymptotics} is equivalent 
to the asymptotic formula
\beq\label{agasymp}
 \tilde{a}_g(r)  \; \sim \; \frac{(r+1)^{\frac1r-\frac12}}{\sqrt{\pi g^3}} \; 
\frac{ \sin\bigl(\frac{2g-1}{r}\pi\bigr) }{\bigl(-\,4\,(r+1)^{\frac2r}\sin^2\bigl(\tfrac{\pi}r\bigr)\bigr)^g}\,,\qquad r>2
\eeq
for the coefficients $\tilde{a}_g(r)$ defined in~\eqref{algbraicacase}.  By a standard principle, this 
is in turn equivalent to studying the asymptotic properties (to lowest order) of the generating function $y(x)$ near 
its singularities of smallest absolute value. Let $P=P(x,y)$ be the polynomial on the left-hand side of~\eqref{algbraicacaset}
and denote by $P_x$ and $P_y$ its partial derivatives with respect to~$x$ and~$y$, respectively.  The singularities 
of $y(x)$ are located at the points where its graph becomes vertical, i.e., where both $P$ and~$P_y$ vanish (and where
$P_x$ doesn't vanish, but one sees easily that these three polynomials have no common zeros). Since 
$P_y/(r+1)=((y+x)^r-(y-x)^r)/(2x)$,  we see that the curve $\{P_y=0\}$ has $(r-1)$ components, parametrized by the non-trivial
$r$th roots of unity~$\zeta$ and given by $\frac{y+x}{y-x}=\zeta$.  Setting $(x,y) = c\bigl(\frac{\zeta-1}2,\frac{\zeta+1}2\bigr)$ 
with $c\in\mathbb C$, and substituting this into the equation $P(x,y)=0$, we find that $c^r=r+1$. Thus, setting $c_0=(r+1)^{1/r}$, 
we have that the $r(r-1)$ common points of $P=0$ and $P_y=0$ are given parametrically by 
$(x,y) = c_0\bigl(\frac{\alpha-\beta}2,\frac{\alpha+\beta}2\bigr)$ with $(\alpha,\beta)$ ranging over all pairs of distinct $r$th 
roots of unity. Of these, the ones with $x$ nearest to the origin are those with $\alpha/\beta=\zeta_r^{\pm1}$, where $\zeta_r=e^{2\pi i/r}$,
i.e., they are the $2r$ points $(\pm x_j,y_j)$,  where $j\in\mathbb Z/r\mathbb Z$ and  
$(x_j,y_j):=\frac{c_0}2\bigl(\zeta_r^j-\zeta_r^{j-1},\zeta_r^j+\zeta_r^{j-1}\bigr)$, all of them with $|x_j|=c_0\sin(\pi/r)$.

To complete the proof, we must (a) compute the potential contribution from each nearest singularity and (b) see that only the four
singularities $(\pm x_0,y_0)$ and~$(\pm x_1,y_1)$ belong to the subset of $\{P=0\}$ that is parametrized by $x\mapsto(x,y(x))$ in 
the closed disk $|x|\le c_0\sin(\pi/r)$.  For~(a), we note that from the Taylor expansion of~$P$ near the singularity $(x_j,y_j)$ we 
find that for a point $(x,y)=(x_j-\varepsilon,y_j+\delta)$ near $(x_j,y_j)$ lying on the curve $P=0$ we must have 
$P_x(x_j,y_j) \, \varepsilon\,\approx\,P_{yy}(x_j,y_j) \, \delta^2/2$ or
$y\,\approx\,y_j + C_j\,\sqrt{1-x/x_j}$ with $C_j= \pm \sqrt{(2\,x\,P_x/P_{yy})(x_j,y_j)}$.  A short computation
gives that $x_jP_x(x_j,y_j)=r\,(r+1)$ and $P_{yy}(x_j,y_j)=-\,\zeta_r^{1-2j}\,r\,(r+1)^{2-2/r}$, so  
$C_j= \pm \,i\,\sqrt2\,\zeta_r^{j-\frac12} \, (r+1)^{\frac1r-\frac12}$. The contribution from this singularity to the coefficient 
$2^{2g}\tilde a_g(r)$ of~$x^{2g}$ in~$y(x)$ is then asymptotically equal to $C_j \, \binom{1/2}{2g}\,x_j^{-2g}$, and of
course the contribution from~$(-x_j,y_j)$ is the same since $P$ is even in~$X$.  Since $\binom{1/2}{2g}$ is asymptotically 
equal to $-1/\sqrt{32\pi g^3}$, we find that the sum of the four contributions coming from $(\pm x_0,y_0)$ and $(\pm x_1,y_1)$ 
indeed gives the asymptotic formula stated in~\eqref{agasymp}.  
For~(b), we first note that since $P$ is a polynomial in~$x^2$ and~$y$ having degree~$r$ in~$y$, the map $\pi:(x,y)\mapsto X=x^2$ 
represents the Riemann surface $\{P=0\}$ as an $r$-sheeted branched cover of $\mathbb P^1(\CC)$. By the above calculations, this map
is unramified over the open disc $D=\{|X|<|x_j^2|=c_0^2\sin^2(\pi/r)\}$, but has $r$ ramification points $(X_j=x_j^2,y_j)$, at each
of which exactly two of the $r$~sheets come together, over the closure $\overline D$. The inverse image of the disk $D$ consists
of $r$ disjoint disks $D_\zeta$, indexed by the $r$th roots of unity~$\zeta$, where $D_\zeta$ is the component of $\pi^{-1}(D)$
containing~$(0,\zeta)$ and is parametrized by $y =\zeta y(x/\zeta)$ with $x\in D$.  If we write $D_j$ ($j\in\ZZ/r\ZZ$) for $D_{\zeta_r^j}$,
then it is not hard to see that $\overline D_j$ meets $\overline D_{j-1}$ at the point $(X_j,y_j)$ and, of course (replacing $j$ by~$j+1$) 
also meets $\overline D_{j+1}$ at~$(X_{j+1},y_{j+1})$, and that the closed disks $\overline D_j$ have no other points in common.  
In particular, the component $D_0$ parametrized by $y=y(x)$ contains the two ramification points $(X_0,y_0)$ and $(X_1,y_1)$
in its closure, and no others, and this means that the two points nearest to the origin where the series $y(\sqrt X)$ is not analytic 
are $X_0$ and~$X_1$, as claimed. 
For $r=2$, this proof of the asymptotics also applies, 
but the situation on how the various branches 
over the closed disk~$\overline D$ meet at their boundaries is 
slightly different, and we get the extra factor~$1/2$. This can also 
  be seen directly, since $y=\sqrt{1-x^2/3}$ in this case, but in any case 
there is no need to do any of this since the well-known formula $\tau_{A_1}(g) = \frac1{24^g\,g!}$ 
immediately gives the asymptotics.

This completes the proof of all parts of Theorem~\ref{algebraicthma}.
\end{proof}

\begin{remark}\label{BrezinHikami}
A different approach for computing the $r$-spin intersection numbers, using the theory of matrix models,
was obtained by Br\'ezin--Hikami~\cite{BH1,BH2}.  For the one-point numbers~$\tau_{A_{r-1}}(g)$, 
 Br\'ezin--Hikami discovered the following integral formula:
\beq\label{BHintegralformula}
\tau_{A_{r-1}}(g) \=  \frac{(-1)^g\, r^{1-g}}{1-\frac\alpha r} \, \bigl[t^{2g}\bigr] 
\, \int_0^\infty \, 
\exp\biggl(-\frac{(s+t/2)^{r+1}-(s-t/2)^{r+1}}{(r+1) \, t} \biggr) \, ds \,,
\eeq
where $g\geq0$ with $2g-1\equiv \alpha \, ({\rm mod} \, r)$, $\alpha=1,\dots,{r-1}$.
We note that formula~\eqref{tauge2g-aa} of part~(vi) of Theorem~\ref{algebraicthma} 
can also be proved by using~\eqref{BHintegralformula}.
Indeed, expanding the integrand suitably and integrating term-by-term, 
one can obtain that 
\beq\label{tauge2g-aagamma}
\tau_{A_{r-1}}(g) \= \frac{(-1)^{g} \, r^{-g}}{4^g \, \Gamma\bigl( 1- \bigl\{\frac{2g-1}r\bigr\}\bigr)} 
 \, \sum_{d\geq0} \, (-1)^d \, \frac{\Gamma(d-\frac{2g-1}r)}{(r+1)^d} 
 \sum_{m_1+m_2+\dots +m_R=d \atop m_1+2m_2+\dots+R \, m_R=g} \,
\prod_{i=1}^R \frac{\binom{r+1}{2i+1}^{m_i}}{m_i!} \,,
\eeq
which by Euler's formula is equivalent to~\eqref{tauge2g-aa}. 
Liu--Vakil--Xu~\cite{LVX} obtained formula~\eqref{tauge2g-aagamma} 
by using the integral formula~\eqref{BHintegralformula}; 
this was also our original derivation of~\eqref{tauge2g-aa}. 
Therefore, part~(vi) of Theorem~\ref{algebraicthma} is not new; however, 
 our current proof for~\eqref{tauge2g-aa} using the wave-function-pair approach from the theory of integrable systems 
 is a self-contained one. 
We also note that the integral formula~\eqref{BHintegralformula} 
in principle could also be used 
together with the method of steepest descent to give 
a different proof for part~(vii) of Theorem~\ref{algebraicthma}.
(The recursion given by the dual topological ODE, cf.~Theorem~\ref{odethmg}, 
could give a third proof of part~(vii) of Theorem~\ref{algebraicthma} with 
the constant term in the right-hand side of~\eqref{agasymp} undetermined.)
\end{remark}

By using~\eqref{algbraicacase}--\eqref{tauagag} one immediately gets the first few values of~$\tau_{A_{r-1}}(g)$:
\begin{align}
& \tau_{A_{r-1}}(1) \= \frac{r-1}{24} \,, \quad r\geq 2\,,\\
& \tau_{A_{r-1}}(2) \= \left\{\begin{array}{cl}
1/1152\,, &  r=2\,, \\
\frac{(r-3) (r-1) (2r+1)}{5760 \, r} \,, &  r\geq 3\,, \\
\end{array}\right. \\
& \tau_{A_{r-1}}(3) \= \left\{\begin{array}{cl}
1/82944\,, &  r=2\,, \\
1/31104\,, &  r=3\,, \\
3/20480\,, &  r=4\,, \\
\frac{(r-5) (r-1) (2r+1) (8 r^2 - 13 r -13 )}{2903040 \, r^2} \,, & r\geq 5\,.\\
\end{array}\right.  
\end{align}

\begin{cor}  
For $r\ge 2g$, the value of~$\tau_{A_{r-1}}(g)$ is a Laurent polynomial in~$r$. Moreover,  the value of this Laurent polynomial at $r=-1$ is 
equal to $\frac{B_{2g}}{2g}$,
where $B_n$ is the $n$th Bernoulli number.
\end{cor}
\begin{proof}
If $r\ge 2g$, then $m=0$, so the right-hand side of~\eqref{tauagag} reduces to $-\tilde a_g(r)/(-r)^{g-1}$. The first statement then 
follows from the fact that $\tilde a_g(r)\in\QQ[r]$. 
Now by taking the $r\to -1$ limit in~\eqref{algbraicacaset} we find the unique solution $y=\frac{x}{\tanh{x}}$. 
The second statement then follows. 
\end{proof}
\noindent According to Witten~\cite{Witten1992, Witten2} the second statement in the corollary 
should give a new proof of the Harer--Zagier formula~\cite{HZ, Penner} on the 
orbifold Euler characteristic of~$\mathcal{M}_{g,1}$ (see also~\cite{BH2,LVX}).

\smallskip

We now give the proof of the three integrality statements for~$\tau_{A_4}(g)$ stated in the introduction.

\begin{proof}[Proof of Theorem~\ref{thmintegrality}]
The algebraicity of the generating function of the numbers~$a_g$ as given in part~(iii) of Theorem~\ref{mainr5} 
implies their integrality away from the primes 2, 3, and~5.
To prove the integrality of~$b_g$ and~$c_g$, 
we use~\eqref{BH-hyper}.
It follows from~\eqref{BH-hyper} that
\begin{align}
& c_{g}  \=  
6^{-g} \, \sum_{0\leq s \leq g/2} \, 2^{-2s} \, (-9)^s \, c_{g}^{[s]}  \,, 
\end{align}
where the numbers $c_{g}^{[s]}$ are given according to the value of $g\,({\rm mod}\,5)$ by
\begin{align}
c_{5n}^{[s]}&:= \; \frac{5^{-2s} (\frac35)_n (\frac45)_n (\frac15)_{3n-s}}{s! \, (5n-2s)!} \,,
&c_{5n-1}^{[s]}&:= \; \frac{5^{-2s} (\frac25)_n (\frac45)_n (\frac35)_{3n-1-s}}{s! \, (5n-1-2s)!} \,, \nn \\
c_{5n-3}^{[s]}&:= \; \frac{5^{-2s} (\frac35)_{n-1} (\frac15)_n (\frac25)_{3n-2-s}}{s! \, (5n-3-2s)!} \,, 
&c_{5n-4}^{[s]}&:= \; \frac{5^{-2s} (\frac15)_n (\frac25)_{n-1} (\frac45)_{3n-3-s}}{s! \, (5n-4-2s)!} \,. \nn
\end{align}
To show that $c_{g}$ belongs to $\ZZ\bigl[1/30\bigr]$ we will actually show the stronger statement that 
each $c_{g}^{[s]}$ belongs to $\ZZ[1/5]$.  

Consider first the case $g=5n$. For each prime $p\neq 5$, the $p$-adic valuation of~$c_{5n}^{[s]}$ is given by
\[\nu_p\bigl(c_{5n}^{[s]}\bigr)  \=  \nu_p \bigl(\bigl(\tfrac35\bigr)_n\bigr) 
\+ \nu_p \bigl(\bigl(\tfrac45\bigr)_n\bigr) \+ \nu_p \bigl(\bigl(\tfrac15\bigr)_{3n-s}\bigr) \,-\, \nu_p(s!) \,-\, \nu_p((5n-2s)!) \,. \]
Namely,
\[\nu_p\bigl(c_{5n}^{[s]}\bigr)  \=  \sum_{k\geq1} \Bigl[
u \bigl(\bigl(\tfrac35\bigr)_n,p^k\bigr) 
\+ u \bigl(\bigl(\tfrac45\bigr)_n,p^k\bigr) \+ u \bigl(\bigl(\tfrac15\bigr)_{3n-s},p^k\bigr) \,-\, u\bigl(s!,p^k\bigr) \,-\, u\bigl((5n-2s)!,p^k\bigr) \Big] \,, \]
where $u\bigl((a/5)_n, p^k\bigr)$ denotes the number of elements in $\{a,a+5,\dots,a+5n-5\}$ that 
are divisible by~$p^k$ ($a=1,\dots,5$, $k\ge1$). Counting these numbers we find
\begin{align}
& u \bigl(\bigl(\tfrac35\bigr)_n, p^k \bigr) 
\+ u \bigl(\bigl(\tfrac45\bigr)_n, p^k \bigr) \+ u \bigl(\bigl(\tfrac15\bigr)_{3n-s},p^k\bigr) \,-\, u\bigl(s!,p^k\bigr) \,-\, u\bigl((5n-2s)!,p^k\bigr) \nn\\
& \qquad\qquad\qquad \= f_j\bigl(\tfrac{n}{p^k}, \tfrac{s}{p^k}\bigr) \quad \mbox{if }  p^k \equiv j \, ({\rm mod} \, 5), ~ 1\le j \le 4\,,
\end{align}
where the functions $f_j: \RR^2 \to \ZZ$ are defined by
\begin{align}
& f_1(x,y) \:= [x+\tfrac25] \+ [x+\tfrac15] \+  [3x-y+\tfrac45] \,-\, [y] \,-\, [5x-2y]\,,\\
& f_2(x,y) \:= [x+\tfrac15] \+ [x+\tfrac35] \+ [3x-y+\tfrac25] \,-\, [y] \,-\, [5x-2y]\,,\\
& f_3(x,y) \:= [x+\tfrac45] \+ [x+\tfrac25] \+ [3x-y+\tfrac35] \,-\, [y] \,-\, [5x-2y]\,,\\
& f_4(x,y) \:= [x+\tfrac35] \+ [x+\tfrac45] \+ [3x-y+\tfrac15] \,-\, [y] \,-\, [5x-2y]\,.
\end{align}
Therefore it suffices to show that 
each of $f_1,f_2,f_3,f_4$ is nowhere negative on~$\mathbb{R}^2$.
To show this, we observe that each $f_\alpha(x,y)$ is periodic in both variables and is also piecewise constant, 
with jumps only along finitely many lines in the unit square $[0,1]^2$, so one only has to compute the values of $f_\alpha(x,y)$ 
in each component of the complement of the union of these lines. 
(The value of~$f_i$ at a point lying on one of these lines is
always equal to the value of~$f_i$ in one of the adjacent open regions.)
This is most easily seen graphically, as illustrated in Figure 1,
\begin{figure}[h]
\hspace{0.99cm}
\begin{subfigure}{0.40\textwidth}
\includegraphics[scale=0.52]{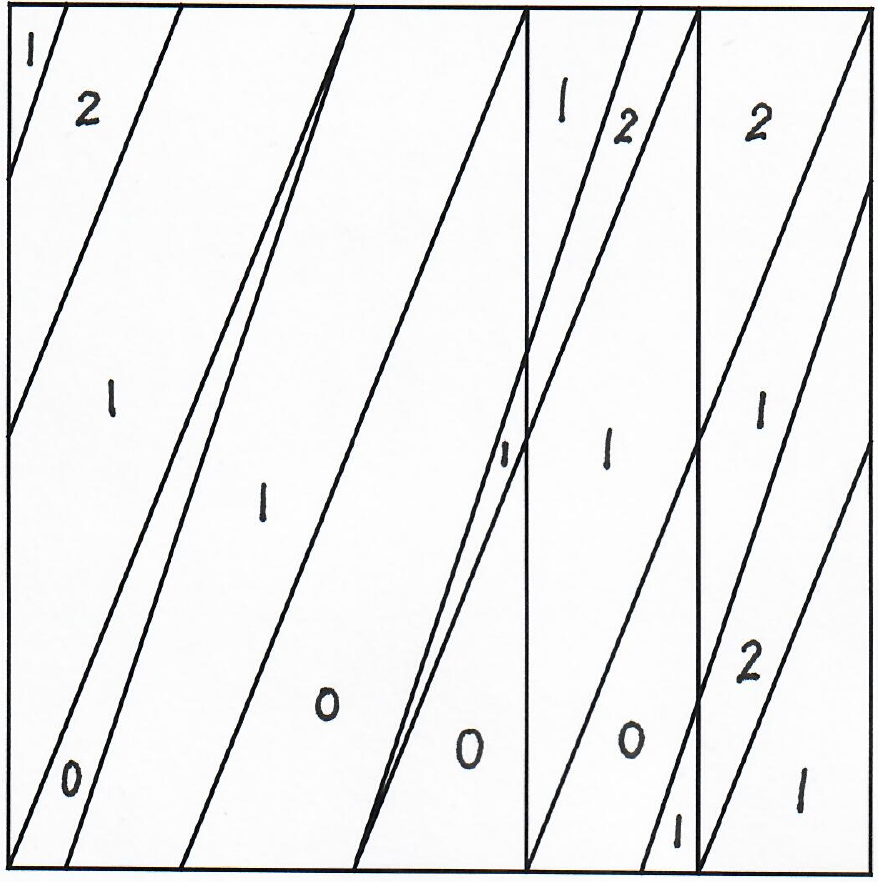} 
\end{subfigure}
\hspace{-0.5cm}
\begin{subfigure}{0.40\textwidth}
\includegraphics[scale=0.52]{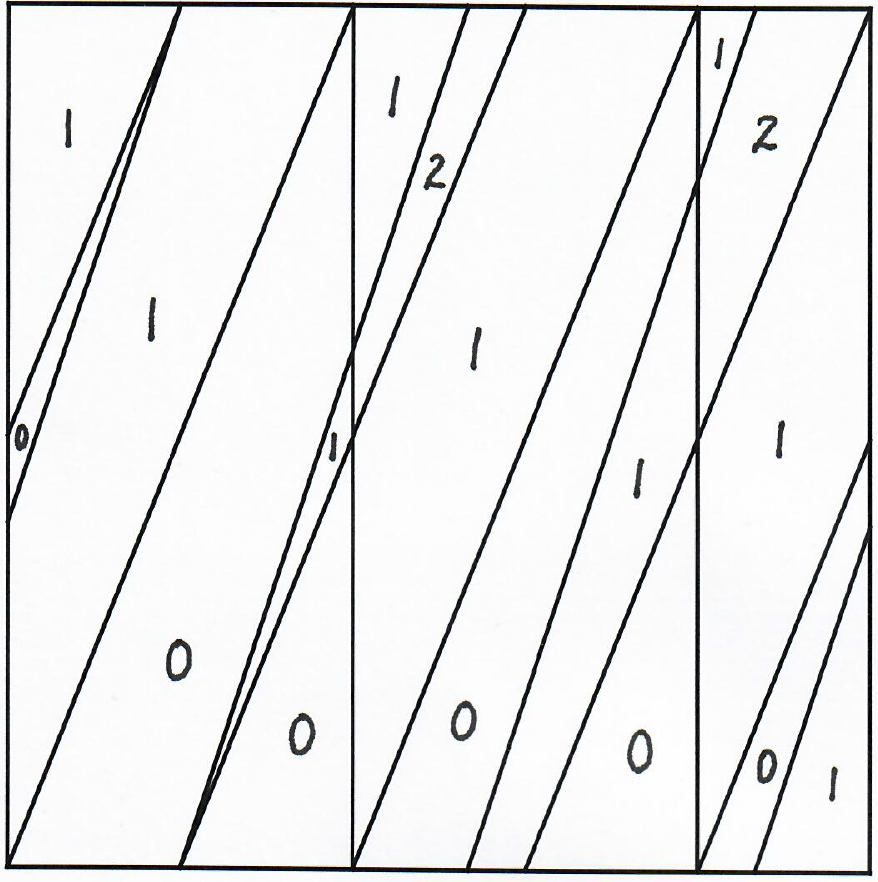}
\end{subfigure}
\caption{The values of~$f_1$ and~$f_2$ in $[0,1]^2$}
\label{figure1}
\end{figure}
which shows that each of $f_1$ and~$f_2$ assumes only the values 0, 1, and~2, 
and this also gives the non-negativity of~$f_3$ and~$f_4$ as well since $f_3(x,y) = 2-f_2(-x,-y)$ and $f_4(x,y) = 2-f_1(-x,-y)$. 

The proof for each other residue class of $g\,({\rm mod}\,5)$ is exactly similar 
and again reduces to the non-negativity of four piecewise continuous functions on $\RR^2/\ZZ^2$, but when one goes 
through the details it turns out that these four are a permutation of the same four functions $f_j$ as above in 
each case, so that we do not need to make new graphs. 
The proofs for~$b_g$ are also similar, and again a priori involve the non-negativity of 16 periodic functions on~$\RR^2$, 
which again turn out to be repetitions of only four functions:
\begin{align}
& g_1(x,y) \= [x+\tfrac15] \+ [x+\tfrac45] \+ [3x-y+\tfrac45] \,-\, [y] \,-\, [5x-2y]\,,\\
& g_2(x,y) \= [x+\tfrac35] \+ [x+\tfrac25] \+ [3x-y+\tfrac25] \,-\, [y] \,-\, [5x-2y]\,,\\
& g_3(x,y) \= [x+\tfrac25] \+ [x+\tfrac35] \+ [3x-y+\tfrac35] \,-\, [y] \,-\, [5x-2y]\,,\\
& g_4(x,y) \= [x+\tfrac45] \+ [x+\tfrac15] \+ [3x-y+\tfrac15] \,-\, [y] \,-\, [5x-2y]\,.
\end{align}
Their non-negativity can be proved as for the $f_j$.
The theorem is proved.
\end{proof}

Note that the the first part of Theorem~\ref{thmintegrality}, concerning the integrality of
the numbers~$a_g$ away from the primes 2, 3 and~5, is generalized by Theorem~\ref{algebraicthma} to the
integrality of the numbers $\tilde a_g(r)$ away from a finite number of primes depending only on~$r$
(in fact, away from $r(r+1)$) for any~$r$, as a direct consequence of the algebraicity of the power series 
$y(x)$ in~\eqref{algbraicacase}.  In fact, as was mentioned briefly in Remark~\ref{remarkintag} for 
the $r=5$ case, for any positive value of~$r$ we have the stronger statment that the numbers 
$\frac{\tilde a_g(r)}{2g-r-1}$ are integral away from a finite set of primes (in fact, 
$\frac{(r+2)\tilde a_g(r)}{2g-r-1}$ is integral away from $r(r+1)$). 
To see this, define a second algebraic power series $y_1=y_1(x)$ by
\beq
y_1 \= \frac{(y+x)^{r+1} \+ (y-x)^{r+1}}{2\,(r+1)}\,,
\eeq
with $y=y(x)$ as in~\eqref{algbraicacase}. Then it is straightforward to verify that
\beq
 x\,y_1' \,-\,(r+1)\,y_1  \= (r+1) \, x \, y' \,-\,y 
\eeq
or equivalently that the coefficient of $(2x)^{2g}$ in $y_1$ is equal to $\frac{2(r+1)g-1}{2g-r-1}\,\tilde a_g(r)$
for $2g\ne r+1$, and the divisibility statement follows. An even stronger divisibility statement, saying that 
the numbers $\tilde a_g(r)/(2g-kr-1)$ are integral away from a finite number of primes for any~$r$ and 
any fixed $k\ge0$, will be discussed in the subsequent publication~\cite{YangZagier}. 
In the case $r=5$ this follows easily from Theorem~\ref{thmintegrality}.

\section{The $D$ series}\label{sectiondseries}
In this section, we derive explicit formulas  
for the numbers $\tau_{\g}(g)$ with~$\g$ being $D_l$ ($l\geq 4$).
Denote by $r=2l-2$ the Coxeter number of~$\g$, and recall that 
$\tau_{D_l}(g) = \langle\tau_{\alpha,q}\rangle$ are the one-point FJRW 
invariants of~$D_l$-type, where $\alpha,q$ are determined by~\eqref{generalgap}.
Denote by~$Z=Z(\bt)$ the partition function for the FJRW invariants of $D_l$-type as in~\eqref{defpartitionfunction}.

\begin{theorem}\label{algebraicthmd} 
Let $l\geq 4$ be an integer and $r=2l-2$. 
Set $m:= [(2g-1)/r]$. Then the following statements are true:

\smallskip

\noindent {\rm (iii)} {\rm [algebraicity]}
Define $n_g(r)$ from the generating function 
\beq\label{algbraicacase-d}
y(t) \= \sum_{g\ge0} \, n_g(r) \, t^{2g} \= 1  \,-\, \frac{r+2}{24} \, t^2 \+ \frac{(r+2)(r-6)(2r+1)}{5760} \, t^4 \+ \cdots \,, 
\eeq 
where $y(t)$ is 
the unique solution in $1 + t^2 \, \QQ[r][[t^2]]$  
to the algebraic equation
\beq\label{algbraicacaseequation-d}
\sum_{j=0}^{r/2} \, \frac1{2j+1} \, \binom{j+r/2}{2j} \, y^{r-2j} \, t^{2j} \,-\, 1 \=  0
\eeq
or alternatively $t$ and~$y=y(t)$ are related algebraically by
\beq\label{algebraicd2}
t/y = Z-Z^{-1}\,, \quad t^r \, \frac{Z^{r+1}-Z^{-r-1}}{r+1} \= \bigl(Z-Z^{-1}\bigr)^{r+1} \,. 
\eeq
Then $n_g(r)$, $g\geq0$, are integral away from a finite number of primes, and for all $g\geq 1$ we have
\beq\label{tauagag-d}
\tau_{D_l}(g) \= \frac{(-1)^{m+g} \, r^{1-g}}{ \bigl(\bigl\{\frac{2g-1}r\bigr\}\bigr)_m } \, n_g(r)  \;. 
\eeq

\smallskip
\noindent {\rm (iv)} {\rm [closed formula]} 
For all $g\geq1$, we have
\begin{align}\label{ci2g1d}
& \tau_{D_l}(g) \= \frac{ (-1)^{g+m-1} \, r^{-g}}{\bigl(\bigl\{\frac{2g-1}r\bigr\}\bigr)_{2g+m+1}} \, \sum_{j=0}^{2g} \sum_{p=0}^{j} \, c_{p,j}(r) 
\, \biggl(\frac{2g(r+1)-1}{r}\biggr)_{2g+p}^-  \, \binom{-\frac{1}2}{2g-j}   \,.
\end{align}

\smallskip
\noindent {\rm (v)} {\rm [product formula]} 
Define $C_n(r,j)\in \QQ[r,j]$ by~\eqref{wudef}--\eqref{defc}  
and $f_j(T)$ by~\eqref{fjT}, 
i.e., $f_j(T) = \sum_{k\ge0} (2k+1)!!\,C_{2k}(r,j) \,(-T)^k$.
Then the following identity holds true:
\beq\label{f12m12thm}
\Bigl[f_{\frac12}(T)\,f_{-\frac12} (-T)\Bigr]_{\rm even}
\= \sum_{g\geq0} \, \Bigl(1-\frac{1-2g}r\Bigr)_{2g} \, (1-2g) \, r^{2g} \, n_g(r) \, T^{2g}\,.
\eeq
Here for a power series $v(T)$, $[\,v(T)\,]_{\rm even}$ means taking the even degree part of~$v(T)$.

\smallskip

\noindent {\rm (vi)} {\rm [terminating hypergeometric sum]}
For all $g\geq 1$, 
\beq\label{tauge2g-dd}
\tau_{D_l}(g) \= \frac{(-1)^{g+m} }{\bigl(\bigl\{\frac{2g-1}r\bigr\}\bigr)_m} \, \frac{r^{1-g}}{1-2g} \, \sum_{d\geq0} \, \binom{\frac{2g-1}{r}}{d} \,K_{g,d} \,,
\eeq
where
\beq
K_{g,d} \= 
\sum_{m_1+m_2+\dots +m_{l-1}=d \atop m_1+2m_2+\dots+(l-1) \, m_{l-1}=g} \binom{d}{m_1,\dots,m_{l-1}} 
\, \prod_{i=1}^R \, \Biggl[\frac{\binom{\frac{r+i+1}2}{i+1}}{i+2}\Biggr]^{m_i} \,.
\eeq

\smallskip

\noindent {\rm (vii)} {\rm 
[asymptotics]\footnote{The precise form of part~(vii) was written by D.Y. and~D.Z. after the first author~B.D. passed away.}
} For $l>4$, 
  as $g\to\infty$, $\tau_{D_l}(g)$ is given asymptotically by  
\beq\label{dasymptotics}
\tau_{D_l}(g) \; \sim \;  \frac{r \, \sqrt{\pi} \,  \cos(\frac{\pi}{r}) }{\sqrt{r+1}^{\frac{r-2}r}} \;
\frac{1}{\Gamma\bigl(1- \bigl\{\frac{2g-1}r\bigr\}\bigr)} \, \frac{1}{ \Gamma(\frac{2g-1}{r}) } \,  g^{-\frac32} \,
\Bigl(4 \, r \, (r+1)^{\frac2r}\,  \sin^2\bigl(\tfrac{\pi}r\bigr)\Bigr)^{-g}\,.\eeq
For $l=4$ the formula is the same with an extra factor of~3.
\end{theorem}

\begin{proof}
Recall that the DS hierarchy of $D_l$-type admits 
the following scalar Lax operator~\cite{DS}:
\beq\label{laxD}
L \= \p^r \+ \p^{-1} \circ \sum_{\alpha=1}^{l-1} 
\bigl(u^\alpha \circ \p^{2\alpha-1} +\p^{2\alpha-1}\circ u^\alpha\bigr) \+ \p^{-1} \circ u^l \circ \p^{-1} \circ u^l \,.
\eeq
For $\alpha=1,\dots,l-1$, the corresponding DS flows  
can be written in terms of~$L$ by
\beq\label{DShierarchyD}
\frac{\p L}{\p s^{\alpha,q}} \= \Bigl[ \bigl(L^{k/r}\bigr)_+ \,, \,L \Bigr] \,, \quad q\ge 0\,,
\eeq
where $k=m_\alpha+r q$.
Denote ${\bf s}:=(s^{\alpha,q})_{\alpha=1,\dots,l, \, q\geq 0}$. 
For $\alpha=l$, the corresponding flows could also be obtained in terms of~$L$ via the so-called negative flows~\cite{LWZ,Takasaki} 
or by definition from the original Drinfeld--Sokolov's matrix Lax system~\cite{DS}.
Since $\p u^\alpha/\p s^{1,0} = \p u^\alpha/\p x$,
we identify~$s^{1,0}$ with~$x$. 
Witten's ADE conjecture~\cite{FFJMR, FJR, Witten2} for the $D_l$ case 
(the Fan-Jarvis-Ruan theorem~\cite{FFJMR,FJR}; see also~\cite{GM, KW, LRZ}) 
can then be stated as follows: the partition function~$Z:=Z(\bt({\bf s}))$ is a particular tau-function for the 
 DS hierarchy of $D_l$-type, where
 \beq\label{tsdcase}
 t^{\alpha,q} \= (-1)^{q+1} \, (\sqrt{-r})^{\frac{3k}{r+1}+1} \,\Bigl(\frac{m_\alpha}r\Bigr)_{q+1} \, s^{\alpha,q} \,, \quad q\geq 0\,;
 \eeq
moreover, $Z(\bt)$ satisfies the string equation~\eqref{stringgeneralZFJRW} with 
$\eta_{\alpha\beta}$ in~\eqref{stringgeneralZFJRW} given by
\beq
(\eta_{\alpha\beta}) \= 
\begin{pmatrix}
0 & \dots & 0 & 1 & 0 \\
\vdots & \reflectbox{$\ddots$} & \reflectbox{$\ddots$} & 0 & 0\\
0 & 1 & \reflectbox{$\ddots$}  & \reflectbox{$\ddots$} & \vdots\\
1 &  0 & \reflectbox{$\ddots$} & \reflectbox{$\ddots$} & 0 \\
0 & 0 & \dots & 0 & 1\\
\end{pmatrix} \,.
\eeq
One can alternatively write the string equation~\eqref{stringgeneralZFJRW} using the variables~$s^{\alpha,q}$ from~\eqref{tsdcase} as
\beq\label{stringZD}
\sum_{\alpha=1}^l \sum_{q\geq 0} \, (m_\alpha+rq+r) \, s^{\alpha,q+1} 
\frac{\p Z}{\p s^{\alpha,q}} \+  \sum_{\alpha=1}^{l-1}  \frac{m_\alpha (r-m_\alpha)}2 \, s^{\alpha,0} s^{l-\alpha,0} Z \+ \frac{m_l^2}2 (s^{l,0})^2 Z \= \frac{\p Z}{\p s^{1,0}}\,.
\eeq

Denote by $u^\alpha=u^\alpha({\bf s})$ the solution corresponding to the particular tau-function~$Z$, the so-called topological solution. 
The Lax operator~$L$ given by~\eqref{laxD} is now also subjected to this solution.
Similarly as in the proof for the $A$~case, we note that 
\beq\label{WtauhamD}
h_k \:=  \frac{(-1)^{q+1}}{(\sqrt{-r})^{3 \frac{1+k}{r+1}} (\frac{m_\alpha}{r})_{q+1} } \, {\rm res} \, L^{\frac{k}r}
\qquad \text{($k>0$ odd)}
\eeq
are a part of the tau-symmetric Hamiltonian densities 
for the DS hierarchy of $D_l$-type. Here $k=m_\alpha+r q$.
Therefore,
\beq\label{conjwD}
h_k({\bf s}({\bf t})) \= \frac{\p^2 \log Z(\bt)}{\p t^{\alpha,q} \p t^{1,0}} \,.
\eeq
The initial Lax operator $L|_{s^{\alpha,q} = \, x \, \delta^{\alpha,1} \delta^{q,0}}$ will again be denoted by~$L$. 
We have 
the following lemma.
\begin{lemma}
We have
\beq
 L \= \p^r \+ r \, x  - \frac{1}2 \, r \, \p^{-1} \,.
\eeq 
\end{lemma}
\begin{proof}
For the topological solution,  
 the corresponding normal coordinates~$r_{\alpha}$ satisfy 
\beq
r_\alpha(\bt) \= \frac{\p^2 \log Z(\bt)}{\p t^{\alpha,0}\p t^{1,0}} \,.
\eeq 
Dividing the string equation~\eqref{stringgeneralZFJRW} by~$Z(\bt)$, then differentiating it with respect to~$t^{\alpha,0}$, and finally taking 
$t^{\alpha,q} = t^{1,0} \, \delta^{\alpha,1} \delta^{q,0}$ 
we obtain that 
\beq
r_\alpha(\bt)|_{t^{\alpha,q} = t^{1,0} \, \delta^{\alpha,1} \delta^{q,0}} \= \eta_{\alpha1} \, t^{1,0} \= \delta_{\alpha,l-1} \, t^{1,0}\,.
\eeq
Then, as in~\cite{BDY2,BY}, by using the Miura-type transformation between the normal coordinates and the $u$-coordinates 
together with a degree argument 
we find that 
\[u^{\alpha}|_{s^{\beta,q}= \, x\,\delta^{\beta,1}\delta^{q,0}} \= \frac{r}2 \, x \, \delta^{\alpha,1} \,.\]  The lemma is proved.
\end{proof}
As before, denote $z_k(x):=\res L^{k/r}$.
Using~\eqref{conjwD}, \eqref{stringrec} and~\eqref{generalgap} we obtain that 
\beq\label{taughapDcase}
\tau_{D_l}(g)  \= (-1)^{q+g} \, \frac{z_{2g(r+1)-1}(0)}{r^{3g} \, (\frac{m_\alpha}{r})_{q+2}} \,.
\eeq

\begin{lemma}
The following formula holds true:
\beq\label{conjugationDLax}
L \= \p^{-\frac12} \circ \bigl(\p^r \+ r\,x\bigr) \circ \p^{\frac12} \,.
\eeq
\end{lemma}
\begin{proof} 
\quad $ {\rm RHS} \=  \p^r \+ r \, \p^{-\frac12} \circ x \circ \p^{\frac12}  \= {\rm LHS} $.
\end{proof}
\noindent Using~\eqref{Don'sformulaL} and~\eqref{conjugationDLax}, we have
\begin{align}
& \Bigl(\p^r \+ r \, x \,-\, \frac12 \, r \, \p^{-1} \Bigr)^{\lambda} \nn\\
& \quad \= \p^{-\frac12} \circ \bigl(\p^r\+r\,x\bigr)^{\lambda} \circ \p^{\frac12} \nn\\
& \quad \=  \sum_{0\leq p \leq j \atop s\geq 0} \, r^{j+s} \, c_{p,j}(r) \, (\lambda)_{s+j+p}^-  \, \sum_{v=0}^s \binom{-\frac{1}2}{v} \, 
\frac{x^{s-v}}{(s-v)!} \, \p^{r(\lambda-s)-(r+1)j-v}   \,. \nn
\end{align}
Combining with~\eqref{taughapDcase} we obtain~\eqref{ci2g1d}.

To prove~\eqref{algbraicacaseequation-d}, 
let us use the wave-function-pair approach (see Appendix~\ref{wavebi} and Section~\ref{techsection}). 
We first 
construct a particular pair of wave and dual wave functions for~$L$.
Start with solving $L\psi=z^r\psi$, i.e.,
\beq\label{wavedtype}
 \left(\p^r \+ (r \, x -z^r)  - \frac{r}2 \, \p^{-1} \right) \psi \= 0\,.
\eeq
Denote $X = z^r/r - x$. We have 
\beq \left((-\p_X)^r \,-\, r \, X \+ \frac{r}2 \, \p_X^{-1} \right) \psi \= 0 \,.\eeq
Similarly as in Section~\ref{techsection} 
we find that this equation has a unique formal solution of the form
\beq\label{p1dX}
P_1(X)\=e^{- \frac{1}{r+1} r^{\frac{r+1}{r}} X^{\frac{r+1}{r}} } X^{-\frac12}\sum_{m\geq 0} \frac{c_m}{X^{ \frac{(r+1)m}{r}}} \,, \qquad c_0 \:=1 \,.
\eeq
Therefore, $\psi(x,z)=\alpha_1(z)\,P_1(X)$ for some $\alpha_1(z)$ to be determined. 
The choice of~$\alpha_1(z)$ is not unique. 
We will use the particular choice of~$\alpha_1(z)$ given by
\beq
\alpha_1^{\rm bisp}(z) \:=  \frac{1}{e^{- \frac{r}{r+1} r^{\frac1r} X^{\frac{r+1}{r}}} X^{-\frac12}}\bigg|_{x=0} 
\=  r^{-\frac12} e^{ \frac{r}{r+1} C^{-1} z^{r+1}} z^{\frac{r}2}  \,.  
\eeq
We call this particular choice the {\it bispectral} one.
Namely, we construct 
\beq
\psi \= \psi(x,z) \:= \alpha_1^{\rm bisp} (z) \, P_1(X) \,.
\eeq 
Similarly, denote by  
\beq
P_2(X)\:=e^{\frac{r}{r+1}r^{\frac1r} X^{\frac{r+1}{r}}} X^{-\frac12+\frac1r} 
\sum_{m\geq 0} \frac{c_m^*}{X^{ \frac{(r+1)m}{r}}}\,, \qquad c_0^*\:=1
\eeq
the unique formal solution to the following linear equation:
\beq \left(\p_X^r \,-\, r \, X - \frac{r}2 \, \p_X^{-1}  \right)  \psi^* \= 0\,.\eeq
Define 
\beq
\alpha_2^{\rm bisp}(z) \:= r^{-\frac12+\frac1r} \, e^{ - \frac{1}{r+1} z^{r+1}} z^{\frac{r}2-1} \,, 
\eeq
and construct 
\beq
\psi^*\=\psi^*(x,z) := \alpha_2^{\rm bisp} (z) \, P_2(X) \,. \eeq

\begin{prop}\label{linipair}
The $\psi, \psi^*$ form a particular pair of wave and dual wave functions of~$L$. 
\end{prop}
\begin{proof}
It is easy to verify that~$\psi$ is a wave function and 
$\psi^*$ is a dual wave function for~$L$.
Write $\psi = \Phi_1 (e^{xz})$ and $\psi^*=\Phi_2(e^{-xz})$, where   
\[
\Phi_1 \= \sum_{k\geq 0} \phi_{1,k}(x) \, \p^{-k}\,, \quad \Phi_2 \= \sum_{k\geq 0} \phi_{2,k}(x) \, \p^{-k}
\]
with $\phi_{1,0}=\phi_{2,0}\equiv1$. To show $\psi$ and $\psi^*$ form a pair, it suffices to show $\Phi_1\circ \Phi_2^* =1$. 
Similarly as before, we know that this is equivalent to show that for all $i\geq 0$, 
\beq
\res_z \, \p^i\bigl(\psi(z,x)\bigr) \, \psi^*(z,x) \, dz \= 0\,. 
\eeq
Before continuing the proof of the proposition let us prove two lemmas.

\begin{lemma}\label{xzdd} 
Introduce two linear operators~$R_z$ and~$R_z^*$:
\begin{align}
& R_z \= \frac{1}{z^{r-1}} \, \p_z \,-\, \frac{r}{2} \,  z^{-r} \,-\,  z  \,,  \qquad R_z^* \= - \frac{1}{z^{r-1}} \, \p_z \+ \frac{r-2}{2} \, z^{-r} \,-\,  z  \,. 
\end{align}
Then we have for any~$i\geq 0$, 
\begin{align}
& \p_x^i \bigl(\psi(x,z)\bigr) \= \bigl(-R_z\bigr)^i \, \bigl(\psi (x,z)\bigr) \,, \qquad \p_x^i \bigl(\psi^*(x,z) \bigr) \= \bigl(R_z^*\bigr)^i \, \bigl(\psi^*(x,z) \bigr) \,. 
\end{align}
\end{lemma}
\begin{proof}  By direct calculations. 
\end{proof}
\begin{lemma}\label{psipsistarexpressiond} The $\psi$ and $\psi^*$ have the following expressions:
\begin{align}
& \psi(x,z) \= \sum_{i\ge0} 
\frac{(-1)^i}{i!} \, R_z^{\, i} \bigl(c(z)\bigr) \, x^i  \= \sum_{i\ge0} \, \frac{(xz)^i}{i!} \, f_{i-1/2} \bigl(\tfrac{1}{z^{r+1}}\bigr) \,,  \label{1501} \\
& \psi^* (x,z) \= 
\sum_{i\ge0} \frac{1}{i!} \, (R_z^*)^i \bigl(c^*(z)\bigr) \, x^i  \= \sum_{i\ge0} \, (-1)^i \, \frac{(xz)^i}{i!} \, f_{i+1/2}\bigl(\tfrac{-1}{z^{r+1}}\bigr)  \, . \label{1502}
\end{align}
Here, $c(z):=\psi(0,z)$, $c^*(z):=\psi^*(0,z)$, and $f_{i\pm1/2}$ are given by~\eqref{fjT}.
\end{lemma}
\begin{proof}
Doing the Taylor expansion of~$\psi$ with respect to~$x$ at $x=0$ and 
using Lemma~\ref{xzdd} we immediately get the first equality in~\eqref{1501}. From~\eqref{p1dX} we find that 
$c(z)$ has the form
\beq
c(z) \= \sum_{m\geq 0} \, r^{\frac{(r+1)m}{r}} \, \frac{c_m}{z^{(r+1)m}}\,,\qquad c_0=1\,.
\eeq
Define 
\[\tilde f_i \:= z^{-i} \, (-R_z)^i \bigl(c(z)\bigr) \,.\]
From the definition of~$S_z$ one easily finds that $\tilde f_i \in \CC[[z^{-(r+1)}]]$. Write
\[\tilde f_i \= \tilde f_i(T) \,, \quad  T\= \frac{1}{z^{r+1}} \,,\] 
and we find from the definition of~$R_z$ and~\eqref{wavedtype} that 
\begin{align}
& \tilde f_{j+1}(T) \= \biggl(1\+ \Bigl(\frac{r}2-j\Bigr)\,T \+  (r+1)\,T^2\frac d{dT} \biggr)\, \tilde f_j(T)\,, \nn\\
& \tilde f_{j+r}(T) \= \tilde f_j(T) \,-\, r \, \Bigl(j-\frac12\Bigr) \, T \, \tilde f_{j-1}(T)\,. \nn
\end{align}
Comparing with~\eqref{fid1}--\eqref{fid2} and using the uniqueness argument that is similar to the one given after the proof of Lemma~\ref{fjTrecursionlemma}, 
we conclude that $\tilde f_i = f_{i-1/2}$. This proves~\eqref{1501}. 
The proof of~\eqref{1502} is similar. The lemma is proved.
\end{proof}

\noindent{\it End of the proof of Proposition~\ref{linipair}.} 
Using Lemma~\ref{psipsistarexpressiond} and identity~\eqref{fifjTmT} we have 
\begin{align}
& \res_z \p^i\bigl(\psi(z,x)\bigr) \, \psi^*(z,x) \, dz   \nn\\
& = \; \res_z \sum_{m\ge0} \frac{z^{m+i} \, x^{m}}{m!} \, f_{m+i-1/2}\bigl(\tfrac{1}{z^{r+1}}\bigr) \, 
\sum_{\ell\ge0} (-1)^\ell \frac{(xz)^\ell}{\ell!} f_{\ell+1/2}\bigl(\tfrac{-1}{z^{r+1}}\bigr)  \, dz \nn\\
& = \; \res_z \sum_{m,\ell,q \geq 0}  (-1)^\ell \, \frac{z^{m+\ell+i-q(r+1)} x^{m+\ell}}{m! \, \ell!} 
\Bigl(1+\frac{q-i-m-\ell-1}{r}\Bigr)_q  \widetilde C_q\Bigl(r,m+i-\frac12,\ell+\frac12\Bigr)  \Bigl(\frac r2\Bigr)^q dz \nn\\
& = \; 0 \,.  \nn
\end{align}
Proposition~\ref{linipair} is proved. 
\end{proof}

Using formulas~\eqref{taughapDcase}, \eqref{defccH}, \eqref{formulaHz}, 
Proposition~\ref{linipair}, Lemma~\ref{psipsistarexpressiond} 
and Proposition~\ref{propequivanlencebetweenalgprod} we obtain that 
\beq \label{taudttC}
\tau_{D_l}(g) \= \frac{(-1)^{g+m-1}}{ \bigl(\bigl\{\frac{2g-1}r\bigr\}\bigr)_{m+1} } \, \frac{\tC_{\frac12,-\frac12,2g}(r)}{2^{2g} \, r^{3g}} \,.
\eeq
\begin{lemma}
For any $g\geq 0$, we have
\beq\label{bgtc2gd}
n_g \= -\frac{\tC_{2g}(r,\frac12,-\frac12)}{2^{2g} \, (2g-1)} \,.  
\eeq
\end{lemma}
\begin{proof}
From the definition of~$\tC_{n}(r,i,j)$ and 
the fact that $X$ is an odd Laurent series in~$u^{-1}$ one easily obtains the equality
\beq\label{dXdu}
\sum_{g\ge0} \, \tC_{2g} \, \Bigl(r,\frac12,-\frac12\Bigr) \, u^{2g} \= 
-\frac{u^2}2\,  \biggl(\sqrt{\frac{X+1}{X-1}} + \sqrt{\frac{X-1}{X+1}}\biggr) \, \frac{dX}{du} \= \frac{-u^2 X }{\sqrt{X^2-1}} \, \frac{dX}{du} \,.
\eeq
 Then it suffices to show that the series~$P$ defined by 
\beq\label{defPseries}
P\:=\sum_{g\ge0} \, \tC_{2g}\Bigl(r,\frac12,-\frac12\Bigr) \, \frac{u^{2g}}{1-2g} \= 1 \,-\, \frac{r+2}6 \, u^2 \+ \cdots
\eeq
satisfies 
\beq\label{defPseries2}
\sum_{j\ge0} \, \binom{j+r/2}{2j} \, \frac{(2u)^{2j} \, P^{r-2j}}{2j+1} \= 1 \, .
\eeq
(Note that since in our case $r=2l-2$ is a non-negative even integer, the above sum terminates at $j=l-1$ and 
expresses~$P$ as an algebraic function of~$u$, but even if $r$ is a complex or formal variable the identity 
makes sense and will be shown to be true.)  Comparing~\eqref{dXdu} and~\eqref{defPseries} we find
\[ \frac{d}{dX} \Bigl(\frac{P}{u}\Bigr) \= \frac{d(P/u)/du}{dX/du} \= \frac{X}{\sqrt{X^2-1}} \,. \]
Therefore, 
\[ P \= u \sqrt{X^2-1} \,.\]
Combining with~\eqref{deftc1} it suffices to show 
\[ \sum_{j\ge0} \, \frac{2^{2j}}{2j+1} \, \binom{j+r/2}{2j} \, (X^2-1)^{r/2-j}  \=  \frac{(X+1)^{r+1}-(X-1)^{r+1}}{2 \, (r+1)}\,, \]
which is an elementary exercise (also equivalent to~\eqref{algebraicd2}).
\end{proof} 
Identity~\eqref{f12m12thm} follows directly from~\eqref{taudttC}.
Formulas~\eqref{taudttC} and~\eqref{bgtc2gd} yield~\eqref{tauagag-d}.  
Formulas~\eqref{tauagag-d} and~\eqref{kuk2} yield~\eqref{tauge2g-dd}.

Finally, the proof of the asymptotic formula~\eqref{dasymptotics}
is very similar to that of the 
corresponding statement in Theorem~\ref{algebraicthma}, and will only be sketched. In view of 
the relation~\eqref{tauagag-d} we see that it is equivalent to the asymptotic formula
\beq\label{ngasymp}
n_g(r) \; \sim \;  \cos\bigl(\frac{\pi}{r}\bigr)  \, \frac{(r+1)^{\frac1r-\frac12}}{\sqrt{\pi g^3}} \; 
\frac{\sin((2g-1)\frac{\pi}{r}) }{\bigl(-4 \, (r+1)^{\frac2r} \sin^2(\frac{\pi}r)\bigr)^g} \qquad (l>4)
\eeq
for the coefficients $n_g(r)$ defined in~\eqref{algbraicacase-d}. 
Writing~\eqref{algbraicacaseequation-d} in the form $\bigl(\frac{Z^2}{Z^2-1}\bigr)^{r+1}
-\bigl(\frac1{Z^2-1}\bigr)^{r+1}=\frac{r+1}{t^r}$ and setting the derivative 
of this expression with respect to~$Z$ equal to~0, we find that $Z^2$ is a 
non-trivial $r$th root of unity at all singularities of $y=y(t)$, and 
substituting this back into~\eqref{algbraicacaseequation-d} we find that the singularities are given 
by $(t,Z^2)=((r+1)^{1/r}(\alpha-\beta),\alpha/\beta))$ where $\alpha$ and 
$\beta$ are distinct $r$th roots of unity.  In particular, the singular points
of~$y(t)$ of smallest absolute value are equal to $2\,(r+1)^{1/r}\sin(\pi/r)$ 
times~$(2r)$th roots of unity. The rest of the proof is exactly along the 
lines of the proof of part~(vii) of Theorem~\ref{algebraicthma} and will be omitted.
Just as in the $A$~case for~$l=1$, the proof has to be modified slightly if $l=4$ 
because the way that the sheets above the closed disk $\{|t|\le 2\,(r+1)^{1/r}\sin(\pi/r)\}$ meet 
at their boundary is slightly different from what happens for~$l>4$;  
we leave the details as an exercise.

This completes the proof of all parts of Theorem~\ref{algebraicthmd}.
\end{proof}

\begin{remark}
It seems worth observing that the factors $1/2$ and $3$ appearing in part~(vii) of Theorems~\ref{algebraicthma} and~\ref{algebraicthmd} 
respectively 
are related to the symmetries of the corresponding Dynkin diagrams. Indeed, 
$|{\rm Sym}(A_l)| = 2$ for $l>1$, 1 for $l=1$;
$|{\rm Sym}(D_l)| = 2$ for $l>4$, 6 for $l=4$.
\end{remark}

We note that the polynomial~$P$ on the left-hand side of~\eqref{algbraicacaseequation-d} also has the following expression:
\beq\label{algbraicacasetdcase}
P \= \frac{2 \, y^{r + 1} \sinh\bigl[(r + 1) \, {\rm arcsinh} \, \bigl(\frac{t}{2 y}\bigr)\bigr]}{(r + 1) \, t} \,-\, 1\,.
\eeq

Using the formulas~\eqref{algbraicacase-d}--\eqref{tauagag-d} one can compute the first few values of~$\tau_{D_l}(g)$:
\begin{align}
& \tau_{D_l}(1) \= \frac{r+2}{24} \,, \\
& \tau_{D_l}(2) \= \frac{(r+2)(r-6)(2r+1)}{5760\,r} \,,\\
& \tau_{D_l}(3) \=  \frac{(r+2) (2 r+1) (8 r^3-77 r^2+196 r+188) }{2903040 \, r^2} \,.
\end{align}
Here we note that $r\geq 6$ (as $l\geq 4$). For $r=6$, our formulas agree with the explicit computations in~\cite{BDY1}. 
More precisely, the explicit expression of the dual topological 
ODE of~$D_4$-type was computed in~\cite{BDY1}, which reduces to a second order ODE for~$\phi_3$ (in our current notation):
\beq\label{dualtopod4}
108 \, x^2 \, \phi_3'' \,-\, \bigl(104 \,x^8+108\,x\bigr) \, \phi_3' \,-\, \bigl(4 \, x^{14}+260 \, x^7+39\bigr) \, \phi_3 \= 0\,.
\eeq
One could then give an alternative proof of Theorem~\ref{algebraicthmd} for~$D_4$ by using~\eqref{dualtopod4}. 
We leave this as an exercise because in the next section we will prove the algebracity for~$E_6$ in this way. 

Similarly as in the $A$~case, we have the following corollary. 
\begin{cor}
For all $r\geq 2g$, the value of~$\tau_{D_l}(g)$ is 
a Laurent polynomial in~$r$. Moreover, the value of this Laurent polynomial at $r=-1$ is 
equal to $\frac{(1-2^{1-2g}) \, B_{2g}}{2g}$.
\end{cor}
\begin{proof} 
If $r\ge 2g$, then $m=0$, so the right-hand side of~\eqref{tauagag-d} reduces to $-n_g(r)/(-r)^{g-1}$. The first statement then 
follows from the fact that $n_g(r)\in\QQ[r]$. 
By taking the $r\to-1$ limit in~\eqref{algebraicd2} we find the unique solution $y=\frac{t/2}{\sinh{(t/2)}}$, which yields 
the second statement of the corollary.
\end{proof}

It might be of interest to see whether the coefficients or the values of the Laurent polynomials in~$r$ occurring in the above corollary 
and the corollary to Theorem~\ref{algebraicthma} have any topological meaning.

\section{The $E_6$ case\titlefootnote{This section is by D.Y.~and~D.Z. only, who found
 the proof in the~$E_6$ case after B.D. passed away.}}\label{sectione6}
In the proof of  
Theorems~\ref{algebraicthma} and~\ref{algebraicthmd} for the $A_l$, $D_l$ cases in the previous two sections, 
we used the equivalent scalar Lax representation of the corresponding DS hierarchy~\cite{DS}. 
For the $E_6$ case, as far as we know, the existence of such a representation 
is an open question. However, following Proposition~\ref{dualandinvsg}, 
we can get the higher-genera one-point invariants by computing the dual topological ODE of $E_6$-type. 
(Of course, for any fixed $l$, as it was mentioned  above (see~\eqref{dualtopod4}), 
  we could alternatively have used the corresponding dual topological ODE to give 
 a different proof of Theorems~\ref{algebraicthma}--\ref{algebraicthmd}; 
 this deserves a further study.)  In this section we will use this to prove the algebraicity
of a generating function of the $\tau$-numbers for~$E_6$.

We first compute the dual topological ODE of $E_6$-type.
We use the 27-dimensional representation~\cite{DLZ} of~$\g=E_6$. 
Recall that  the Coxeter number and the exponents for this case read
 as follows: 
\[r=12\,, \quad m_1=1\,,\quad m_2=4\,, \quad m_3=5\,, \quad m_4=7\,, \quad m_5=8\,, \quad m_6=11\,,\]
and that 
the dimension of this simple Lie algebra is~$78$. Denote 
\begin{align}
& X_1 \= E_{6,7} + E_{8,9} + E_{10,11} + E_{12, 14} + E_{15,17} + 
   E_{26, 27} \,, \nn\\
& X_2 \= E_{4, 5} + E_{6, 8} + E_{7, 9} - E_{18, 20} - E_{21, 22} - 
   E_{23, 24} \,, \nn\\
& X_3 \= E_{4, 6} + E_{5, 8} + E_{11, 13} + E_{14, 16} + E_{17, 19} + 
   E_{25, 26} \,, \nn\\
& X_4 \= E_{3, 4} - E_{8, 10} - E_{9, 11} - E_{16, 18} - E_{19, 21} + 
   E_{24, 25} \,, \nn\\
& X_5 \= E_{2, 3} - E_{10, 12} - E_{11, 14} - E_{13, 16} + E_{21, 23} + 
   E_{22, 24} \,, \nn\\
& X_6 \= E_{1, 2} + E_{12, 15} + E_{14, 17} + E_{16, 19} + E_{18, 21} + 
   E_{20, 22} \,. \nn
 \end{align}
Kostant's $sl_2$-subalgebra of~$\g$ can be given by
\begin{align}
& I^- \= 16 X_1^T + 22 X_2^T + 30  X_3^T + 
   42 X_4^T + 30  X_5^T + 16 X_6^T\,,\\
& I^+ \= X_1 + X_2 + X_3 + X_4 + X_5 + X_6\,,\\
& \rho^\vee \= \frac12 \, \bigl(I^+ I^- \,-\, I^- I^+\bigr)\,.
\end{align}

Using the above data one can compute explicitly the dual topological ODE of~$E_6$-type, 
i.e., equation~\eqref{dualtopo} for~$G\in\g$, or equivalently~\eqref{dualphi} for the vector 
$\phi=(\phi_1,\dots,\phi_6)\in\CC^6$. 
We denote by $(\phi_{\alpha;\beta})_{\beta=1,\dots,6}$ ($\alpha=1,\dots,6$) the six linearly 
independent vector solutions that were introduced in Section~\ref{dualtopoodesection}. 
(Here $\beta$ labels the components of the vector.)
Recall that each $\phi_{\alpha;6}$ has the form~\eqref{phialphalint}. 
Following the principle of Theorem~\ref{odethmg}, namely, by reducing  
the ODE~\eqref{dualphi} for the vector-valued function~$\phi$ 
to a scalar ODE for the top component~$\phi_6$ (the highest weight vector of~$\g$), 
we find the following fourth-order one:
\begin{align}
& 0 \= 2985984 \, x^4 \bigl(37 \, x^{39}-2775 \, x^{26}-36960 \, x^{13}+11520\bigr) \, \phi_6'''' \nn\\
& \qquad - \, 466560 \, x^3 \bigl(2331 \, x^{52}-162985 \, x^{39}-2985600 \, x^{26}-4951296 \, x^{13}+811008\bigr) \, \phi_6''' \nn\\
& \qquad - \, 27 \, x^2 \bigl(6545189 \, x^{65}-1276342935 \, x^{52}+10115971680 \, x^{39}-127523831040 \, x^{26}\nn\\
& \qquad\qquad\quad +860446310400 \, x^{13}-52110950400\bigr) \,  \phi_6'' \nn\\
& \qquad - \, 27 \, x \, \bigl(23310 \, x^{78}-8293439 \, x^{65}-3559160940 \, x^{52}-153887586840 \, x^{39} \nn\\
& \qquad\qquad +1228034776320 x^{26}+236111616000 \, x^{13}+49235558400\bigr)  \, \phi_6' \nn\\ 
& \qquad +\bigl(37 x^{91}-3464310  \, x^{78}+2278737540 \, x^{65}+114309996390 \, x^{52} \nn\\ 
&\qquad \qquad +10889113435200 \, x^{39}-60840963615600 \, x^{26} \nn\\
&\qquad \qquad -15770999462400 \, x^{13}-328914432000\bigr) \, \phi_6\,. \label{phi6equaiton}
\end{align}
This means that every $\phi_{\alpha;6}$ must satisfy~\eqref{phi6equaiton}.
From the discussion in Section~\ref{dualtopoodesection}, 
we know that each~$\phi_{\alpha;6}$ has the form 
\beq\label{phialphaform6}
\phi_{\alpha;6} \= x^{1+\frac{13}{12}m_\alpha} \, f_\alpha\bigl(x^{13}\bigr)
\eeq
for some power series $f_\alpha(u) \in \QQ[[u]]$ (resp.~$u^{-1} \QQ[[u]]$ for $\alpha=6$). 
But if we use the Frobenius method, then we find that the indicial equation for~\eqref{phi6equaiton}
at $x=0$ has only four roots $25/12$, $77/12$, $103/12$, $-1/12$.
This implies that~$\phi_{2;6}$ and~$\phi_{5;6}$ must both 
vanish (which was not obvious from their original definition), 
while the expansions of the other~$\phi_{\alpha;6}$, if we 
normalize to make the power series monic, are given by~\eqref{phialphaform6} with 
$1+\frac{13}{12}m_\alpha = \frac{26\alpha-1}{12}$ and 
\begin{align*}
f_1(u) & \= 1 \+\frac{4235 }{2^9\, 3^1\, 13} \, u \+ \frac{23102233}{2^{18}\, 3^2\, 13^2} \, u^2
\+\frac{381109489145}{2^{29}\, 3^3 \, 13^3} \, u^3 \+ \dots  \,,\\
f_3(u) & \= 1 \+\frac{4613}{2^{10} \, 13} \, u\+ 
\frac{340813583}{2^{19}\, 3^2 \, 5^1 \, 13^2} \, u^2 \+ \frac{1468738987769}{2^{28} \, 3^2 \, 13^3 \, 17} \, u^3 \+ \dots \,,\\
f_4(u) & \= 1 \+ \frac{34829}{2^8\, 5^1\, 7^1 \, 13} \, u \+ 
\frac{112497481}{2^{20}\, 3^2 \, 13^2} \, u^2 \+\frac{45611422760339}{2^{28} \, 3^2 \, 5^1 \, 7^1 \, 13^3\, 19} \, u^3 \+ \dots \,,\\
u\, f_6(u) &\=  1 \+ \frac{435}{2^8 \, 13} \, u  
\+\frac{330276383}{2^{19} \, 3^2 \, 11^1 \, 13^2} \, u^2 
\+\frac{7178883185}{2^{27} \, 3^1 \, 13^3} \, u^3 \+ \dots \,.
\end{align*}

According to Proposition~\ref{dualandinvsg}, 
the numbers~$\tau_{E_6}(g)$ with $g\ge1$ in the particular normalizations that we are using, 
can be expressed in terms of~$\phi_{\alpha;6}$ as follows: for $g\not\equiv2\,({\rm mod}\,3)$, 
\begin{align}
\tau_{E_6}(g) \= \frac{c_{\alpha} }{2^{6m} \,3^{4m}}  \, \bigl[u^{m}\bigr] (f_\alpha) \,, \label{taue6phi6}
\end{align}
where $\alpha\in\{1,3,4,6\}$ and~$m$ are determined by $2g-1=m_\alpha+12m$ and 
\beq
c_1 = \frac14 \,, \quad c_3 = \frac5{1152} \,, \quad c_4 = \frac{25}{27648}\,, \quad c_6= \frac1{2^63^4} \,;
\eeq
otherwise $\tau_{E_6}(g)=0$. 
For simplicity we set $\tau_{E_6}(0)=1$ which also agrees with~\eqref{taue6phi6}.
It should be noted that here the freedom of normalizations is fixed in such a way that it agrees with 
 the explicit Frobenius manifold potential of $E_6$-type 
 given by Klemm--Theisen--Schmidt~\cite{KTS} (see~\eqref{KTSpotential} below) 
 as well as with the genus~1 formula of Dubrovin--Zhang~\cite{DZ1} (cf.~\cite{DLZ, FJR}).
For the reader's convenience, let us list the first few values of~$\tau_{E_6}(g)$ in the following table:
\begin{align}
& \arraycolsep=5.0pt\def\arraystretch{1.5}
\begin{array}{?c?c|c|c|c|c|c|c|c|c|c|c?}
\Xhline{2\arrayrulewidth} g &  0 & 1 & 2 & 3 & 4 & 5 & 6 &7 & 8 & 9  & 10 \\
\Xhline{2\arrayrulewidth}  \tau_{E_6}(g) & 1 & \frac14 & 0  & \frac{5}{1152} & \frac{25}{27648}
& 0 & \frac{145}{5750784}& \frac{4235}{414056448}  
& 0 & \frac{23065}{79498838016} & \frac{174145}{3338951196672}  \\
\Xhline{2\arrayrulewidth}
\end{array} \nn \\
& \qquad \qquad \qquad \qquad \qquad    \mbox{One-point FJRW invariants of $E_6$-type}  \nn
\end{align}

\begin{theorem}\label{algebraicthme}
Set $B= 2^{1/12}  \bigl(3+2\sqrt{3}\bigr)^{1/4}$.
Define $A_k$, $k\geq-1$ from the generating series
\[ U(V) \= \frac1B \, \sum_{k\ge-1} \, A_k \, V^{-k/12}\,,\]  
where $U(V)$ is the unique solution 
in $$\frac1B \, V^{1/12} \+ \CC\bigl[\bigl[V^{-1/12}\bigr]\bigr]$$
to the polynomial equation 
\begin{align}
&U^{24} \,-\, 36  \, U^{22} \+ 540 \, U^{20} \,-\, 4488 \, U^{18} \+ 22992 \, U^{16} \,-\, 76032 \, U^{14} \nn\\
&\+\Bigl(5\,V+\frac{2140032}{13}\Bigr) \, U^{12} 
\,-\, 2 \, \Bigl(19\,V+\frac{1501248}{13} \Bigr) \, U^{10} 
\+ 108 \, \Bigl(V + \frac{24320}{13}\Bigr) \,  U^{8} \nn\\
&\,-\, 32 \, \Bigl( 5 \, V + \frac{41104}{13}\Bigr) \, U^6 \+ 32 \, \Bigl( 5 \, V + \frac{9696}{13} \Bigr) \, U^4 
\,-\, 96 \, \Bigl( V + \frac{192}{13}\Bigr) \, U^2  \nn\\
&\,-\, \frac1{108} \, \Bigl( V^2 - \frac{34560}{13} \, V - \frac{442368}{169}\Bigr) \=\, 0 \,,
\end{align}
of degree~$12$ in~$U^2$ and~$2$ in~$V$.
Then for all $g\ge0$ with $g\not\equiv2\,({\rm mod}\,3)$ we have
\beq\label{taue6ag}
\tau_{E_6}(g) \= \psi_\rho \, 2^{-\frac{13}6g} \, 3^{-\frac76g} \, \frac{A_{2g-1}}{\bigl(\bigl\{\frac{2g-1}{12}\bigr\}\bigr)_m}  \,,
\eeq
where $\rho\in\{1,5,7,11\}$ and~$m$ are such that $2 g-1 = \rho + 12 \, m$, 
and $\psi_\rho$ are given by
\beq
\psi_1 = 2^{\frac12} \, 3^{\frac{5}{12}} \,, \quad \psi_5=2^1\, 3^{\frac34} \, \varepsilon_3^{-\frac12} \,, \quad 
\psi_7 = - 2^{\frac{3}{2}} \, 3^{\frac23} \, \varepsilon_3^{-\frac12} \,, \quad
\psi_{11} = - 12 \,,
\eeq
where $\varepsilon_3=2+\sqrt{3}$. 
\end{theorem}
\begin{proof} 
The differential equation of~$\phi_6$ translates into recursions for the coefficients of~$f_\alpha$ with $\alpha=1,3,4,6$. On the other hand, 
any algebraic function satisfies a linear differential equation and hence its coefficients satisfy a recursion.  
One verifies by computer that the recursions and the initial conditions agree.
\end{proof}

The first few values of~$A_k$ are given by
\beq
A_{-1} = 1\,, \quad A_1 = 2^{-\frac13} \, 3^{\frac34} \,, \quad A_5 = 2^{-\frac32} \, 3^{\frac34} \, 5^1 \, \varepsilon_3^{\frac12}   \,, \quad 
A_7 =  - 2^{-\frac{17}6} \, 3^1 \, 5^2 \, \varepsilon_3^{\frac12}  \,.
\eeq
In general, the numbers $A_k$ are (up to fractional powers of 2, 3 and~$\varepsilon_3$) algebraic integers belonging to~$\QQ(\sqrt3)$, 
so that Theorem~\ref{algebraicthme} implies an integrality statement for the~$\tau_{E_6}(g)$ similar to the first one in 
Theorem~\ref{thmintegrality}.  Another consequence of the theorem 
 is an asymptotic formula for the numbers $\tau_{E_6}(g)$ similar to the ones we found for the~$A$ and $D$~cases:

\begin{cor}
As $g\to\infty$ with $g\not\equiv2 \, ({\rm mod}\,3)$, we have 
\beq
\tau_{E_6}(g) \; \sim \;  13^{-5/12} \, \sqrt{\pi} \,
\frac{  \theta_\rho}{\Gamma\bigl(1-\bigl\{\frac{2g-1}{12}\bigr\}\bigr)} \, \frac1{\Gamma\bigl(\frac{2g-1}{12}\bigr)} 
 \, g^{-\frac{3}2} \, \Biggl(\frac{\sqrt{3+2\sqrt{3}}}{2^1 \, 3^{\frac76} \, 13^{\frac1{6}}}\Biggr)^g \, ,
\eeq
where $\rho\in\{1,5,7,11\}$ and~$m$ are such that $2 g-1 = \rho + 12 \, m$,  
and $\theta_\rho$ are constants defined by 
\beq
\theta_1 \= 2^1 \, 3^{\frac5{12}} \, \bigl(1+\sqrt{3}\bigr) \,, \quad \theta_5 \= 2^2 \, 3^{\frac34} \,, \quad 
\theta_7 \=  2^{\frac52} \, 3^{\frac23} \,, \quad \theta_{11} \= 2^{\frac52} \, 3^1\, \bigl(1+\sqrt{3}\bigr)\,.
\eeq
\end{cor}

\begin{proof}
Using~\eqref{taue6ag} we find that it suffices to show that as $g\to\infty$ with $g\not\equiv2\,({\rm mod}\,3)$,
\beq
A_{2g-1} \; \sim \; \frac{13^{-5/12} }{\sqrt{\pi}}\, \chi_\rho \,  Q^{g/6} \, g^{-3/2}
\eeq
with $Q = 2^7\, \bigl(3+2\sqrt{3}\bigr)^3/13 $ and with 
\beq
\chi_1 \= -\chi_{11} \=1 \,, \quad \chi_5 \= -\chi_7 \= \varepsilon_3  \,.
\eeq
The corollary can then be proved again by using the similar argument as for the $A$~case. 
\end{proof}

\begin{remark}\label{remarkalgebraicitye}
We expect that Theorem~\ref{algebraicthme} (and its corollary) 
will have an analogue for~$E_7$ and~$E_8$, and 
actually also for all non-simply-laced simple Lie algebras. 
\end{remark}


\section{Explicit relation for FJRW invariants of $\g$-type}\label{sectionsurprising}
In this section, based on the results in the previous sections and 
on the theory of Frobenius manifolds (see~\cite{Du-3,Du-6,Du-7,Hertling,Manin,Teleman}), 
we generalize Theorem~\ref{thmduality0inftya4} from~$A_4$ to 
$A_l$, $D_l$ and $E_6$.

According to Kontsevich--Manin~\cite{KontsevichManin,Manin} (cf.~Appendix~\ref{appfm}), 
a homogeneous CohFT of charge~$d$ gives a formal Frobenius manifold of charge~$d$ with the 
formal Frobenius potential $F=F(v)$ given by~\eqref{cohof}. 
Denote by~$B$ the domain of convergence of~$F$.
Let $\g$ be a simply-laced simple Lie algebra of rank~$l$. 
As before, denote by $\Omega_{g,n}$ the FJRW CohFT of $\g$-type and by~$r$ the Coxeter number. 
It follows from the quasi-homogeneity~\eqref{ddg} 
that the formal Frobenius potential~$F$ is a polynomial. 
Therefore, the convergence domain~$B$ is the whole of~$\CC^l$. 
We call the Frobenius manifold~$B$ constructed from $\Omega_{g,n}$ 
{\it the Frobenius manifold of $\g$-type}. Note that 
the Frobenius structure on~$B$ given by~\eqref{cohofttofm1}--\eqref{multifromcohft} 
can alternatively be constructed from the orbit space of the 
Coxeter group of $\g$-type (\cite{Du-3, Du-5, Sa3, Zuber}), 
or else from the miniversal deformation of a simple 
singularity of $\g$-type (\cite{Buryak3, DVV2, DW, Du-6, Hertling, Sa1}).  
The charge and the spectrum of~$B$ are
\beq
d \= \frac{r-2}r\,, \quad \mu_\alpha=\frac{m_\alpha}r -\frac12 \,, \quad R\=0\,,
\eeq
and the Euler vector field of~$B$ is 
\beq\label{EulerABCDEFG}
E \= \sum_{\alpha=1}^l  \frac{r+1-m_\alpha} r  \, v^\alpha \p_\alpha  \,.
\eeq

Before stating and proving the generalized theorem, 
we first prove a useful lemma.  
\begin{lemma}\label{uniquenesscali}
If $\g$ is a simply-laced simple Lie algebra,
there exists a unique choice of calibration 
$\{\theta_{\alpha,m}\}_{m\geq 0}$ for the Frobenius manifold of~$\g$-type.
\end{lemma}
\begin{proof}
The existence of $\theta_{\alpha,m}$ that satisfy~\eqref{theta1c}--\eqref{theta3c}
is known. (For the meaning of calibration and the proof of the existence see Appendix~\ref{appfm}.) 
We only need to prove the uniqueness. First of all, 
observing that set of the differences $\{\mu_\alpha-\mu_\beta\}$ does not contain 
positive integers, we know from~\cite{Du-3,Du-6} that
equations \eqref{yvalpha}--\eqref{yz} and \eqref{YThetazmuzr}--\eqref{normalThetatt} determine  
$\Theta(v;z)$ uniquely. Therefore, $\theta_{\alpha,m}$ are 
uniquely determined by~\eqref{theta0c} and~\eqref{thetagradientTheta} up to constants only, i.e.,  
if $\{\tilde{\theta}_{\alpha,m}\}_{m\geq 0}$ also satisfies~\eqref{theta1c}--\eqref{theta4c}, 
then $\tilde\theta_{\alpha,m}(v)- \theta_{\alpha,m}(v)=b_{\alpha,m}$ with $b_{\alpha,m}$ being constants. 
The uniqueness then follows immediately from~\eqref{theta3c}. The lemma is proved.
\end{proof}

Lemma~\ref{uniquenesscali} is true for the Frobenius manifold associated to any finite Coxeter group 
with the same proof as above.

\smallskip

The following theorem gives the main result of this section.

\begin{theorem} \label{thmduality0infty}
Let $\g$ be a simple Lie algebra of type $A_l$, $D_l$, or $E_6$, 
and let~$(B,\langle\,,\,\rangle,\cdot)$ denote the Frobenius manifold of $\g$-type. 
Take $v=(v_1,\dots,v_l)$ the flat coordinates on~$B$ and
 $(\theta_{\alpha,m})_{\alpha=1,\dots,l,m\ge0}$ the associated unique calibration.
Denote $v^*_\alpha=\langle\tau_{\alpha, m_\alpha}\rangle$.
Then the following identity holds:  
\beq\label{thetatauduality}
 \langle \tau_{\alpha, m_\alpha+(r+1)m} \rangle  \=  \theta_{\alpha,m}(v^*) \,, \quad \forall\, m\geq0\,.
\eeq
Here, we recall that $\langle \tau_{\alpha,q} \rangle$ denote the one-point FJRW invariants of $\g$-type.
\end{theorem}
\begin{proof} 
Let us prove this theorem case by case from~$A$ to~$E$. 

Start with the $A_l$ case.
Consider the Frobenius manifold of $A_l$-type. In this case, 
$r=l+1$ is the Coxeter number. Define
\beq\label{lambdaaaa}
\lambda \= \lambda(p;s) \=  p^{l+1} \+ \sum_{\beta=1}^l s_\beta \, p^{l-\beta} \,.
\eeq
Here $s=(s_1,\dots,s_l)\in\CC^l$. 
The reader who is familiar with the theory of Frobenius manifolds recognizes~$\lambda$ as 
the superpotential, and the Frobenius structure on~$B$ could be defined by the 
standard formulas~\cite{Du-3,Du-6,Du-7} (see~also \cite{DVV2,EKHYY,EYY1,Sa1,Sa2}) using 
the superpotential under some carefully chosen normalization factors, and the flat coordinates~$v$
are related to~$s$ by 
\beq\label{Wtauham0}
v^\alpha \= - \frac{(\sqrt{-r})^{\frac{3\alpha}{r+1}-1}}{r-\alpha} \, {\rm res}_{p=\infty} \, \lambda^{\frac{r-\alpha}r}  dp\,.
\eeq
It is not difficult to verify that 
$\det \bigl(\frac{\p v^\alpha}{\p s_\beta}\bigr) \, \neq  \, 0$, so 
the map $s\mapsto v$ given by~\eqref{Wtauham0} is indeed an invertible coordinate transformation. 
By a degree argument, this coordinate change~\eqref{Wtauham0} actually has the triangular nature.
\begin{lemma}\label{lemmaperiodresidue}
For the $A_l$ case, 
the unique functions $\theta_{\alpha,m}(v)$, $m\geq 0$ have the following expressions:
\beq\label{Wtauhamacase}
\theta_{\alpha,m}(v) \= 
- \frac{(-1)^{m+1}}{(\sqrt{-r})^{3 \frac{1+\alpha+rm}{r+1}} (\frac{\alpha}{r})_{m+1} } \, {\rm res}_{p=\infty} \, \lambda^{\frac{\alpha+rm}r} dp\,.
\eeq
\end{lemma}
\begin{proof}
We prove this lemma by showing that the right-hand side of~\eqref{Wtauhamacase}, denoted by~$\tilde\theta_{\alpha,m}$, satisfies the 
defining conditions~\eqref{theta1c}--\eqref{theta3c}.
The conditions~\eqref{theta4c} with $z=0$ and~\eqref{theta1c} are known to hold true. Indeed, they  
follow from the Laplace-type transform between the deformed flat coordinates
and $\lambda$-periods of the Frobenius manifold, 
with $\lambda$ being the corresponding superpotential~\cite{AVG,DVV2,DW,Du-3,Du-6,EYY1,EYY2, Sa1,Sa2} 
and with the careful choosing of the rescaling factor to match with Witten's normalization of times~\cite{Witten2}. 
Here, the definitions of the product and the invariant flat metric can be found~{\it ibid.}, 
and we also normalize the metric such that $\eta_{\alpha\beta}=\delta_{\alpha+\beta,l+1}$.
The equality~\eqref{theta6a} can be easily proved by computing the residues, which also implies~\eqref{theta6b}. 
Since equality~\eqref{theta4c} is a rewritting of~\eqref{normalThetatt}, 
we know that \eqref{theta4c} is actually a consequence of~\eqref{theta1c}--\eqref{theta6b}. 
To show~\eqref{theta2c}, introduce the following extended Euler operator:
\beq\label{eEE}
\widetilde{E} \:= E\+ \frac1r \, p \, \frac{\p}{\p p} \,,
\eeq
where $E$ in this case is given by~\eqref{EulerABCDEFG} with $m_\alpha=\alpha$. 
From~\eqref{lambdaaaa} and~\eqref{Wtauham0} we obtain 
\beq\label{ELL}
\widetilde E \, \lambda \= \lambda\,.
\eeq
The required quasi-homogeneity~\eqref{theta2c} for the gradients of~$\theta_{\alpha,m}$ then holds true for the gradients 
of~$\tilde\theta_{\alpha,m}$. 
It remains to show~\eqref{theta3c}.  
Observe from the definition~\eqref{Wtauhamacase}--\eqref{ELL} that the $\tilde\theta_{\alpha,m}$
themselves  
satisfy the following
quasi-homogeneity condition:
\beq\label{quasithetaa}
 E \, \theta_{\alpha,m} \= \Bigl(m+\mu_\alpha+\frac1r+\frac12\Bigr) \, \theta_{\alpha,m} \,,\quad \forall\,m\geq0\,.
\eeq
Hence the $\tilde\theta_{\alpha,m}$, $m\geq0$, which are polynomials of~$v$, do not contain constant terms. 
This gives~\eqref{theta3c}. The lemma is proved. 
\end{proof}

We are ready to show~\eqref{thetatauduality}.
For $m=0$, the validity of~\eqref{thetatauduality} is  
obvious, because from the definition we know that $\theta_{\alpha,0}=v_\alpha$. 
Let us now prove the validity of~\eqref{thetatauduality} for an arbitrary $m\geq0$. Following~\cite{Du-3,Du-6,Du-7}, 
consider the polynomial equation
\begin{align}
& \lambda(p;s) \= \xi^r \,. \label{inversefunction}
\end{align}
We know that equation~\eqref{inversefunction} 
has a unique solution $p=p(\xi;s)$ in $\xi + \CC[s][[\xi^{-1}]]$.
Write
\beq
p(\xi;s) \= \xi \+ \sum_{k\geq 1} u_k(s) \, \xi^{-k} \,,
\eeq
where $u_k(s)\in \QQ[s]$, $k\geq 1$.  Using similar arguments to those in the previous sections we have
\begin{align}
k \, u_k(s) \= \res_{p=\infty} \, \lambda(p;s)^{\frac{k}{r}} \, dp \,.
\end{align}
Comparing this equality with~\eqref{Wtauhamacase} we obtain that 
\beq\label{Wtauhamtheta}
\theta_{\alpha,m}(v) \= 
- \, \frac{(-1)^{m+1}(\alpha+mr)}{(\sqrt{-r})^{3 \frac{1+\alpha+rm}{r+1}} (\frac{\alpha}{r})_{m+1} } \, u_{\alpha+mr}(s)\,,\quad \forall \,m\geq0\,.
\eeq
Now define 
\beq
s_\beta^*\= \left\{ \begin{array}{cl} 
\binom{r}{2i}\frac{a^{2i}}{1+2i}\,, & \beta=2i-1\; (i=1,\dots,[r/2])\,;\\
0 \,,  & \mbox{ otherwise.}\\
\end{array}\right. \, \label{sstarexpression}
\eeq
And we restrict our discussion to the particular point~$s=s^*$ on~$B$. 
Here 
\beq 
a\:=  \sqrt{-r}^{\frac{2-r}{r+1}} /2 \,. \label{aexpression}
\eeq
From~\eqref{lambdaaaa}, \eqref{inversefunction} and~\eqref{sstarexpression}, 
we know that the series $p(\xi;s^*)$ satisfies
\[
\sum_{i=0}^{[r/2]} \binom{r}{2i} a^{2i} \, \frac{p(\xi;s^*)^{r-2i}}{1+2i}  \= \xi^r \,.
\]
Comparing this equation with formula~\eqref{algbraicacase}--\eqref{tauagag} (cf.~\eqref{equivpolya}) we get 
\[\tau_{A_{r-1}}(g)\= \frac{(-1)^{m+g} \, r^{1-g}}{2^{2g} \, a^{2g} \, (\frac{\alpha}{r})_{m}}  \, u_{2g-1} \,, \]
where $g,\alpha,m$ are related by $2g-1=\alpha+rm$.
Substituting~\eqref{Wtauhamtheta} in this equality 
and using \eqref{rgaq} and~\eqref{aexpression}, 
we obtain~\eqref{thetatauduality}.  This completes the proof for the $A$~case. 

We continue to prove the statement for the $D$~case. In this case, the Coxeter number~$r=2l-2$. Similarly as in the $A$~case, introduce the 
superpotential~\cite{DVV2}: 
\beq\label{lambdadddd}
\lambda \=  p^{r} \+ \sum_{\beta=1}^{l-1} s_\beta \, p^{r-2\beta} \+ s_l^2 \, p^{-2} \,.
\eeq
The flat coordinates in our normalization are given by 
\beq\label{Wtauham0d}
v^\alpha \:= - \, \frac{(\sqrt{-r})^{\frac{3m_\alpha}{r+1}-1}}{r-m_\alpha} \, {\rm res}_{p=\infty} \lambda^{\frac{r-m_\alpha}r}  dp\,, \quad \alpha=1,\dots,l-1
\eeq
and  
\beq\label{Wtauham0dl}
v^l \:= - \frac{2\sqrt{l-1}}{\sqrt{-r}^{\frac32\frac{r+2}{r+1}}} \, s_l \,.
\eeq
It is not difficult to verify again that 
$\det \bigl(\p v^\alpha/ \p s_\beta\bigr) \, \neq  \, 0$, and the map $s\mapsto v$ given by~\eqref{Wtauham0d}--\eqref{Wtauham0dl} indeed
gives an invertible coordinate transformation.
\begin{lemma}
The unique functions $\theta_{\alpha,m}(v)$, $m\geq 0$ have the following expressions:
\begin{align}
& \theta_{\alpha,m}(v) \= - \frac{(-1)^{m+1}}{(\sqrt{-r})^{3 \frac{1+m_\alpha+rm}{r+1}} (\frac{m_\alpha}{r})_{m+1} } \, {\rm res}_{p=\infty} \, \lambda^{\frac{m_\alpha}r+m} dp\,,\quad \alpha=1,\dots,l-1\,,\label{Wtauhamd} \\
& \theta_{l,m}(v) \= \frac{(-1)^{m+1} \sqrt{l-1}}{(\sqrt{-r})^{3 \frac{l+rm}{r+1}} (\frac12)_{m+1} } \, 
{\rm res}_{p=0} \, \lambda^{\frac12+m} dp \,.\label{Wtauhamdl}
\end{align}
\end{lemma}
\begin{proof}
The proof is almost identical with that of Lemma~\ref{lemmaperiodresidue}, so we only give some brief indications 
here. Note that the 
invariant metric~$\eta$ in our normalization is given by 
\[
\eta \= 
\begin{pmatrix}
0 & \dots & 0 & 1 & 0 \\
\vdots & \reflectbox{$\ddots$} & \reflectbox{$\ddots$} & 0 & 0\\
0 & \reflectbox{$\ddots$} & \reflectbox{$\ddots$}  & \reflectbox{$\ddots$} & \vdots\\
1 &  0 & \reflectbox{$\ddots$} & \reflectbox{$\ddots$} & 0 \\
0 & 0 & \dots & 0 & 1\\
\end{pmatrix} \,.
\]
The extension of the Euler vector field is again given by~\eqref{eEE}. The lemma is proved.
\end{proof}

Following~\cite{Du-3,Du-6,Du-7}, we consider the following polynomial equation:
\beq
 \lambda(p;s) \= \xi^r \,. \label{inversefunctiond}
\eeq
It is easy to see that~\eqref{inversefunctiond} 
has a unique solution~$p=p^+(\xi;s)$ in $\xi + \CC[s][[\xi^{-1}]]$ as well as a unique solution $p=p^-(\xi;s)$ in 
$s_l \, \xi^{-r/2} + \xi^{-(r+2)/2} \CC[s][[\xi^{-1}]]$. Write 
\begin{align}
& p^+(\xi;s) \, = \, \xi \+ \sum_{k\geq 1} u_k(s) \, \xi^{-k} \,, 
\quad p^-(\xi;s) \, = \, \sum_{m\geq 0} v_m(s) \, \xi^{-m-\frac{r}2} \,.\label{defintionQstardminus}
\end{align}
Here $v_0(s)=s_l$. Differentiating both equations in~\eqref{defintionQstardminus} 
with respect to~$\xi$, we find 
\begin{align}
& \frac{dp^+}{d\xi}
\, = \, 1  \,-\,  \sum_{k\geq 1} k \, u_k(s) \, \xi^{-k-1} \,, \quad \frac{dp^-}{d\xi}
\, = \,   \,-\, \sum_{m\geq 0}  \Bigl(m+\frac{r}2\Bigr) \, v_m(s) \, \xi^{-m-\frac{r}2-1} \,.\label{defintionQstardminusderivative}
\end{align}
Therefore, 
\begin{align}
& k u_k(s) \= \res_{\xi=\infty} \frac{dp^+(\xi;s)}{d\xi} \xi^{k} d\xi 
\= \res_{p=\infty} \lambda(p;s)^{\frac{k}{r}} dp \,, \label{uksd}\\
& \Bigl(m+\frac{r}2\Bigr) v_m(s) \= \res_{\xi=\infty} \frac{dp^-(\xi;s)}{d\xi} \xi^{rm+\frac{r}2} d\xi
\= - \res_{p=0} \lambda(p;s)^{m+\frac12} dp \,. \label{vbsd}
\end{align} 
Comparing~\eqref{uksd} and~\eqref{vbsd} with~\eqref{Wtauhamd} and~\eqref{Wtauhamdl}, respectively, we have for $\alpha=1,\dots,l-1$,
\beq\label{Wtauhamthetad}
\theta_{\alpha,m}(v) \= 
- \frac{(-1)^{m+1}(m_\alpha+mr)}{(\sqrt{-r})^{3 \frac{1+m_\alpha+rm}{r+1}} (\frac{m_\alpha}{r})_{m+1} } u_{m_\alpha+mr}(s)\,,\quad \forall \,m\geq0\,,
\eeq
and for $\alpha=l$,
\beq\label{Wtauhamthetadl}
\theta_{l,m}(v) \= 
- \sqrt{l-1}\frac{(-1)^{m+1}(l-1+mr)}{(\sqrt{-r})^{3 \frac{l+rm}{r+1}} (\frac12)_{m+1} } v_m(s)\,,\quad \forall \,m\geq0\, .
\eeq
Define 
\beq
s_\beta^*\= \left\{\begin{array}{cl}
\binom{l-1+\beta}{2\beta} \frac{a^{2\beta}}{1+2\beta}  \,, & \beta=1,\dots,l-1\,, \\
0 \,,  & \mbox{ otherwise,}
\end{array}\right. \, \label{sstarexpressiond}
\eeq
and we now restrict the discussion to the particular point~$s=s^*$ on~$B$. 
Here 
\beq 
a\:= \sqrt{-r}^{\frac{2-r}{r+1}} \,. \label{aexpressiond}
\eeq
From~\eqref{lambdadddd}, \eqref{inversefunctiond} and~\eqref{sstarexpressiond}, 
we know that the series $p^+(\xi;s^*)$ and~$p^-(\xi;s^*)$ satisfy
\[
\sum_{i=0}^{l-1} \, \binom{i+l-1}{2i} \, \frac{a^{2i}}{1+2i} \, p^+(\xi;s^*)^{r-2i}  \= \xi^r \,,\qquad p^-(\xi;s^*) \= 0 \,.
\]
Comparing with~\eqref{algbraicacase-d}--\eqref{tauagag-d} 
 we get
\[\tau_{D_l}(g)\= \frac{(-1)^{m+g}  \, r^{1-g} }{a^{2g} \, (\frac{m_\alpha}{r})_{m}}  \, u_{2g-1} \,, \]
where $g$ is related with $\alpha,m$ by $2g-1=m_\alpha+rm$.
Substituting~\eqref{Wtauhamthetad} in this identity and using~\eqref{generalgap}, and~\eqref{aexpression}, 
we find~\eqref{thetatauduality}. This completes the proof for the $D$~case.

Finally, let us prove the statement for the $E_6$ case.
Let $B=\CC^6$ be the Frobenius manifold of $E_6$-type. 
The spectrum data $(\mu,R)$ are now given by
\[\mu_1 = -\frac5{12}\,, \quad \mu_2 = -\frac16\,, \quad \mu_3= -\frac1{12}\,, \quad \mu_4 = \frac1{12} \,, \quad \mu_5 = \frac16 \,, \quad \mu_6 = \frac5{12} \,,\quad R=0\,. \]
According to Klemm--Theisen--Schmidt~\cite{KTS}, the Frobenius potential has the explicit expression:
\begin{align}
F\= 
& \; \frac{1}{2} \, (v^1)^2 \, v^6 \+ v^1 \,  v^2 \, v^5 \+ v^1 \, v^3 \, v^4  \+\frac{1}{2} \, (v^2)^2 \, v^3 \nn\\
&\+\frac{1}{2} \, (v^2)^2 \, v^4 \, v^6 \+ \frac{1}{2} \, v^2 \, (v^4)^2 \, v^5 \+\frac{1}{6} \, (v^3)^3 \, v^6 \+\frac{1}{4} \, (v^3)^2 \, (v^5)^2\nn\\
&\+\frac{1}{2} \, v^2 \, v^3 \, v^5 \, (v^6)^2 \+ \frac{1}{6} \, v^2 \, (v^5)^3 \, v^6 \+\frac{1}{2} \, v^3 \, v^4 \, (v^5)^2  \, v^6 \+
\frac{1}{12} \, v^4 \, (v^5)^4 \+ \frac{1}{12} \, (v^4)^4 \, v^6 \nn\\
&\+\frac{1}{6} \, v^2 \, v^4 \, v^5 \, (v^6)^3 \+\frac{1}{24} \, (v^2)^2 \, (v^6)^4\+\frac{1}{4} \, (v^4)^2 \, (v^5)^2 \, (v^6)^2 
\+\frac{1}{6} \, v^3 \, (v^4)^2 \, (v^6)^3 \nn\\
&\+\frac{1}{60} \, (v^3)^2 \, (v^6)^5 \+ \frac{1}{24} \, v^4 \,  (v^5)^2 \, (v^6)^5\+\frac{1}{24} \, v^3 \, (v^5)^2 \, (v^6)^4\+\frac{1}{24} \, (v^5)^4 \, (v^6)^3\nn\\
&\+\frac{1}{252} \, (v^4)^2 \, (v^6)^7\+\frac{1}{576} \, (v^5)^2 \, (v^6)^8\+\frac{(v^6)^{13}}{185328} \,. \label{KTSpotential}
\end{align}
Here, $(v^1,\dots,v^6)\in B$ is a system of flat coordinates satisfying that 
$\p_1$ is the identity vector field. The invariant flat metric and the Euler vector filed 
in this coordinate system are given by
$\eta_{\alpha\beta}=\delta_{\alpha+\beta,7}$, and 
\beq\label{EulerE6}
E\= v^1 \, \p_1 \+ \frac34 \, v^2 \, \p_2 \+ \frac23 \, v^3 \, \p_3 \+ \frac12 \, v^4 \, \p_4 \+ \frac5{12} \, v^5 \, \p_5 \+ \frac16 \, v^6 \, \p_6\,.
\eeq
The first few values of the unique calibration $\theta_{\alpha,p}$ for~$B$ are
\begin{align*}
&\theta_{\alpha,0}\= v_\alpha \,, \\
&\theta_{1,1} \= v^1 \, v^6 \+ v^3 \, v^4 \+ v^2 \, v^5 \,,\\ 
&\theta_{2,1} \= v^1 \, v^5 \+ v^2 \, v^3 \+v^2 \, v^4 \, v^6 \+\frac{1}{2} \, (v^4)^2 \, v^5  \\
& \quad \quad \quad \+\frac{1}{2} \, v^3 \, v^5 \, (v^6)^2\+\frac{1}{6} \, v^4 \, v^5 \, (v^6)^3\+\frac{1}{6} \, (v^5)^3 \, v^6 \+ \frac{1}{12} \, v^2 \, (v^6)^4 \,.
\end{align*}

To prove~\eqref{thetatauduality}, similarly as in the~$A$ and $D$~cases, we 
will use the corresponding superpotential.
Following Eguchi--Yang~\cite{EY} (cf.~\cite{Brini,KO,LW}), introduce three polynomials $Q_1,P_1,P_2$:
\begin{align}
& Q_1 \=  270 \, p^{15} \+ \bigl(171+57\sqrt{3}\bigr) \, t_{10} \, p^{13} \+ \bigl(54 + 27 \sqrt{3}\bigr) \, t_{10}^2 \, p^{11} \nn\\
& \qquad \+ \bigl(126+84 \sqrt{3}\bigr) \, t_7 \, p^{10} 
\+ \Bigl(\Bigl(\frac{35}4+\frac{175}{36}\sqrt{3}\Bigr) \, t_{10}^3+ \bigl(144+72\sqrt{3}\bigr) \, t_6\Bigr) \, p^9 \nn\\
& \qquad \+ \Bigl(\frac{135}2+\frac{81}2\sqrt{3}\Bigr) \, t_7 \, t_{10} \, p^8 \nn\\
& \qquad \+ 
\Bigl(\Bigl(\frac{225}4+\frac{125}4\sqrt{3}\Bigr) \, t_6 \, t_{10}
+ \Bigl(\frac{345}{384}+\frac{35}{96}\sqrt{3}\Bigr) \, t_{10}^4+ \bigl(135+81\sqrt{3}\bigr) \, t_4\Bigr) \, p^7 \nn\\
& \qquad \+ \Bigl(\bigl(126+72\sqrt{3}\bigr) \, t_3+\Bigl(10+\frac{35}6\sqrt{3}\Bigr) \, t_7\,t_{10}^2\Bigr) \, p^6\nn\\
& \qquad 
\+ \Bigl( \Bigl(\frac{63}4+9\sqrt{3}\Bigr) \, t_7^2+\bigl(36+21\sqrt{3}\bigr) \, t_4\,t_{10}
+ \Bigl(\frac{11}{768}+\frac{19\sqrt{3}}{2304}\Bigr) \, t_{10}^5+ \Bigl(\frac{21}4+3\sqrt{3}\Bigr)\, t_6\,t_{10}^2\Bigr) \, p^5\nn\\
& \qquad \+ \Bigl(\Bigl(\frac{33}2+\frac{19}2\sqrt{3}\Bigr) \, t_3 \, t_{10} 
+ \Bigl(\frac{19}{48}+\frac{11}{48}\sqrt{3}\Bigr)  \, t_7\,  t_{10}^3 + \bigl(24+14\sqrt{3}\bigr) \, t_6 \,t_7\Bigr) \, p^4\nn\\
& \qquad \,-\, \Bigl(\frac{11}8 + \frac{19}{24}\sqrt{3}\Bigr) \, t_7^3 \, t_{10} \, p^3 \+ \Bigl(\frac{45}4+\frac{13}2\sqrt{3}\Bigr) \, t_3 \, t_7 \, p 
\+ \Bigl(\frac58+\frac{13}{36}\sqrt{3}\Bigr) \, t_7^3 \,,\\
& P_1 \= 78 \, p^{10} \+ \bigl(30+10\sqrt{3}\bigr) \, t_{10} \, p^8 \+ 
\Bigl(\frac{14}3+\frac73\sqrt{3}\Bigr) \, t_{10}^2 \, p^6 \+ \Bigl(\frac{33}2+11\sqrt{3}\Bigr) \, t_7 \, p^5 \nn\\
& \qquad \+ \Bigl( \Bigl(\frac14+\frac5{36}\sqrt{3}\Bigr) \, t_{10}^3 + \bigl(16+8\sqrt{3}\bigr) \, t_6\Bigr) \, p^4 
\+ \Bigl(\frac{25}{12}+\frac54\sqrt{3}\Bigr) \, t_7 \, t_{10} \, p^3 \nn\\
& \qquad \+ \Bigl( \bigl(5+3\sqrt{3}\bigr) \, t_4 + \Bigl(\frac7{3456}+\frac1{864}\sqrt{3}\Bigr) \, t_{10}^4 
+ \Bigl(\frac34+\frac5{12}\sqrt{3}\Bigr) \, t_6 \, t_{10} \Bigr) \, p^2 \nn\\
& \qquad \,-\, \Bigl(\frac72+2\sqrt{3}\Bigr) \, t_3 \, p \,-\, \Bigl(\frac7{12}+\frac13\sqrt{3}\Bigr) \, t_7^2 \,,\\
& P_2 \= 12 \, p^{10} \+ \bigl(6+2\sqrt{3}\bigr) \, t_{10} \, p^8 \+ \Bigl(\frac43+\frac23\sqrt{3}\Bigr) \, t_{10}^2 \, p^6 \+ \bigl(6+4\sqrt{3}\bigr) \, t_7 \, p^5 \nn\\
& \qquad \+ \Bigl(\bigl(8+4\sqrt{3}\bigr) \, t_6+ \Bigl(\frac18+\frac5{72}\sqrt{3}\Bigr) \, t_{10}^3\Bigr) \, p^4 
\+ \Bigl(\frac53+\sqrt{3}\Bigr) \, t_7 \, t_{10} \, p^3 \nn\\
& \qquad \+ \Bigl( \bigl(10+6\sqrt{3}\bigr) \, t_4+ \Bigl(\frac7{1728}+\frac1{432}\sqrt{3}\Bigr) \, t_{10}^4 
+ \Bigl(\frac32+\frac56\sqrt{3}\Bigr) \, t_6 \, t_{10} \Bigr) \, p^2\nn\\
& \qquad \+ \bigl(14+8\sqrt{3}\bigr) \, t_3 \, p \+ \Bigl(\frac7{12}+\frac13\sqrt{3}\Bigr) \, t_7^2 \,,
\end{align}
and define
\beq\label{lambdaeeee}
\lambda(p;t)\= \frac1{270+156\sqrt{3}} \, \biggl(-u_0 \+ \frac{Q_1+P_1\sqrt{P_2}}{p^3}\biggr) \,,
\eeq
where 
\begin{align}
&u_0\=- \bigl(270+156\sqrt{3}\bigr) \, t_0 -\frac1{16}\bigl(19+11\sqrt{3}\bigr) \, t_4 \, t_{10}^2-\frac1{576}\bigl(33+19\sqrt{3}\bigr) \, t_6 \, t_{10}^3 \nn\\
&\qquad\quad -\frac14 \bigl(21+12\sqrt{3}\bigr) \, t_6^2-\frac1{16}\bigl(33+19\sqrt{3}\bigr) \, t_{10} \, t_7^2\,.
\end{align}
Here, $(t_0,t_3,t_4,t_6,t_7,t_{10})$ gives the flat coordinates of Eguchi--Yang, which  
relate to the flat coordinates $(v^1,\dots,v^6)$ of Klemm--Theisen--Schmidt via a rescaling:
\beq
v^1 \= \kappa_6 \, t_0 \,, \quad v^2 \= \kappa_5 \, t_3\,, \quad v^3 \= \kappa_4 \, t_4 \,, \quad v^4 \= \kappa_3 \, t_6\,, \quad 
v^5 \= \kappa_2 \, t_7 \,, \quad v^6 \= \kappa_1 \, t_{10}\,,
\eeq
where
\begin{align}
&\kappa_1 \= -\frac1{2\sqrt{3}} \,, \quad \kappa_2 \= (\sqrt{3}-1)^{\frac12} \,, \quad \kappa_3  \= - (\sqrt{3}-1)\,, \nn\\ 
&\kappa_4 \= 2 (\sqrt{3}-1) \,, \quad \kappa_5\= -2 \sqrt{3} \, (\sqrt{3}-1)^{\frac12} \,, \quad \kappa_6 \= 8 \sqrt{3}\,. \nn
\end{align}

\begin{lemma}
We have
\begin{align}
\theta_{\alpha,m} (v) \= - \kappa_\alpha \, \frac{(8\sqrt{3})^m}{ (\frac{m_{\alpha}}{12} )_{m+1}} \res_{p=\infty} \lambda^{\frac{m_\alpha+12m}{12}} dp \,,\quad 
\alpha=1,\dots,6,\, m\geq0\,. \label{Wtauhame6}
\end{align}
\end{lemma}
\begin{proof}
The conditions~\eqref{theta4c} with $z=0$ and~\eqref{theta1c} 
have been verified by Eguchi--Yang~\cite{EY}, where we use 
different normalization constants $\kappa_\alpha$ from the ones in~\cite{EY} to 
agree with Klemm--Theisen--Schmidt's normalization 
for the product and the metric.
The proof is again similar to that of Lemma~\ref{lemmaperiodresidue}. For this case, 
the extended Euler operator is again defined by
\beq\label{eEEE}
\widetilde{E} \:= E\+ \frac1r \, p \, \frac{\p}{\p p} \,,
\eeq
where $E$ is given by~\eqref{EulerE6}. We have $\widetilde{E}\lambda=\lambda$ 
and the quasi-homogeneity~\eqref{theta2c} for the gradients of~$\theta_{\alpha,m}$ 
follows. The condition~\eqref{theta3c} follows from the quasi-homogeneity (just as~\eqref{quasithetaa})
of~$\theta_{\alpha,m}$. The lemma is proved. 
\end{proof}

Following~\cite{Du-3,Du-6,Du-7}, we consider the polynomial equation
\beq
 \lambda(p;t) \= \xi^{12} \,. \label{inversefunctione6}
\eeq
It is easy to see that~\eqref{inversefunctione6} has a unique solution $p=p(\xi;t)$ 
in $\xi + \CC[t][[\xi^{-1}]]$. Write 
\begin{align}
& p(\xi;t) \= \xi \+ \sum_{k\geq 1} u_k(t) \, \xi^{-k} \, . \label{defintionQstare6}
\end{align}
By using the same arguments as before we have
\begin{align}
& k \, u_k(t) 
\= \res_{p=\infty} \lambda(p;t)^{\frac{k}{12}} dp \,. \label{ukse6}
\end{align} 
Comparing~\eqref{ukse6} with~\eqref{Wtauhame6} we obtain 
\beq\label{Wtauhamthetae6}
\theta_{\alpha,m}(v) \= 
- \kappa_\alpha \, (8\sqrt{3})^m \, \frac{m_\alpha+12m}{ (\frac{m_\alpha}{12})_{m+1} } u_{m_\alpha+12m}(t)\,,\quad \forall \,m\geq0\,.
\eeq
Now define $(v^2)^*=(v^5)^* =0$ and 
\beq
(v^1)^*= \frac{145}{5750784} \,, 
\quad (v^3)^*=\frac{25}{27648} \,, \quad (v^4)^*=\frac5{1152}\,, 
\quad (v^6)^*= \frac14 \,,
\label{sstarexpressione6}
\eeq
and restrict the discussion to the particular point~$v=v^*$ on~$B$, which corresponds $t=t^*$.
From~\eqref{lambdaeeee}, \eqref{inversefunctione6}  
we know that the series $p=p(\xi;t^*)$ satisfies
\beq\label{algbraice6equation}
\frac1{270+156\sqrt{3}} \, \Biggl(-u_0^* \+ \frac{Q_1^*+P_1^*\sqrt{P_2^*}}{p^3}\Biggr) \= \xi^{12} \,,
\eeq
where $Q_1^*$, $P_1^*$, $P_2^*$, $u_0^*$ are respectively $Q_1$, $P_1$, $P_2$, $u_0$ evaluated at $t=t^*$.
Theorem~\ref{thmduality0infty} 
 is proved by simplifying~\eqref{algbraice6equation}.
\end{proof}

Observe that on the left-hand side of~\eqref{thetatauduality}, 
$g=\frac{m_\alpha+rm+1}{2}$. So if $m_\alpha+rm+1$ is an odd number, then as 
a direct consequence of~\eqref{thetatauduality} the value $\theta_{\alpha,m}(v^*)$ must vanish. 

We also have the following corollary.

\begin{cor}\label{equivalenttotheorem}
For a FJRW CohFT of $\g$-type with $\g$ being $A_l$, $D_l$, or $E_6$, we have
\beq
\langle \tau_{\alpha, \, m_{\alpha}+(r+1)m}\rangle  \= 
\sum_{n\geq 0}  
\sum_{1\leq \alpha_1,\dots,\alpha_n\leq l \atop 1\leq \beta_1,\dots,\beta_n \leq l}
\frac{\langle \tau_{\alpha,m-1} \tau_{\alpha_1,0}\dots \tau_{\alpha_n,0}\rangle_{0}}{n!} \, 
\prod_{i=1}^n \, \eta^{\alpha_i \beta_i} \, \langle \tau_{\beta_i,m_{\beta_i}}\rangle \,,  \quad \forall\,m\geq1\,. \label{si}
\eeq
\end{cor}
\begin{proof}
Following~\cite{Du-3}, consider the principal integrable hierarchy associated to~$B$: 
\beq\label{principalfm}
\frac{\p v^\alpha}{\p t^{\beta,q}} \= \eta^{\alpha\gamma} \p_{t^{1,0}} \biggl(\frac{\p \theta_{\beta,q+1}}{\p v^\gamma}\biggr) \,,\quad q\geq 0
\eeq
and its initial value problem (viewing $t^{1,0}$ as the space variable) with the initial data:
\beq
v^\alpha|_{t^{\beta,q}=t^{1,0} \, \delta^{\beta,1} \, \delta^{q,0}} \= \delta^{\alpha,1} \, t^{1,0}\,.
\eeq
The defining equations~\eqref{theta1c}--\eqref{theta3c} for $\theta_{\beta,q}$ and 
the axioms of Frobenius manifolds imply that 
the equations in~\eqref{principalfm} all commute, and thus the above initial value problem 
has a unique solution in $\CC[[\bt]]^{\otimes l}$, called the topological solution to the principal hierarchy, 
denoted by~$v^{\rm top}(\bt)= \bigl(v^{{\rm top}, 1}(\bt),\dots,v^{{\rm top},l}(\bt)\bigr)$. 
It is shown in~\cite{Du-3} that the following identity holds true:
\beq
\frac{\p^2 \F_0(\bt)}{\p t^{1,0} \p t^{\beta,m}} \= \theta_{\beta,m}\bigl(v^{\rm top}(\bt)\bigr)\,.
\eeq
Taking $t^{\gamma,p}=t^{\gamma,0}\,\delta^{p,0}$ on both sides of this identity, and noticing that 
\beq
v^{{\rm top},\alpha}(\bt)\big|_{t^{\gamma,p}=t^{\gamma,0}\,\delta^{p,0}} \= t^{\alpha,0} \,,
\eeq
 we obtain
\beq\label{twopointtheta}
\frac{\p^2 \F_0(\bt)}{\p t^{1,0} \p t^{\beta,m}} \bigg|_{t^{\gamma,p}=t^{\gamma,0}\,\delta^{p,0}} 
\= \theta_{\beta,m}\bigl(t^{1,0},\dots,t^{l,0}\bigr)\,.
\eeq
On the other hand, the genus zero part of the string equation~\eqref{stringgeneralZFJRW} reads as follows:
\beq\label{stringgenuszero}
\sum_{p\geq 1} t^{\alpha,p} \frac{\p\mathcal{F}_0(\bt)}{\p t^{\alpha,p-1}} \+ \frac12 \, \eta_{\rho\sigma} \, t^{\rho,0} t^{\sigma,0} 
\= \frac{\p\mathcal{F}_0(\bt)}{\p t^{1,0}} \,.
\eeq
For $m\geq1$, by differentiating~\eqref{stringgenuszero} with respect to~$t^{\beta,m}$ and taking $t^{\gamma,q}=t^{\gamma,0}\,\delta^{q,0}$ we get 
\beq\label{twopointonepoint}
\frac{\p\mathcal{F}_0(\bt)}{\p t^{\beta,m-1}} \bigg|_{t^{\gamma,p}=t^{\gamma,0}\,\delta^{p,0}}\=
\frac{\p^2\mathcal{F}_0(\bt)}{\p t^{1,0}\p t^{\beta,m}} \bigg|_{t^{\gamma,p}=t^{\gamma,0}\,\delta^{p,0}} \,.
\eeq
Using the definition~\eqref{FgCohFT} for~$\F_0(\bt)$, identities~\eqref{twopointtheta} and~\eqref{twopointonepoint}, as well as 
identity~\eqref{thetatauduality}, we obtain~\eqref{si}. 
\end{proof}
Note that due to~\eqref{ddg} the sum ``$\sum_{n\geq 0}$" 
in the right-hand side of~\eqref{si} is actually a finite sum. We also note that identity~\eqref{si} can 
be written alternatively as follows: for all $g\geq(r+2)/2$,
\beq
\tau_{\g} (g) \= 
\sum_{n\geq 0}  
\sum_{1\leq \alpha_1,\dots,\alpha_n\leq l \atop 1\leq \beta_1,\dots,\beta_n \leq l}
\frac{\langle \tau_{\alpha,m-1} \tau_{\alpha_1,0}\dots \tau_{\alpha_n,0}\rangle_{0}}{n!} \, 
\prod_{i=1}^n \, \eta^{\alpha_i \beta_i} \, \langle \tau_{\beta_i,m_{\beta_i}}\rangle\,, \label{si2}
\eeq
where $\alpha,m$ are such that $2g-1=m_\alpha+rm$.

\begin{remark}\label{remark10}
For $\g=E_7$ or $E_8$, 
we also expect the validity of identity~\eqref{thetatauduality} (or~\eqref{si},~\eqref{si2}). 
It might be possible to get proofs in these cases by applying the constructions of the $\lambda$-periods 
 for the orbit space of the Coxeter group of $\g$-type (\cite{Du-3,Du-4,Du-5,Sa3}) 
 (or for the simple singularity of~$\g$-type~\cite{AVG,BaBu,Du-6,GM,Hertling,Loo,Sa1,Sa2,SaM}) 
as well as their Laplace-type transforms~\cite{Du-3,Du-6,Du-7} to compute the~$\theta_{\alpha,m}$,
and then matching these with the left-hand side~$\tau_\g(g)$ of~\eqref{thetatauduality}, which can be read off from  
the coefficients of the top components~$\phi_{\alpha;l}$ 
of the fundamental solutions to the dual topological ODE of~$\g$-type. 
Moreover, since we know that the $\lambda$-periods for the Frobenius manifold of~$\g$-type are algebraic, 
this method of proof, if it works, would also lead to algebraicity (as already mentioned in Remark~\ref{remarkalgebraicitye}) 
and therefore integrality of the renormalized numbers of~$\tau_\g(g)$.
When $\g$ is a non-simply-laced simple Lie algebra, in order for identity~\eqref{thetatauduality} (or~\eqref{si},~\eqref{si2}) to remain valid,
one may need to use the notion of the {\it partial CohFT}~\cite{LRZ}.   
We hope to study these cases and other more general situations 
(semisimple or nonsemisimple CohFTs 
including the cases with the Novikov ring
mentioned in Remark~\ref{cohomonovikov}; cf.~e.g.~\cite{Brini,Du-3,DSZZ,DZ2,Zhou4}) in later work.
\end{remark}

\begin{remark}
We also observe that
the particular point~$v^*$ on the Frobenius manifold of $A$-type that we use 
for $l\ge2$ is different from the particular semisimple point 
used by Pandharipande, Pixton and Zvonkine~\cite{PPZ1,PPZ2}
for obtaining relations in the cohomology ring of $\overline{\mathcal{M}}_{g,n}$. 
It might be interesting to see 
whether their method 
can be applied also to our~$v^*$ to give 
further information connected with the topology of~$\overline{\mathcal{M}}_{g,n}$. 
\end{remark}

We provide a few examples to illustrate the results of this section.

\begin{example}[$A_1$] The superpotential reads as follows:
\[\lambda\=p^2 \+ s_1 \,, \quad v^1 \= \frac{s_1}2\,.\]
We have $\eta=1$ and $ F=(v^1)^3/6$.
For $m\geq0$, we know that $\theta_{1,m}(v) = (v^1)^{1+m}/(1+m)!$. 
Recall again the well-known formula for the one-point invariants: 
\begin{align}
& \langle \tau_{1,1+3m} \rangle_{g=1+m} \= \frac{1}{24^{1+m} (1+m)!} \,.
\end{align}
From these explicit expressions we immediately see the validity of identity~\eqref{thetatauduality}, where 
the particular point of the Frobenius manifold is given by $v_1^*=1/24$.
\end{example}

\begin{example}[$A_2$]
The superpotential reads 
\begin{align}
&  \lambda\=p^3\+s_1\, p \+ s_2\,,  \quad  v^1 \= \frac{s_2}{3 \sqrt{-3}^{1/4}}\,, \quad v^2 \= \frac{\sqrt{-3}^{1/2}}3 \, s_1  \,. \nn
\end{align}
We have 
$F =  (v^1)^2 v^2 /2 + (v^2)^4/72$ and $\eta_{\alpha\beta}=\delta_{\alpha+\beta,3}$.
The first few terms of the unique calibration can be read from
\begin{align}
& \theta_{2} (v;z) \= v^1 \+ \biggl(\frac{(v^2)^3}{18}+\frac{(v^1)^2}{2}\biggr) z  
\+ \biggl(\frac{(v^1)^3}{6}+\frac{1}{18} \, (v^2)^3 \, v^1 \biggr) z^2  \+\cdots \,,\nn\\
& \theta_1(v;z) \=  v^2 \+ v^1 \, v^2 \, z \+ \biggl( \frac{(v^2)^4}{36}+\frac{1}{2} \, (v^1)^2 \, v^2 \biggr) \, z^2 \+ \cdots \,.\nn
\end{align}
The particular point~$v^*$ of the Frobenius manifold is given by $v_1^*=1/12$ and $v_2^*=0$. 
\end{example}

\begin{example}[$A_4$] We have
\[\lambda\= p^5 \+ s_1 \, p^3 \+ s_2 \, p^2 \+ s_3 \, p \+ s_4 \,,\]
\[
v^1 \= - \frac{s_1 s_2-5s_4}{25 \sqrt{-5}^{\frac12}}  \,, \quad  v^2 \= - \frac{s_1^2-5s_3}{25}  \,, 
\quad v^3 \= \frac15 \sqrt{-5}^{\frac12} s_2 \,, \quad v^4 \= \frac{\sqrt{-5}}{5}s_1 \,, \nn
\]
\begin{align}
& F\= \frac{1}{2} (v^1)^2 (v^4) \+ v^1 v^2 v^3 \+ \frac{(v^2)^3}{6} \+ \frac{(v^4)^6}{15000}\+\frac{1}{150} (v^3)^2 (v^4)^3 \nn\\
&   \qquad \qquad \+\frac{1}{20} (v^2)^2 (v^4)^2
\+\frac{1}{10} v^2 (v^3)^2 v^4\+\frac{(v^3)^4}{60}   \,. \nn
\end{align}
and $\eta_{\alpha\beta}=\delta_{\alpha+\beta,5}$.
The first few terms of the unique calibration and the particular point~$v^*$ of the Frobenius manifold are already 
given in Theorem~\ref{thmduality0inftya4}.
\end{example}

\begin{example} [$D_4$]
The superpotential reads
\[
\lambda \=   p^6 \+ s_1 \, p^4 \+ s_2 \, p^2 \+ s_3 \+ \frac{s_4^2}{p^2} \,,
\]
and
\begin{align}
& p^+(\xi;s) \= \xi \,-\, \frac{s_1}{6 \xi}  \+ \frac{s_1^2-4 s_2}{24 \xi^3}  
\+\frac{-7 s_1^3+36 s_2 s_1-216 s_3}{1296 \xi^5} \nn\\
& \qquad \qquad  \+ \frac{-55 s_1^4+360 s_2 s_1^2-864 s_3 s_1-432 s_2^2-5184 s_4^2}{31104 \xi^7} \+ \cdots \,, \nn\\
& p^-(\xi;s) \= \frac{s_4}{\xi^3} \+ \frac{s_3 s_4}{2 \xi^9} \+ 
\frac{(3 s_3^2+4 s_2 s_4^2) s_4}{8 \xi^{15}}  \+ 
\frac{(8 s_1 s_4^4+20 s_2 s_3 s_4^2+5 s_3^3) s_4}{16 \xi^{21}}
  \+ \cdots \,. \nn
\end{align}
We have $\eta={\tiny \begin{pmatrix} 0 & 0 & 1 & 0 \\ 0 & 1 & 0 & 0 \\ 1 & 0 & 0 & 0 \\ 0 & 0 & 0 & 1 \end{pmatrix}}$ and 
\begin{align}
& F\= \frac{1}{2} \, (v^1)^2 \, v^3\+\frac{1}{2} \, v^1\,  (v^2)^2\+\frac{1}{36} \, (v^2)^3 \, v^3 \+\frac{1}{216}  \, (v^2)^2 \, (v^3)^3\+\frac{(v^3)^7}{272160}\nn\\
& \qquad \qquad \qquad \+ \biggl(\frac{v^1}{2}-\frac{1}{12} v^2 v^3+\frac{(v^3)^3}{216}\biggr) \, (v^4)^2 \,.
\end{align}
We have $\theta_{\alpha,0}=v_\alpha=\eta_{\alpha\beta} v^\beta$, and 
\begin{align*}
& \theta_{1,1} \= \frac{(v^2)^2}{2}\+\frac{(v^4)^2}{2} \+ v^1 \, v^3 \,, \\
& \theta_{2,1} \= \frac{1}{108}  \, v^2 \, (v^3)^3 \+ \frac{1}{12} \,  (v^2)^2 \, v^3 \,-\, \frac{1}{12} \, v^3 \, (v^4)^2 \+ v^1 \, v^2\,,\\
& \theta_{3,1} \= \frac{(v^3)^6}{38880}\+\frac{1}{72} (v^2)^2 \, (v^3)^2\+\frac{1}{72} \, (v^3)^2 \, (v^4)^2 \+\frac{(v^2)^3}{36}
\,- \, \frac{1}{12} \, v^2 \, (v^4)^2\+ \frac{(v^1)^2}{2} \,, \\
& \theta_{4,1} \= \frac{1}{108} \, v^4 \, (v^3)^3 \,-\, \frac{1}{6} \, v^2 \, v^3 \, v^4 \+ v^1 \, v^4\,.
\end{align*}
The particular point~$v^*$ of the Frobenius manifold is given by
\beq
v_1^* = \frac13 \,, \quad v_2^* =0 \,, \quad v_3^* = \frac{13}{40824} \,, \quad v_4^*=0 \,.
\eeq 
We list in the following table the first few of the numbers $\tau_{D_4}(g)$, again putting $\tau_{D_4}(0)=1$.
\begin{align}
& \arraycolsep=5.0pt\def\arraystretch{1.5}
\begin{array}{?c?c|c|c|c|c|c|c|c|c|c?}
\Xhline{2\arrayrulewidth} g &  0 & 1 & 2 & 3 & 4 & 5 & 6 &7 & 8 & 9   \\
\Xhline{2\arrayrulewidth}  \tau_{D_4}(g) & 1 & \frac13 & 0  & \frac{13}{40824} & \frac{13}{122472}
& 0 & \frac{1433}{16665989760} & \frac{253}{9999593856}  
& 0 & \frac{33917}{2041117097886720}   \\
\Xhline{2\arrayrulewidth}
\end{array} \nn \\
& \qquad \qquad \qquad \qquad \qquad    \mbox{One-point FJRW invariants of $D_4$-type}  \nn
\end{align}
\end{example}

\appendix
\section{Wave functions}  \label{wavebi}
In this appendix we give a method of computing residues of 
pseudo-differential operators by means of wave functions. 
Denote by~$r$ a positive integer. Let $q_m(x)$ ($m\leq r-1$) be arbitrarily given power 
series of~$x$, and~$L$ the pseudodifferential operator 
\beq\label{deflaxqmL}
L \:= \p^r \+ \sum_{m=-\infty}^{r-1}  q_m(x) \, \p^m \, .
\eeq
An element~$\psi=\psi(x,z)\in \CC[[x]][[z^{-1}]]\otimes e^{xz}$ of the form
$\psi = \sum_{i= 0}^\infty \, \phi_i(x) \, z^{-i} \, e^{xz}$ with $\phi_0(x) \equiv 1$ 
is called a {\it wave function} of~$L$ if 
\begin{align} 
& L \, \psi \= z^r \, \psi  \, .  \label{wave}
\end{align}
Here we recall that $\p^{-i} (e^{xz}) :=  e^{xz} z^{-i}$, $i\geq 0$, and that for all $f(x)\in\CC[[x]]$,
$$
\p^m \circ f(x) \:= \sum_{\ell\geq 0} \binom{m}{\ell} \, f^{(\ell)}(x) \, \p^{m-\ell} \,, \quad\forall\, m\in \mathbb{Z}\,.
$$
Multiplying by~$e^{-xz}$ on both sides of~\eqref{wave} and comparing 
the coefficients of powers of~$z$ give
\begin{align}
\sum_{j,\ell\geq 0\atop j+\ell\leq i} q_{\ell+j+r-i}(x) \, 
\binom{\ell+j+r-i}{\ell} \, \phi_j^{(\ell)}(x)  \= \phi_i(x) \,, \quad i\geq0\,. \label{recursivephi}
\end{align}
This leads to recursive relations for the power series $\phi_i(x)$, $i\geq 0$. 
The solution~$\psi$ depends on a sequence of arbitrary constants 
$C_1,C_2,\cdots$. Alternatively, we observe that if~$\psi(x,z)$ is a wave function, then
for an arbitrary power series~$g(z)$ of~$z^{-1}$ with constant coefficients $g(z) =\sum_{i\geq 0} g_i \, z^{-i}$, $g_0=1$,  
 the product $g(z) \, \psi(z,x)$ is again a wave function.   

A pseudo-differential operator~$\Phi$ of the form 
\beq\label{formdressing}
\Phi \= \sum_{i\geq 0}  \widetilde \phi_i(x) \, \p^{-i} \,, \quad \widetilde\phi_i(x)\in \CC[[x]]\,, ~ \widetilde \phi_0(x)\equiv1
\eeq
 is called a {\it dressing operator of~$L$} if  
\beq \label{dressing}
\Phi \circ \p^r \circ \Phi^{-1} \= L \,.
\eeq
For the given~$q_m(x)$ ($m\leq r-1$), the dressing operator is not unique, its freedom 
being in one-to-one correspondence with the coefficients of~$g(z)$ above. 
Indeed, there is a one-to-one correspondence between wave functions~$\psi$ and dressing 
operators by 
$\psi = \sum_i\,  \phi_i(x) \, z^{-i}\,e^{xz} \leftrightarrow \Phi = \sum_i \, \phi_i(x) \, \partial^{-i}$.  
Define the formal adjoint operator~$\Phi^*$ of~$\Phi$ by
\[\Phi^* \:= \sum_{i\geq 0} \, (-1)^i \, \p^{-i} \circ  \widetilde\phi_i(x)\,.\]
Fix~$\psi$ a wave function of~$L$, and take~$\Phi$ to be the corresponding dressing operator. 
Define~$\psi^*$ by
\beq\label{uniquedual}
\psi^* \= \psi^*(x,z) \:= \bigl(\Phi^{-1}\bigr)^* \, \bigl(e^{-xz}\bigr) \,.
\eeq
Clearly, $\psi^*\in \CC[[x]][[z^{-1]}]]\otimes e^{-xz}$, the product $\psi^*e^{xz}$ has leading term~$1$, and 
\begin{align} 
& L^* \, \psi^* \= z^r \, \psi^*  \, .  \label{waved} 
\end{align}
We call $\psi^*$ {\it a dual wave function of~$L$} associated to~$\psi$, and we call  
 $(\psi,\psi^*)$ {\it a pair of wave and dual wave functions}.

\begin{lemma}\label{pppp}
Let $\bigl(\psi,\psi^*\bigr)$ and $\bigl(\widetilde\psi,\widetilde\psi^*\bigr)$ be two pairs of wave and dual wave functions of~$L$. Then
$$
\psi(x,z) \, \psi^*(x,z) \= \widetilde\psi(x,z) \, \widetilde\psi^*(x,z) \,. 
$$
\end{lemma}
\begin{proof}
Let $g=g(x,z):=\widetilde\psi(x,z) / \psi(x,z)$, which
must have the form 
\[g\=\sum_{i\geq0} g_i(x)\,z^{-i}\in \CC[[x]]\bigl[\!\bigl[z^{-1}\bigr]\!\bigr]\,, \quad g_0(x)\equiv 1 \,.\] 
It follows from~\eqref{recursivephi} that $g_i'(x) =0$, $i\geq 1$. 
Therefore, for $i\geq 1$, $g_i$ are all constants. 
Let $\Phi$, $\widetilde\Phi$ be the dressing operators corresponding to~$\psi$, 
$\widetilde \psi$, respectively. 
We have
$\widetilde\psi = g(z)  \, \Phi  \bigl(e^{xz}\bigr) 
= \Phi \, \circ G \bigl( e^{xz}\bigr) $, where 
$G := \sum_{i\geq 0} g_i \, \p_x^{-i} $.
Therefore, $\widetilde \Phi = \Phi \circ G$. It follows that 
$\bigl(\widetilde\Phi^{-1}\bigr)^* \circ G^* = \bigl(\Phi^{-1}\bigr)^*$. So 
$\bigl(\widetilde\Phi^{-1}\bigr)^* \bigl(g(z) \, e^{-xz}\bigr) =  \bigl(\Phi^{-1}\bigr)^*  \bigl( e^{-xz} \bigr)$. 
Namely,
$g(z) \widetilde\psi^*(x,z)  = \psi^*(x,z)$.
The lemma is proved. 
\end{proof}

Let $\psi$ and $\psi^*$ be a pair of wave and dual wave functions of~$L$. Define 
\beq\label{defccH}
c(z) \,:=\, \bigl(e^{-xz}\, \psi(z,x)\bigr)\big|_{x=0} \,, \quad c^*(z) \,:=\, \bigl(e^{xz} \, \psi^*(z,x)\bigr)\big|_{x=0} \,, \quad H(z) \,:=\, c(z) \, c^*(z) \,. 
\eeq
It follow from Lemma~\ref{pppp} that the product $H(z)\in \CC[[z^{-1}]]$ is uniquely determined by~$L$, where we recall that
 $L$ is defined by~\eqref{deflaxqmL}. We have the following lemma. 
\begin{lemma} \label{Lemmaprod}
Define $z_k(x) = \res  L^{\frac{k}{r}}$ for all $k\geq 0$.
Then
\beq\label{formulaHz}
H(z) \= 1 \+ \sum_{k\geq 1} \, (-1)^k \, z_{k-1} (0)  \, z^{-k}  \,. 
\eeq
\end{lemma}
\begin{proof}
Following Liu--Vakil--Xu~\cite{LVX}, 
define 
$\Phi_-= \sum_{i\geq 0}  \phi_{-,i}(x) \, \p^{-i}$ with $\phi_{-,i}(x)\in \CC[[x]]$ and $\widetilde \phi_{-,0}(x)\equiv1$ 
as the particular dressing operator of~$L$ fixed by the additional conditions:
\[
 \phi_{-, i}(0)\=0\,, \quad i\geq 1\,.
\]
Denote by $(\psi_-, \psi_-^*)$ the corresponding pair of wave and dual wave functions, 
and define $c_-$, $c_-^*$ as $\psi_-$, $\psi_-^*$ evaluated at~$x=0$, respectively. It is clear from the definition that $c_-(z) \equiv 1 $. Therefore, $H(z) \equiv c_-^*(z)$. Now, on one hand, noticing that 
\begin{align}
\bigl(\Phi^{-1}_-\bigr)^*  & \= \sum_{i\geq 0} (-\p)^{-i}  \circ \phi_{-,i}(x)  
\=  \sum_{k\geq 0}   (-1)^k \sum_{i, \ell\geq 0 \atop i+\ell =k}  \binom{k-1}{\ell} \, \phi_{-,i}^{(\ell)}(x) \, \p^{-k}   \,, \nn
\end{align}
we find
\beq
\psi_-^*(x,z) \=  \bigl(\Phi^{-1}_-\bigr)^* \bigl(e^{-xz}\bigr)  \=  
\sum_{k\geq 0} (-1)^k \sum_{i, \ell\geq 0 \atop i +\ell =k}  \, \binom{k-1}{\ell} \, \phi_{-,i}^{(\ell)}(x) \, z^{-k}   \, e^{-xz} \, .  \nn
\eeq
Hence
\begin{align} \label{cmz}
c_-^*(z) & \=  
\sum_{k\geq 0} (-1)^k \sum_{i=0}^k  \, \binom{k-1}{i-1} \, \phi_{-,i}^{(k-i)}(0) \, z^{-k} \,. 
\end{align}
On the other hand, from the definition of the dressing operator~$\Phi_{-}$ we know that 
\[
z_k(x)  \= \res  \, \bigl(\Phi_- \circ \p^r \circ \Phi^{-1}_-\bigr)^{\frac{k}{r}}  
\= \res \, \Phi_- \circ \p^k \circ \Phi_-^{-1}\,.\]
Taking $x=0$ in this formula we find that 
\begin{align} 
& z_k(0) \= \bigl(\res \, \p^k \circ \Phi_-^{-1} \bigr) \, \big|_{x=0} 
\= \Biggl(\res \,  \sum_{i\geq 0} \sum_{\ell=0}^k \binom{k}{\ell} \, \phi_{-,i}^{(k-\ell)}(x) \, \p^{\ell-i} \Biggr) \, \Bigg|_{x=0}  
\= \sum_{\ell=0}^k \binom{k}{\ell} \phi_{-,\ell+1}^{k-\ell}(0) \,.\nn
\end{align}
The lemma is proved by comparing this expression with~\eqref{cmz} and using $c_-(z) \equiv 1 $.
\end{proof}

We note that, in the above proof, the function $c_-(z)$ is very simple, being just the constant function~$1$, but the formula~\eqref{cmz} for
$c_-^*(z)$ might be very complicated. The main point of the wave-function-pair 
approach is the following: using Lemma~\ref{pppp} it is sometimes possible to find a particular choice 
of pair of wave functions such that $c(z)$ and $c^*(z)$ both have relatively simple 
expressions so that their product can be given in closed form, and this is the case in particular both for this 
paper and for~\cite{DYZ0,DYZ}, where {\it bispectrality}~\cite{DYZ0,DYZ,DG} is used for fixing the 
particular choice. In other situations, the Sato tau-function, theta-functions, and etc. could also be 
used to construct a pair of wave functions; instead of giving the details 
we refer to~\cite{BBT, BDY0, BY, Buryak1, Dickey, Du-0} 
for specific constructions. 

\section{Frobenius manifolds}  \label{appfm}
In this appendix, we give a brief review of the theory of Frobenius manifolds~\cite{Du-3,Du-6} 
(cf.~also~\cite{Du-7,DZ-norm,Givental-2,Manin,Teleman}). 
Recall that 
a {\it Frobenius algebra} is a triple $(V,\mathbbm{1},\langle\,,\,\rangle)$, where $V$ is a commutative associative algebra over~$\CC$
with unity~$\mathbbm{1}$, and $\langle\,,\,\rangle: V\times V\rightarrow \mathbb{C}$ is a symmetric 
non-degenerate bilinear form satisfying 
$\langle x\cdot y ,  z\rangle \= \langle x, y\cdot z\rangle$, $\forall\,x,y,z \in V$. 
\begin{defi}[\cite{Du-1,Du-3}] \label{FM}
A {\it Frobenius structure} of charge~$d$ on a complex manifold~$M^l$ is a family of 
Frobenius algebras $\bigl(T_p M,{\bf 1}_p,\langle\,,\,\rangle_p\bigr)$, $p\in M$ 
depending holomorphically on~$p$ and satisfying:
\begin{itemize}
\item[${\bf FM1}$]  The metric $\langle\,,\,\rangle$ on $M$ is flat, and $\nabla {\bf 1} = 0 $, 
where $\nabla$ is the Levi--Civita connection of~$\langle\,,\,\rangle$ and ${\bf 1}$
is the unity vector field.
\item[${\bf FM2}$] Define a $3$-tensor field~$c$ by 
$c(X,Y,Z):=\langle X\cdot Y,Z\rangle$, for any three holomorphic vector fields $X,Y,Z$ on~$M$. 
The $4$-tensor field $\nabla c$ must be 
symmetric. 
\item[${\bf FM3}$] There exists a holomorphic vector field~$E$ on~$M$,
called the Euler vector field, satisfying
\begin{align}
& \nabla\nabla E \= 0 \,, \label{linear}\\
&  [E, X\cdot Y]-[E,X]\cdot Y \,-\, X\cdot [E,Y] \= X\cdot Y \,,\label{E2}\\
&  E \, \langle X,Y \rangle \,-\, \langle [E,X],Y\rangle \,-\, \langle X, [E,Y]\rangle 
\= (2-d) \, \langle X,Y\rangle \, . \label{E3}
\end{align}
\end{itemize} 
A complex manifold endowed with a Frobenius structure is called a {\it Frobenius manifold}, with 
$\langle \,, \rangle$ being called the {\it invariant flat metric}. 
\end{defi}

Let $M$ be a Frobenius manifold of complex dimension~$l$. 
Following~\cite{Du-3}, define a one-parameter family of affine connections on~$M$:
\beq
\widetilde \nabla_{X} Y \:= 
\nabla_X Y \+ z \, X \cdot Y \,,\quad z\in \mathbb{C} \,.
\eeq
This family of affine connections 
$\widetilde \nabla$ are all flat~\cite{Du-3}, and is called 
 the {\it deformed flat connection}. 
Moreover, it  
can be extended to a flat affine connection~\cite{Du-3} on~$M\times \mathbb{C}^*$, 
still denoted by~$\widetilde\nabla$, 
whose definition along the $z$-direction is given as follows:
\beq
\widetilde \nabla_{\p_z} X \:= \frac{\p X}{\p z} \+ E \cdot X \,-\, \frac1 z \,\mu \, X \,, \quad 
\widetilde \nabla_{\p_z} \p_z \:= 0 \,,  \quad  \widetilde \nabla_X \p_z \:= 0\,,
\eeq
for $X$ being an arbitrary holomorphic vector field on $M \times \mathbb{C}^*$ with 
zero component along the~$z$ direction. Here, $\mu:=\frac{2-d}2 -\nabla E$.
We call~$\widetilde\nabla$
 the {\it extended deformed flat connection}.
A holomorphic function~$f(v;z)$ on some open subset of $M\times \mathbb{C}^*$ is called {\it $\widetilde\nabla$-flat} if
\beq
\widetilde \nabla df \= 0 \,.
\eeq

To understand the flat coordinates for the extended deformed flat connection~$\widetilde\nabla$, 
let us take $v=\bigl(v^1,\dots,v^l\bigr)$ a system of flat coordinates with respect to $\langle \,,\rangle$. 
Denote $\p_\alpha=\frac{\p}{\p v^\alpha}$, 
$\eta_{\alpha\beta} := \langle \p_\alpha , \p_\beta \rangle$, $\eta=(\eta_{\alpha\beta})$, and 
$\bigl(\eta^{\alpha\beta}\bigr):=\eta^{-1}$. 
By FM1 we choose~$v^1$ satisfying~$\p_1={\bf 1}$. 
For simplicity we will assume that~$\nabla E$ is diagonalizable, so the flat coordinates are chosen such that 
$\mu={\rm diag} (\mu_1,\dots,\mu_l)$, and we have
\beq\label{evfagain}
E \= \sum_{1\leq\alpha\leq l} \Bigl(\bigl(1-\tfrac d 2 -\mu_\alpha\bigr) \, v^\alpha \, \p_\alpha  \+ r^\alpha \, \p_\alpha \Bigr)\,.
\eeq
The axioms FM1--FM3 imply the local existence of a holomorphic 
function~$F$, called the {\it Frobenius potential} of~$M$, satisfying
\begin{align}
& \frac{\p^3 F}{\p v^\alpha \p v^\beta \p v^\gamma} \= c_{\alpha\beta\gamma}\:=
c \bigl(\p_\alpha,\p_\beta,\p_\gamma\bigr) \,, \\
& EF \= (3-d) \, F \+ \mbox{a quadratic function of } v \, .
\end{align}
Clearly, $F$ is uniquely determined by the Frobenius structure up to a quadratic function of~$v$.
The flatness of~$\widetilde \nabla$ implies that there locally exist $l$ independent 
$\widetilde\nabla$-flat holomorphic functions 
\[\tilde v_1,\dots,\tilde v_l\] on $M\times \CC^*$, called the {\it deformed 
flat coordinates} for the Frobenius manifold. Let us give their construction. 
For a $\widetilde\nabla$-flat holomorphic function~$f$ on $M\times \CC^*$, denote $y^\alpha=\eta^{\alpha\beta}\p f/\p v^\beta$ and 
$y=(y^1,\dots,y^l)^T$. 
By definition the column-vector-valued function~$y$ satisfies the following system of linear differential equations:
\begin{align}
&\frac{\p y}{\p v^\alpha} \= z \, C_\alpha \, y \,,\label{yvalpha}\\
&\frac{d y}{d z} \= \Bigl(\mathcal{U} + \frac{\mu}{z}\Bigr) y\,. \label{yz}
\end{align}
Here, $(C_\alpha)^\beta_\gamma:=c^\beta_{\alpha\gamma}$ and $\mathcal{U}^\alpha_\beta:=E^\rho c^\alpha_{\rho\beta}$. 
To fix a system of deformed flat coordinates we need to choose a basis of the solution space to~\eqref{yvalpha}--\eqref{yz}.
Observe that the ODE system~\eqref{yz} possesses a Fuchsian singular point at $z=0$ and an irregular  
singular point of Poincar\'e rank~1 at $z=\infty$. 
The axioms of Frobenius manifolds imply the isomonodromicity of~\eqref{yvalpha}--\eqref{yz}. 
The monodromy data around $z=0$ is then given by two constant matrices $\mu,R$, satisfying 
\begin{align}
&R^*\= -e^{\pi i \mu} R\, e^{-\pi i \mu}\,\\
&z^\mu R z^{-\mu} \= R_0\+R_1z\+R_2z^2\+\dots
\end{align}
for some matrices $R_k$, $k\geq0$ satisfying 
\begin{align}
[\mu,R_k] \= k \, R_k\,, \quad k\geq 0\,.
\end{align}

Let us fix a choice of~$R$. It is shown in~\cite{Du-3,Du-6} (see~also~\cite{Givental-2}) that locally 
there exists a fundamental solution matrix~$\mathcal{Y}=\mathcal{Y}(v;z)$ around $z=0$ to~\eqref{yvalpha}--\eqref{yz} of the form
\beq\label{YThetazmuzr}
\mathcal{Y}(v;z) \= \Theta(v;z) \, z^\mu z^R \= \sum_{m\geq0} \Theta_m(v) \, z^{m+\mu} z^R \,,
\eeq
where $\Theta(v;z)$ is a matrix-valued analytic function on $M\times \CC$ satisfying 
\begin{align}
&\Theta(v;0) \;\equiv\; I \,, \label{normalTheta0}\\
& \eta^{-1} \Theta(v;-z)^T \eta \, \Theta(v;z) \;\equiv\; I \,. \label{normalThetatt}
\end{align}
It can be easily checked that the matrix-valued function $\Theta(v;z)$ satisfies 
\[\p^\beta\Theta^\alpha_\gamma(v;z) \= \p^\alpha\Theta^\beta_\gamma(v;z)\,.\] 
Then by the Poincar\'e lemma, 
there locally exist holomorphic functions $\theta_\gamma(v;z)$ of the form
\begin{align}\label{theta0c}
& \theta_\gamma(v;z) \= \sum_{m\ge 0}\theta_{\gamma,m}(v) \, z^m \,,
\end{align}
such that 
\beq\label{thetagradientTheta}
\eta^{\alpha\beta}\frac{\p \theta_\gamma(v;z)}{\p v^\beta} \= \Theta^\alpha_\gamma(v;z) \,.
\eeq
Hence the functions $\tilde{v}_1(v;z),\dots,\tilde{v}_n(v;z)$ defined by 
\begin{align}
\bigl(\tilde{v}_1(v;z),\dots,\tilde{v}_n(v;z) \bigr) \= \bigl(\theta_1(v;z),\dots,\theta_l(v;z)\bigr) \, z^{\mu} z^R 
\end{align}
give a system of deformed flat coordinates on $M\times \CC^*$ (see \cite{Du-3,Du-6,Du-7,DZ-norm} for more details). 

From the above construction we know that $\theta_{\alpha,m}$, $m\geq 0$, satisfy the following conditions:
\begin{align}
& \p_{\alpha} \p_{\beta} \bigl(\theta_{\gamma,m+1} \bigr) \= c^\sigma_{\alpha\beta} \, 
\p_{\sigma} \bigl(\theta_{\gamma,m}\bigr) \,, \quad m\geq 0\,, \label{theta1c}\\
& E \, \bigl(\p_{\beta}\theta_{\alpha,m} \bigr) \= 
(p+\mu_\alpha+\mu_\beta) \, \p_{\beta} \bigl(\theta_{\alpha,m}\bigr) \+ 
\sum_{1\leq k\leq m} (R_k)^\gamma_\alpha \, \p_{\beta} \bigl(\theta_{\gamma,m-k}\bigr) \,, \quad m\geq 0\,, \label{theta2c} \\
& \frac{\p \theta_{\alpha,0}}{\p v^\beta} \= \eta_{\alpha\beta} \,, \label{theta6b} \\
& \langle\nabla\theta_\alpha(v,z),\nabla\theta_\beta(v,-z)\rangle \= \eta_{\alpha\beta} \, . \label{theta4c}
\end{align}
One can further normalize $\theta_{\alpha,m}(v)$ by requiring 
\begin{align}
& \theta_{\alpha,0} \= v_\alpha\,, \label{theta6a} \\ 
& \p_1 \bigl(\theta_{\alpha, m+1}\bigr) \= \theta_{\alpha, m} \,,  \quad m\geq 0\,. \label{theta3c} 
\end{align}
Indeed, \eqref{theta6a} is obviously compatible with~\eqref{theta1c}--\eqref{theta4c}, and 
we leave the verification of the compatibility 
between~\eqref{theta3c} and~\eqref{theta1c}--\eqref{theta4c} as an exercise to the reader. 
As in~\cite{DLYZ}, we call a choice of $\{\theta_{\alpha,m}\}_{m\geq 0}$ 
satisfying~\eqref{theta1c}--\eqref{theta3c}
a {\it calibration} on~$M$.

According to Kontsevich--Manin~\cite{KontsevichManin,Manin}, 
the genus zero part of a CohFT \[\bigl(V^l, \langle,\rangle, \mathbbm{1}, \{\Omega_{g,n}\}_{2g-2+n>0}\bigr)\]
gives a pre-Frobenius structure on the convergence domain~$B$ 
under the assumption (see Section~\ref{section2point1}). 
Here ``pre" means the axiom FM3 is 
not required. (In some literature a pre-Frobenius structure is called a Frobenius structure.)
Let us recall Kontsevich--Manin's construction of the Frobenius structure from the CohFT. 
As in Section~\ref{section2point1}, take $e_1=\mathbbm{1},e_2,\dots,e_l$ a basis of~$V$. 
Define a metric $\langle \,, \, \rangle$ on~$B$ and a multiplication~``$\,\cdot\,$" on the tangent spaces of~$B$ by 
\begin{align}
& \langle \p_\alpha,\p_\beta\rangle \:= \eta_{\alpha\beta} \,, \label{cohofttofm1}\\
& \langle \p_\alpha \cdot \p_\beta, \p_\gamma\rangle \:= \frac{\p^3 F}{\p v^\alpha \p v^\beta \p v^\gamma}  \,, \label{multifromcohft}
\end{align}
where $F$ is defined in~\eqref{cohof}.
From~\eqref{cohofttofm1}, it is obvious that 
the metric $\langle\,,\,\rangle$ is flat and in fact 
$(v^1,\dots,v^n)$ is a system of flat coordinates for~$\langle\,,\,\rangle$. 
The invariance of~$\langle\,,\,\rangle$ 
with respect to the multiplication~``$\,\cdot\,$" on each tangent space 
 as well as the FM2  
 are also obviously true from the construction. 
The axioms of the CohFT imply~\cite{KontsevichManin} that
 the multiplication~``$\,\cdot\,$" is associative and $\p_1={\bf 1}$ (unity vector field). 
 So $(B, \langle\,,\,\rangle, \,\cdot, \p_1)$ is a pre-Frobenius manifold.
Using~\eqref{cohoft4defeqn} we see further that a {\it homogeneous} CohFT 
$\bigl(V, \langle,\rangle, \mathbbm{1}, \{\Omega_{g,n}\}_{2g-2+n>0}\bigr)$ 
of charge~$d$ endows~$B$ with a 
 Frobenius structure of charge~$d$ with~$E$ given by~\eqref{evf} 
being the Euler vector field.

\medskip
\medskip
\medskip



\noindent Di Yang

\noindent School of Mathematical Sciences, USTC, Jinzhai Road 96, Hefei 230026, P.R.~China

\noindent diyang@ustc.edu.cn

\medskip
\medskip

\noindent Don Zagier

\noindent Max-Planck-Institut f\"ur Mathematik, Vivatsgasse 7, Bonn 53111, Germany\\
and 
International Centre for Theoretical Physics, Strada Costiera 11, Trieste 34014, Italy

\noindent dbz@mpim-bonn.mpg.de

\end{document}